\documentclass[10pt,a4paper]{article} %[review], damit Zeilennummern; [final] anstatt [review], damit keine Zeilennummern

\usepackage{import}
\usepackage{framed}
\usepackage{color}
\usepackage{a4,amsmath,amssymb,amscd,latexsym} 
\usepackage{amsthm} %bei SIAM Article ausgeschaltet
\usepackage{bbm}
\usepackage{epsfig} 
\usepackage{amsfonts}
\usepackage{mathrsfs}
\usepackage{overpic}
\usepackage{subcaption}
\usepackage{paralist}
\usepackage{graphicx}
\usepackage{trfsigns}
\usepackage{url}
\usepackage{stmaryrd}
\usepackage{booktabs}
\usepackage{algorithm, algorithmic}
\usepackage[algo2e, boxed, lined]{algorithm2e} 
\usepackage{hyperref}
\usepackage{cleveref}
\usepackage{scalerel}
\usepackage[export]{adjustbox}
\usepackage{subcaption}
\usepackage{tikz}
\usetikzlibrary{calc,trees,positioning,arrows,chains,shapes.geometric,%
	decorations.pathreplacing,decorations.pathmorphing,shapes,%
	matrix,shapes.symbols,backgrounds,shadows}
\usepackage{tikz-cd}
\usepackage{pgfplots}

\newtheorem{theorem}{Theorem}

\newtheorem{remark}{Remark}

\newcommand\fat[1]{\ThisStyle{\ooalign{%
			\kern.46pt$\SavedStyle#1$\cr\kern.33pt$\SavedStyle#1$\cr%
			\kern.2pt$\SavedStyle#1$\cr$\SavedStyle#1$}}}

\newenvironment{@abssec}[1]{%
	\vspace{.06in}\footnotesize
	\parindent=0in
	\ignorespaces 
	}

\newenvironment{keywords}{\begin{@abssec}{\keywordsname}}{\end{@abssec}}
\newenvironment{AMS}{\begin{@abssec}{AMS subject classification:}}{\end{@abssec}}

\title{Simultaneous Shape and Mesh Quality Optimization using Pre-Shape Calculus}

\author{Daniel Luft\thanks{Trier University, Department of Mathematics, 54286 Trier, Germany (\tt luft@uni-trier.de)}
 \and Volker Schulz\thanks{Trier University, Department of Mathematics, 54286 Trier, Germany (\tt volker.schulz@uni-trier.de)}
		} % In ex-shared.tex File Titel plus Author eintragen

\newcommand{\R}{{\mathbb{R}}} 
\newcommand{\Manifold}{M}
\newcommand{\HoldAll}{\mathbb{D}}
\newcommand{\Diffeomorphism}{\rho}

\newcommand{\Shape}{\Gamma}
\newcommand{\ShapeEmbedding}{\varphi}
\newcommand{\VolumeEmbedding}{\Phi}

\newcommand{\Target}{\mathcal{J}}
\newcommand{\TargetPrShp}{\mathfrak{J}}
\newcommand{\TargetShp}{\mathcal{J}}
\newcommand{\PrShpDeriv}{\mathfrak{D}}
\newcommand{\ShpDeriv}{\mathcal{D}}

\newcommand{\TangentSpaceShape}{\mathcal{T}}
\newcommand{\TangentVector}{\tau}
\newcommand{\NormalSpaceShape}{\mathcal{N}}

\newcommand{\diff}{\mathrm{d}}

\newcommand{\ProjectionCanonical}{\pi}

\newcommand{\RieszEnergyExtForce}{q}

\begin{document}
\maketitle

\begin{abstract}
	\noindent
	Computational meshes arising from shape optimization routines commonly suffer from decrease of mesh quality or even destruction of the mesh. 
	In this work, we provide an approach to regularize general shape optimization problems to increase both shape and volume mesh quality.
	For this, we employ pre-shape calculus as established in \cite{luft2020pre}.
	Existence of regularized solutions is guaranteed.
	Further, consistency of modified pre-shape gradient systems is established.
	We present pre-shape gradient system modifications, which permit simultaneous shape optimization with mesh quality improvement.
	Optimal shapes to the original problem are left invariant under regularization.
	The computational burden of our approach is limited, since additional solution of possibly larger (non-)linear systems for regularized shape gradients is not necessary.
	We implement and compare pre-shape gradient regularization approaches for a hard to solve 2D problem.
	As our approach does not depend on the choice of metrics representing shape gradients, we employ and compare several different metrics.
\end{abstract}


% REQUIRED

\section{Introduction}
Solutions of PDE constrained optimization problems, in particular problems where the desired control variable is a geometric shape, are only meaningful, if the state variables of the constraint can be calculated with sufficient reliability.
A key component of reliable solutions is given by the quality of the computational mesh. 
It is well-known that poor quality of elements affect the stability, convergence, and accuracy of finite element and other solvers.
This comes from effects such as poorly conditioned stiffness matrices (cf. \cite{shewchuk2002good}).

There are a multitude of strategies for increasing mesh quality, particularly in shape optimization.
The authors of \cite{etling2018first} enhance mesh morphing routines for shape optimization by trying to correct for the inexactness of Hadamard's theorem due to discretization of the problem. 
Correcting degenerate steps requires a restriction of deformation directions based on normal fields of shapes, which leads to a linear system enlarged by additional constraints.
In \cite{haubner2020continuous}, an approach to shape optimization using the method of mappings to guarantee non-degenerate deformations of meshes is presented.
For this, the shape optimization problem is regarded as an optimization in function spaces.
A penalty term for the Jacobian determinants of the deformations is added, which leads to a non-smooth optimality system.
Deformations computed by solving this system have less tendency to degenerate volumes of cells.
These techniques are related to the techniques of our work, however they do not include a mechanism to capture a target mesh volume distribution as is presented here in subsequent sections.
In \cite{onyshkevych2020mesh} metrics for representing shape gradients are modified by adding a non-linear advection term. 
For this, the shape derivative is represented on the shape seen as a boundary mesh, which is then used as a Neumann condition to assemble a system incorporating the non-linear advection term to represent a shape gradient in volume formulation. 
This formulation advects nodes of the volume mesh in order to mitigate element degeneration, but requires solution of a non-linear system to construct the shape gradient.
In \cite{schmidt2014two}, the author applies mesh smoothing inspired by centroidal Voronoi reparameterization to construct a tangential field that corrects for degenerate surface mesh cells. For correcting the volume mesh cell degeneration, a shape gradient representation featuring a non-linear advection term is used.
This is motivated by the eikonal equation, while orthogonality of the correction to shape gradients is ensured by a Neumann condition.
In order to mitigate roughness of gradients and resulting degeneration of meshes, the authors of \cite{schulz2015Steklov} construct shape gradient representations by use of Steklov-Poincar\'e metrics.
As an example they propose the linear elasticity metric, giving a more regular shape gradient representation by solution of a linear system using volume formulations of shape derivatives.
In \cite{herzog2020manifold}, the authors construct a Riemannian metric for the manifold of planar triangular meshes with given fixed connectivity, which makes the space geodesically complete. 
They propose a mesh morphing routine by geodesic flows, using the Hamiltonian formulation of the geodesic equation and solving by a symplectic numerical integrator.
Numerical experiments in \cite{herzog2020manifold} suggest that cell aspect ratios are bounded away from zero and thus avoid mesh degradation during deformations.

Several of the aforementioned approaches require modification of the metrics acting as left hand sides to represent shape gradients.
Often, either non-linear terms are added or systems are enlarged, which increases computational burden significantly.
Further, the mesh quality resulting by these regularized deformations is arbitrary or of rather implicit nature, since no explicit criterion for mesh quality is enforced.

In this work, we want to approach these two aspects differently.
By using pre-shape calculus techniques introduced in \cite{luft2020pre}, we propose regularization techniques for shape gradients with two goals in mind.
First, the required computational burden should be limited.
We achieve this by modifying the right hand sides of shape gradient systems, instead of altering the metric.
This ensures that shape gradient systems neither increase in size nor become non-linear.
Secondly, our regularization explicitly targets mesh qualities defined by the user.
To do so, we enforce surface and volume cell distribution via use of pre-shape derivatives of \emph{pre-shape parameterization tracking problems}.
Non-uniform node distributions according to targets can be beneficial, especially in the context of PDE constrained optimization.

In \cite{friederich2014adaptive} local sensitivities for minimization of the approximation error of linear elliptic second order PDE are derived and used to refine computational meshes.
Also, \cite{cao1999study} studies various monitor functions (targets) for mesh deformation methods in 2D by using elliptic and eigenvalue methods, e.g. to ensure certain coordinate lines of the mesh.
Amongst other examples, this shows possible demand for targeted non-uniform adaptation of computational meshes.
Non-uniform cell distributions are possible with our approach
our pre-shape calculus techniques enable efficient routines solving shape optimization problems, which simultaneously optimize quality of the surface mesh representing the shape and the volume mesh representing the hold-all domain.
This is done in a manner that does not interfere with the original shape optimization problem, leaving optimal and intermediate shapes invariant.
An complementing literature review is found in the introduction of \cite{luft2020pre}.
For this work we build on results of \cite{luft2020pre}, where pre-shape calculus is rigorously established.
The results from \cite{luft2020pre} feature a structure theorem in the style of Hadamard's theorem for shape derivatives, which also shows the relationship of classical shape and pre-shape derivatives.
In particular, shape calculus and pre-shape differentiability results carry over to the pre-shape setting.
Further, the pre-shape parameterization tracking problem is introduced in \cite[Prop. 2, Thrm. 2]{luft2020pre}, where existence of solutions is proved and a closed formula for its pre-shape derivative is derived.

This work is structured in 2 sections.
First, in \cref{Section_TheoryPreShapeRegularization} we establish theoretical results for pre-shape regularization routines for shape and volume mesh quality.
In particular, \cref{SubSection_TangentialMeshQOpt} focuses on the regularization for shape mesh quality, while \cref{SubSection_VolumeMeshQOpt} we take care of volume mesh quality regularization for the hold-all domain.
For both cases, existence results of regularized pre-shape solutions are provided.
Also, regularized pre-shape gradient systems are formulated, and consistency of regularized and unregularized gradients is established.
In \cref{Section_ImplementationPreShapeRegu} we display our techniques for a model problem.
Several different (un-)regularized routines are tested.
As a standard approach, we recommend the linear elasticity metric as found in \cite{SchulzSiebenborn}.
To illustrate the general applicability of our regularization method, we also test the very reasonable $p$-Laplacian metrics as inspired by \cite{hinze2021pLap} to represent gradients.

\section{Theoretical Aspects of Regularization by Pre-Shape Parameterization Tracking}\label{Section_TheoryPreShapeRegularization}
The main goal of this paper is to introduce a regularization approach to shape optimization, which increases mesh quality during optimization at minimal additional computational cost.
To achieve this, we proceed in two steps.
First, in \cref{SubSection_TangentialMeshQOpt}, we analyze the case where quality of the mesh modeling a shape is optimized.
This amounts to increasing quality of the (hyper-)surface shape mesh embedded in the volume hold-all domain.
Then, in \cref{SubSection_VolumeMeshQOpt}, we build on the surface mesh case by also demanding to increase the volume mesh quality of the hold-all domain.
Both approaches need to satisfy two properties.
On the one hand, they must not interfere with the original shape optimization problem, i.e. leave the optimal shape or even intermediate shapes invariant.
On the other hand, to be practically feasible, the mesh quality regularization approach should not increase computational cost.
In particular, no additional solution of linear systems should be necessary if compared to standard shape gradient descent algorithms.
We will achieve both properties for the case of simultaneous shape and volume mesh quality optimization.

If not stated otherwise, we assume $\Manifold \subset \R^{n+1}$ to be an $n$-dimensional, orientable, path-connected and compact  $C^\infty$-manifold without boundary.
In light of \cite{luft2020pre}, we will assume  $\operatorname{Emb}(\Manifold, \R^{n+1})$ to be the space of pre-shapes and $B_e^n := \operatorname{Emb}(\Manifold, \R^{n+1})/\operatorname{Diff}(\Manifold)$ to be the space of shapes.

\subsection{Simultaneous Tangential Mesh Quality and Shape Optimization}\label{SubSection_TangentialMeshQOpt}
In this subsection we formulate a regularized shape optimization problem to track for desired shape mesh quality using pre-shapes calculus.
As mentioned in the introduction, we heavily build on the results of \cite{luft2020pre}.

\subsubsection{Preliminaries for Mesh Quality Optimization using Pre-Shapes}
Throughout \cref{Section_TheoryPreShapeRegularization}, we take a look at a general prototype shape optimization problem 
\begin{equation}\label{Eq_GeneralShapeProblem}
	\underset{\Shape\in B_e^n}{\min}\TargetShp(\Shape).
\end{equation}
Here we only assume that the shape functional $\TargetShp: B_e^n \rightarrow \R$ is first order shape differentiable.

Now we reformulate \cref{Eq_GeneralShapeProblem} in the context of pre-shape optimization by use of the canonical projection $\ProjectionCanonical: \operatorname{Emb}(\Manifold, \R^{n+1}) \rightarrow B_e^n$.
The canonical projection $\ProjectionCanonical$ maps each pre-shape $\ShapeEmbedding\in \operatorname{Emb}(\Manifold, \R^{n+1})$ to an equivalence class $\ProjectionCanonical(\ShapeEmbedding) = \Shape \in B_e^n$, which consists of all pre-shapes mapping onto the same image shape in $\R^{n+1}$.
We remind the reader that we can associate $\ProjectionCanonical(\ShapeEmbedding)$ with the set interpretation $\ShapeEmbedding(\Manifold)$ of the shape $\Shape$.
So $\ProjectionCanonical(\ShapeEmbedding)$ can be interpreted as the set of all parameterizations of a given shape $\Shape\in B_e^n$ in $\R^{n+1}$.
The pre-shape formulation of \cref{Eq_GeneralShapeProblem} takes the form
\begin{equation}\label{Eq_GeneralShapeProblemPrShpExt}
	\underset{\ShapeEmbedding\in \operatorname{Emb}(\Manifold, \R^{n+1})}{\min}(\TargetShp\circ\ProjectionCanonical)(\ShapeEmbedding).
\end{equation}
It is important to notice that the extended target functional of \cref{Eq_GeneralShapeProblemPrShpExt} is pre-shape differentiable in the sense of \cite[Definition 3]{luft2020pre}, since we assumed $\TargetShp$ to be shape differentiable (cf. \cite[Prop. 1]{luft2020pre}).

Next, we need the pre-shape parameterization tracking functional introduced in \cite[Prop. 2]{luft2020pre}.
In the discrete setting, this functional allows us to track for given desired sizes of surface and volume mesh cells.
For this, let us assume functions $g^\Manifold: \Manifold \rightarrow (0,\infty)$ and $f^\Manifold_\ShapeEmbedding: \ShapeEmbedding(\Manifold) \rightarrow (0,\infty)$ to be smooth and fulfill the normalization condition
\begin{equation}\label{Assumption_fgNormalization}
	\int_{\ShapeEmbedding(\Manifold)} f^\Manifold_\ShapeEmbedding(s) \; \diff s
	= 
	\int_\Manifold g^\Manifold(s)\; \diff s  
	\quad \forall \ShapeEmbedding\in \operatorname{Emb}(\Manifold, \R^{n+1}).
\end{equation}
Further, we assume $f$ to have shape functionality (cf. \cite[Definition 2]{luft2020pre}), i.e.
\begin{equation}\label{Assumption_ShapeFunctionality}
	f^\Manifold_{\ShapeEmbedding\circ\Diffeomorphism} = f^\Manifold_{\ShapeEmbedding} \quad \forall \Diffeomorphism\in \operatorname{Diff}(\Manifold).
\end{equation}
Shape functionality for $f^\Manifold$ simply means, that $f^\Manifold_{\varphi_1} = f^\Manifold_{\varphi_2}$ for all pre-shapes $\varphi_1, \varphi_2 \in \ProjectionCanonical(\ShapeEmbedding)$ corresponding to the same shape $\Shape = \ProjectionCanonical(\ShapeEmbedding)$.

Finally, the pre-shape parameterization tracking problem takes the form
\begin{equation}\label{Eq_FinalGeneralParamTrackingProblem}
	\underset{\ShapeEmbedding \in \operatorname{Emb}(\Manifold, \R^{n+1})}{\min} \frac{1}{2}\int_{\ShapeEmbedding(\Manifold)}
	\Big(
	g^\Manifold \circ\ShapeEmbedding^{-1}(s) \cdot \det D^\tau \ShapeEmbedding^{-1}(s)
	-
	f^\Manifold_{\ShapeEmbedding}(s) 
	\Big)^2 \; \diff s =: \TargetPrShp^\tau(\ShapeEmbedding)
\end{equation}
where $D^\tau$ is the covariant derivative.
For each shape $\Shape\in B_e^n$ there exists a pre-shape $\ShapeEmbedding \in \ProjectionCanonical(\ShapeEmbedding)=\Shape$ minimizing  \cref{Eq_FinalGeneralParamTrackingProblem}, which is guaranteed by \cite[Prop. 2]{luft2020pre}.

As we need the pre-shape derivative of the parameterization tracking functional $\TargetPrShp^\tau$ in order to devise our algorithms, we state it in the style of the pre-shape derivative structure theorem \cite[Thrm. 1]{luft2020pre}
\begin{equation}\label{Eq_PreShapeDerSplitting}
	\PrShpDeriv\TargetPrShp^\tau(\ShapeEmbedding)[V] = \langle g_\ShapeEmbedding^\NormalSpaceShape, V \rangle + \langle g_\ShapeEmbedding^\TangentSpaceShape, V \rangle \quad \forall V\in C^\infty(\R^{n+1}, \R^{n+1}),
\end{equation}
with normal/ shape component
\begin{align}\label{Eq_PreShapeDerTrackingNormal}
\begin{split}
	\langle g_\ShapeEmbedding^\NormalSpaceShape, V \rangle 
	= 
	&-\int_{\ShapeEmbedding(\Manifold)} \frac{\operatorname{dim}(\Manifold)}{2}\cdot\Big(\big(g^\Manifold\circ\ShapeEmbedding^{-1}\cdot \det D^\tau \ShapeEmbedding^{-1}\big)^2 - f_{\ShapeEmbedding}^2\Big)
	\cdot\kappa\cdot \langle V, n\rangle 
	\\
	&\qquad\qquad\quad+
	\Big(g^\Manifold\circ\ShapeEmbedding^{-1}\cdot \det D^\tau \ShapeEmbedding^{-1} - f_\ShapeEmbedding\Big)
	\cdot
	\Big(
	\frac{\partial f_\ShapeEmbedding}{\partial n} \cdot \langle V, n \rangle + \ShpDeriv(f_\ShapeEmbedding)[V]
	\Big) \;\diff s 
\end{split}
\end{align}
and tangential/ pre-shape component
\begin{align}\label{Eq_PreShapeDerTrackingTangential}
\begin{split}
	\langle g_\ShapeEmbedding^\TangentSpaceShape, V \rangle 
	= 
	&-\int_{\ShapeEmbedding(\Manifold)} \frac{1}{2}\cdot\Big(\big(g^\Manifold\circ\ShapeEmbedding^{-1}\cdot \det D^\tau \ShapeEmbedding^{-1}\big)^2 - f_{\ShapeEmbedding}^2\Big)
	\cdot\operatorname{div}_\Shape (V - \langle V, n \rangle \cdot n) 
	\\
	&\qquad\qquad\quad+
	\Big(g^\Manifold\circ\ShapeEmbedding^{-1}\cdot \det D^\tau \ShapeEmbedding^{-1} - f_\ShapeEmbedding\Big)
	\cdot
	\nabla_\Shape f_\ShapeEmbedding^T V \;\diff s.
\end{split}
\end{align}
Notice that we signify a duality pairing by $\langle \cdot, \cdot \rangle$, and with $\ShpDeriv(f_\ShapeEmbedding)[V]$ the classical shape derivative of $f_\ShapeEmbedding$ in direction $V$.
For the derivation of the pre-shape derivative $\PrShpDeriv\TargetPrShp^\tau$ by use of pre-shape calculus, we refer the reader to \cite[Thrm. 2, Cor. 2]{luft2020pre}.

\subsubsection{Regularization of Shape Optimization Problems by Shape Parameterization Tracking}
All ingredients necessary are now available to formulate a regularized version of a shape optimization problem.
For $\alpha^\tau>0$, we add the parameterization tracking functional in style of a regularizing term to pre-shape reformulation \cref{Eq_GeneralShapeProblemPrShpExt}
\begin{equation}\label{Eq_GeneralPreShapeProblemTangential}
	\underset{\ShapeEmbedding\in \operatorname{Emb}(\Manifold, \R^{n+1})}{\min}(\TargetShp\circ\ProjectionCanonical)(\ShapeEmbedding) + \alpha^\tau\cdot \TargetPrShp^\tau(\ShapeEmbedding).
\end{equation}
We already found that the pre-shape extension \cref{Eq_GeneralShapeProblemPrShpExt} of our initial shape optimization problem is pre-shape differentiable.
Hence \cref{Eq_GeneralPreShapeProblemTangential} is pre-shape differentiable as well.
This is foundational, since a pre-shape gradient system for the regularized problem \cref{Eq_GeneralPreShapeProblemTangential} can always be formulated.

The pre-shape derivative $\PrShpDeriv\TargetPrShp(\ShapeEmbedding)[V]$ of the parameterization tracking problem (cf. \cref{Eq_PreShapeDerSplitting}) is well-defined for weakly differentiable directions $V\in H^1(\R^{n+1}, \R^{n+1})$.
By assuming the same for the shape derivative $\ShpDeriv\TargetShp$ of the original problem \cref{Eq_GeneralShapeProblem}, we can create a pre-shape gradient system for \cref{Eq_GeneralPreShapeProblemTangential} using a weak formulation with $H^1$-functions.
Given a symmetric, positive-definite bilinear form $a(.,.)$ (e.g. \cite{SchulzSiebenborn}, \cite{hinze2021pLap}), such a system takes the form
\begin{equation}\label{Eq_GradientSystemTangFull}
	a(U^{\TargetShp+\TargetPrShp^\tau}, V) = 	
	\ShpDeriv\TargetShp(\Shape)[V]
	\;+\;
	\alpha^\tau\cdot\langle g_\ShapeEmbedding^\NormalSpaceShape, V \rangle 	
	\;+\;
	\alpha^\tau\cdot\langle g_\ShapeEmbedding^\TangentSpaceShape, V \rangle \quad \forall V\in H^1(\R^{n+1}, \R^{n+1}).
\end{equation}
With $\Shape = \ProjectionCanonical(\ShapeEmbedding)$, the right hand side of \cref{Eq_GradientSystemTangFull} is indeed the full pre-shape gradient of \cref{Eq_GeneralPreShapeProblemTangential}.
This stems from the fact that the pre-shape extension $\TargetShp\circ\ProjectionCanonical$ has shape functionality by construction, which makes the pre-shape derivative $\PrShpDeriv(\TargetShp\circ\ProjectionCanonical)$ equal to the shape derivative $\ShpDeriv\TargetShp$ by application of \cite[Thrm. 1 $(iii)$]{luft2020pre}.

The full pre-shape gradient system \cref{Eq_GradientSystemTangFull} already achieves simultaneous solution of the shape optimization problem and regularization of shape mesh quality.
Namely, it is not required to additionally solve linear systems to create a mesh quality regularized descent direction $U^{\TargetShp+\TargetPrShp^\tau}$, since the original shape gradient system to problem \cref{Eq_GeneralShapeProblem} is modified by adding the two force terms $g^\NormalSpaceShape$ and $g^\TangentSpaceShape$.
A calculation of a descent direction $U^{\TargetShp+\TargetPrShp^\tau} = U^\TargetShp + U^{\TargetPrShp^\tau}$ by combining the shape gradient $U^\TargetShp$ and pre-shape gradient $U^{\TargetPrShp^\tau}$  solving the decoupled systems
\begin{equation}\label{Eq_GradientSystemOriginal}
	a(U^\TargetShp, V) = \ShpDeriv\TargetShp(\Shape)[V] \quad \forall V\in H^1(\R^{n+1}, \R^{n+1})
\end{equation}
and
\begin{equation}\label{Eq_GradientSystemParamTracking}
	a(U^{\TargetPrShp^\tau}, V) = 	
	\alpha^\tau\cdot\langle g_\ShapeEmbedding^\NormalSpaceShape, V \rangle 	
	+
	\alpha^\tau\cdot\langle g_\ShapeEmbedding^\TangentSpaceShape, V \rangle \quad \forall V\in H^1(\R^{n+1}, \R^{n+1})
\end{equation}
is in fact equivalent to solution of \cref{Eq_GradientSystemTangFull}.

However, by using \cref{Eq_GradientSystemTangFull} for gradient based optimization, the first demanded property of leaving the optimal or intermediate shapes invariant is not true in general.
This issue comes from involvement of the shape component $g^\NormalSpaceShape_\ShapeEmbedding$ (cf. \cref{Eq_PreShapeDerTrackingNormal}) of the pre-shape derivative to the parameterization tracking functional $\TargetPrShp^\tau$.
It is acting solely in normal directions, therefore altering shapes by interfering with the shape derivative $\ShpDeriv\TargetShp$ of the original problem.

For this reason, we deviate from the full gradient system \cref{Eq_GradientSystemTangFull} consistent with the regularized problem \cref{Eq_GeneralPreShapeProblemTangential} by using a modified system
\begin{equation}\label{Eq_GradientSystemTangTang}
	a(\tilde{U}, V) = 	
	\ShpDeriv\TargetShp(\Shape)[V]	
	+
	\alpha^\tau\cdot\langle g_\ShapeEmbedding^\TangentSpaceShape, V \rangle \quad \forall V\in H^1(\R^{n+1}, \R^{n+1}).
\end{equation}
We project the pre-shape derivative $\PrShpDeriv\TargetPrShp^\tau$ onto its tangential part, which is realized by simply removing the shape component $g^\NormalSpaceShape$ from the right hand side of the gradient system.
By this we still have numerical feasibility, while also achieving invariance of optimal and intermediate shapes. 
Of course invariance is only true up to discretization errors.
From a traditional shape optimization perspective, this stems from considering \cref{Eq_GradientSystemTangTang} as a shape gradient system with additional force term acting on directions $V$ in the kernel of the classical shape derivative $\ShpDeriv\TargetShp$.
In particular, by Hadamard's theorem or the structure theorem for pre-shape derivatives \cite[Thrm. 1]{luft2020pre}, directions tangential to shapes $\Shape$ are in the kernel of $\ShpDeriv\TargetShp$.
Using the pre-shape setting, we interpret these directions as vector fields corresponding to the fiber components of $\operatorname{Emb}(\Manifold, \R^{n+1})$.

To sum up and justify the use of the pre-shape regularized problem \cref{Eq_GeneralPreShapeProblemTangential} and its modified gradient system \cref{Eq_GradientSystemTangTang}, we provide existence of solutions and consistency of the modified gradients with the regularized problem.
\begin{theorem}[Shape Regularized Problems]\label{Thrm_ReguTang}
	Assume $\Manifold \subset \R^{n+1}$ to be an $n$-dimensional, orientable, path-connected and compact  $C^\infty$-manifold without boundary.
	Let shape optimization problem \cref{Eq_GeneralShapeProblem} be shape differentiable and have a minimizer $\Shape\in B_e^n$.
	For shape parameterization tracking, let us assume functions $g^\Manifold: \Manifold \rightarrow (0,\infty)$ and $f^\Manifold_\ShapeEmbedding: \ShapeEmbedding(\Manifold) \rightarrow (0,\infty)$ to be smooth, fulfill the normalization condition \cref{Assumption_fgNormalization}, and $f$ to have shape functionality \cref{Assumption_ShapeFunctionality}.
	
	Then there exists a $\ShapeEmbedding\in \ProjectionCanonical(\ShapeEmbedding) = \Shape \subset \operatorname{Emb}(\Manifold, \R^{n+1})$ minimizing the shape regularized problem \cref{Eq_GeneralPreShapeProblemTangential}.
	
	The modified pre-shape gradient system \cref{Eq_GradientSystemTangTang} is consistent with the full pre-shape gradient system \cref{Eq_GradientSystemTangFull} and the shape gradient system of the original problem \cref{Eq_GradientSystemOriginal}, in the sense that
	\begin{equation}\label{Eq_GradientSystemConsistency}
		\tilde{U} = 0 \iff U^{\TargetShp+\TargetPrShp^\tau} = 0 \text{ and } U^\TargetShp = 0.
	\end{equation}
	In particular, if $U^{\TargetShp+\TargetPrShp^\tau}=0$ is satisfied, the necessary first order conditions for \cref{Eq_GeneralPreShapeProblemTangential} and \cref{Eq_GeneralShapeProblem} are satisfied as well.
\end{theorem}
\begin{proof}
	For the existence of solutions to the regularized problem \cref{Eq_GeneralPreShapeProblemTangential}, let us assume there exists a minimizer $\Shape\in B_e^n$ to the original problem \cref{Eq_GeneralShapeProblem}.
	By construction of the shape space $B_e^n$ by equivalence relation there exists a $\tilde{\ShapeEmbedding}\in \operatorname{Emb}(\Manifold, \R^{n+1})$, such that $\Shape = \ProjectionCanonical(\tilde{\ShapeEmbedding})$.
	So the set of pre-shapes $\ProjectionCanonical(\tilde{\ShapeEmbedding})$ acts as a set of candidates for solutions to \cref{Eq_GeneralPreShapeProblemTangential}.
	Since we required shape functionality of $f$ and normalization condition \cref{Assumption_fgNormalization}, the existence and characterization theorem for solutions to \cref{Eq_FinalGeneralParamTrackingProblem} (cf. \cite[Prop. 2]{luft2020pre}) gives existence of a global minimizer for $\TargetPrShp^\tau$ in every fiber of $\operatorname{Emb}(\Manifold, \R^{n+1})$.
	In particular, we can find such a $\ShapeEmbedding$ in $\ProjectionCanonical(\tilde{\ShapeEmbedding})$.
	From the last assertion of \cite[Prop. 2]{luft2020pre} we also get that $\TargetPrShp^\tau(\ShapeEmbedding) = 0$.
	However, since $\TargetPrShp^\tau \geq 0$ due to the quadratic nature of the parameterization tracking functional, we get that $\ShapeEmbedding$ is a solution to the regularized problem \cref{Eq_GeneralPreShapeProblemTangential}.
	
	Next, we prove the non-trivial direction '$\implies$' for  consistency of gradient systems in the sense of \cref{Eq_GradientSystemConsistency}.
	Let us assume we have a pre-shape gradient $\tilde{U}=0$ stemming from the modified system \cref{Eq_GradientSystemTangTang}.
	This immediately gives us 
	\begin{equation}
		\ShpDeriv\TargetShp(\Shape)[V]	
		+
		\alpha^\tau\cdot\langle g_\ShapeEmbedding^\TangentSpaceShape, V \rangle = 0 \quad \forall V\in H^1(\R^{n+1}, \R^{n+1}).
	\end{equation}
	However, due to the structure theorem for pre-shape derivatives \cite[Thrm. 1]{luft2020pre}, we know that the support of shape derivative $\ShpDeriv\TargetShp(\Shape)$ and pure pre-shape component $g_\ShapeEmbedding^\TangentSpaceShape$ are orthogonal.
	Hence we have 
	\begin{equation}\label{Temp1}
		\ShpDeriv\TargetShp(\Shape)[V] = 0 \text{ and } \alpha^\tau\cdot\langle g_\ShapeEmbedding^\TangentSpaceShape, V \rangle = 0 \quad \forall V\in H^1(\R^{n+1}, \R^{n+1}).
	\end{equation}
	This is the first order condition for \cref{Eq_GeneralShapeProblem}, also giving us $U^\TargetShp = 0$.
	
	It remains to show that the first order condition for the regularized problem \cref{Eq_GeneralPreShapeProblemTangential} is satisfied and the complete gradient $U^{\TargetShp+\TargetPrShp^\tau} = 0$.
	Essentially, this is a special case of result \cite[Prop. 3]{luft2020pre}, which characterizes global solutions of \cref{Eq_FinalGeneralParamTrackingProblem} by fiber stationarity.
	From \cref{Temp1} we see that $\ShapeEmbedding\in \operatorname{Emb}(\Manifold, \R^{n+1})$ is a fiber stationary point, since the full pre-shape derivative $\PrShpDeriv\TargetPrShp^\tau(\ShapeEmbedding)$ vanishes for all directions $V$ tangential to $\ShapeEmbedding(\Manifold)$.
	Hence \cite[Prop. 3]{luft2020pre} states that $\ShapeEmbedding$ is already a global minimizer of $\TargetPrShp^\tau$ and satisfies the corresponding necessary first order condition.
	Even more, \cite[Prop. 3]{luft2020pre} gives vanishing of  $g_\ShapeEmbedding^\NormalSpaceShape$.
	Therefore the right hand side of the full gradient system \cref{Eq_GradientSystemTangFull} is zero, giving us a vanishing full gradient $U^{\TargetShp+\TargetPrShp^\tau}$.
	
	Implication '$\Leftarrow$' can be seen by the fiber stationarity characterization as well.
	Since vanishing of the shape gradient $U^\TargetShp$ gives $\ShpDeriv\TargetShp(\Shape) = 0$, the full pre-shape derivative $\PrShpDeriv\TargetPrShp^\tau(\ShapeEmbedding)$ must be zero if the full pre-shape gradient $U^{\TargetShp^\tau} = 0$.
	Hence the fiber stationarity characterization \cite[Prop. 3]{luft2020pre} tells us that both $g_\ShapeEmbedding^\TangentSpaceShape$ and $g_\ShapeEmbedding^\NormalSpaceShape$ vanish for $\ShapeEmbedding$, which proves the other direction of \cref{Eq_GradientSystemConsistency}.	
\end{proof}

With \cref{Thrm_ReguTang} we can rest assured that optimization algorithms using the tangentially regularized gradient $U$ from  \cref{Eq_GradientSystemTangTang} leave stationary points of the original problem invariant.
Also, vanishing of the modified gradient $\tilde{U}$ indicates that we have a stationary shape $\Shape = \ProjectionCanonical(\ShapeEmbedding)$ with parameterization $\ShapeEmbedding$ having desired cell volume allocation $f_\ShapeEmbedding$.
Of course, this is only true up to discretization error.
We also see that the modified gradient system \cref{Eq_GradientSystemTangTang} captures the same information as the shape gradient system \cref{Eq_GradientSystemOriginal} and full pre-shape gradient system \cref{Eq_GradientSystemTangFull} combined.
This might seem counterintuitive at first glance, especially since necessary information is contained in one instead of two gradient systems.
However, by application of pre-shape calculus to derive the fiber stationarity characterization of solutions \cite[Prop. 3]{luft2020pre}, we recognize this circumstance as a consequence of the special structure of regularizing functional $\TargetPrShp^\tau$.
The fact that pre-shape spaces are fiber spaces with locally orthogonal tangential bundles for parameterizations and shapes (cf. \cite[Section 2]{luft2020pre}) is a fundamental prerequisite to this.
We will discuss numerical results comparing optimization using standard shape gradients \cref{Eq_GradientSystemOriginal} and gradients regularized by modified shape parameterization tracking \cref{Eq_GradientSystemTangTang} in \cref{Section_TheoryPreShapeRegularization}.

\begin{remark}[Applicability of Shape Regularization for General Shape Optimization Problems]\label{Remark_ValidityTangentialGenerality}
	It is important to notice, that we did not need a specific structure of the original shape optimization problem \cref{Eq_GeneralShapeProblem}.
	The only assumptions made are shape differentiability and existence of an optimal shape for $\TargetShp$.
	This means the tangential mesh regularization via $\TargetPrShp^\tau$ (cf. \cref{Eq_GeneralPreShapeProblemTangential}) is applicable for almost every meaningful shape optimization problem.
	In particular, PDE-constrained shape optimization problems can be regularized by this approach, which we will discuss in \cref{Section_ImplementationPreShapeRegu}.
	Also, the modified gradient system structure  \cref{Eq_GradientSystemTangTang} stays the same for different shape optimization targets $\TargetShp$. 
	It is solely the shape derivative $\ShpDeriv\TargetShp$ on the right hand side which changes depending on the shape problem target.
	Finally, existence of solutions and consistency of stationarity and first order conditions \cref{Thrm_ReguTang} is given 
\end{remark}

\begin{remark}[Equivalent Bi-Level Formulation of the Regularized Problem]\label{Remark_BiLevelTangential}
	It is also possible to formulate the shape mesh regularization as a nonlinear bilevel problem
	\begin{align}\label{Eq_BilevelTang}
	\begin{split}
		\underset{\ShapeEmbedding\in\operatorname{Emb}(\Manifold, \R^{n+1})}{\min}&\;\TargetPrShp^\tau(\ShapeEmbedding) \\
		\text{ s.t. } 		
		\ProjectionCanonical(\ShapeEmbedding) =&\; \underset{\Shape\in B_e^n}{\operatorname{arg} \min}\;\TargetShp(\Shape).
	\end{split}
	\end{align}
	The upper level pre-shape optimization problem depends only on the lower level shape optimization problem by the latter restricting the set of feasible points $\ShapeEmbedding$.
	So intuitively, solving \cref{Eq_BilevelTang} amounts to solving the lower level problem for a shape $\Shape$, and then to look for an optimal parameterization $\ShapeEmbedding\in \Shape \in B_e^n$ in the fiber corresponding to the optimal shape.
	If a solution to the lower level shape optimization problem exists, a solution $\ShapeEmbedding$ to the upper level problem exists as well (cf. \cite[Prop. 2]{luft2020pre}).
	The lower level problem enforces that the optimal shape $\Shape$ coincides with the fiber $\ProjectionCanonical(\ShapeEmbedding)$.
	
	The bilevel formulation motivates the modified gradient system \cref{Eq_GradientSystemTangTang} in a consistent manner. 
	For this, we can take the perspective of nonlinear bilevel programming as in \cite{savard1994steepest}.
	In finite dimensions, the authors of \cite{savard1994steepest} propose a way to calculate a descent direction by solving a bilevel optimization problem derived by the original problem. 
	We remind the reader that we formulated our systems as gradient systems and not descent directions, hence a change of sign compared to systems for descent directions in \cite{savard1994steepest} has to occur.
	Also notice that the additional constraint $\ProjectionCanonical(\ShapeEmbedding) = \Shape$ for the feasible set of solutions has to be added to formulate our bilevel problem in the exact style of \cite{savard1994steepest}.
	We can proceed with a symbolic calculation following \cite[Ch. 2]{savard1994steepest}, using relations \cite[Thrm. 1 (ii), (iii)]{luft2020pre} for pre-shape derivatives and the fact that $\TargetPrShp^\tau$ does not explicitly depend on $\Shape$ of the subproblem and $\Target$ not explicitly on $\ShapeEmbedding$ of the upper level problem.
	This yields a bilevel problem for the gradient $U$ to \cref{Eq_BilevelTang}
	\begin{align}\label{Eq_BilevelTangGradient}
	\begin{split}
		&\underset{ U \in H^1(\R^{n+1}, \R^{n+1}), \Vert U \Vert \leq 1}{\max} 
		\PrShpDeriv\TargetPrShp^\tau(\ShapeEmbedding)[U] \\
		\text{ s.t. }& U^\NormalSpaceShape = \underset{W \in H^1(\R^{n+1}, \R^{n+1}), \Vert W \Vert \leq 1}{\operatorname{ arg}\max} \ShpDeriv\TargetShp(\Shape)[W],
	\end{split}
	\end{align}
	where $U^\NormalSpaceShape$ is the component of $U$ normal to $\Shape$.
	In \cite[Ch. 3]{savard1994steepest} a descent method is employed for \cref{Eq_BilevelTangGradient} by alternating computation of $W_k$ and $U_k$.
	
	For our situation, a gradient $W_k$ for the lower level problem of \cref{Eq_BilevelTangGradient} corresponds to $U^\TargetShp$ solving the shape gradient system \cref{Eq_GradientSystemOriginal}.
	With this, the lower level constraint fixes the normal component of $U$ to be the shape derivative of the original problem \cref{Eq_GeneralShapeProblem}.
	By decomposing $U = U^\TangentSpaceShape + U^\NormalSpaceShape$ into tangential and normal directions, we see that the fixed normal component makes $\PrShpDeriv\TargetPrShp^\tau(\ShapeEmbedding)[U^\NormalSpaceShape] = 0$ a constant not relevant for the upper level problem.
	This lets us rewrite the system as
	\begin{align}\label{Eq_BilevelTangGradientFinal}
	\begin{split}
	\underset{ U^\TangentSpaceShape \in H^1(\R^{n+1}, \R^{n+1}), \Vert U^\TangentSpaceShape \Vert \leq 1}{\min}& 
	\langle g_\ShapeEmbedding^\TangentSpaceShape, U^\TangentSpaceShape \rangle \\
	\text{ s.t. } U^\NormalSpaceShape =\; U^\TargetShp& \\
	U =\; U^\TangentSpaceShape& + U^\NormalSpaceShape.
	\end{split}
	\end{align}
	We clearly see that the minimization of the tangential component $g^\TangentSpaceShape_\ShapeEmbedding$ in  \cref{Eq_BilevelTangGradientFinal} does not depend on the constraints given below.
	Hence it can be decoupled, so by considering $\alpha^\tau >0$ this leads to a gradient system
	\begin{equation}\label{Eq_GradientSystemTangModifiedSolo}
	a(U^\TangentSpaceShape, V) = 	
	\alpha^\tau\cdot\langle g_\ShapeEmbedding^\TangentSpaceShape, V \rangle \quad \forall V\in H^1(\R^{n+1}, \R^{n+1}).
	\end{equation}
	With the same orthogonality arguments made at \cref{Eq_GradientSystemOriginal} and \cref{Eq_GradientSystemParamTracking}, a separate computation of $U^\TargetShp$ and $U^\tau$ as in the general case in \cite{savard1994steepest} is not necessary in our case.
	The gradient $U = U^\TangentSpaceShape + U^\NormalSpaceShape$ for the bilevel problem \cref{Eq_BilevelTang} can be calculated using a single system, which coincides exactly with the modified gradient $\tilde{U}$ from system \cref{Eq_GradientSystemTangTang}.
	With \cref{Thrm_ReguTang}, this means using the modified pre-shape gradient $\tilde{U}$ as a descent direction in fact solves the bi-level problem \cref{Eq_BilevelTang}, the regularized problem \cref{Eq_GeneralPreShapeProblemTangential} and the original shape problem \cref{Eq_GeneralShapeProblem} at the same time.
\end{remark}

\subsection{Simultaneous Volume Mesh Quality and Shape Optimization}\label{SubSection_VolumeMeshQOpt}
	In this section we introduce the necessary machinery for a regularization strategy of volume meshes representing the hold-all domain $\HoldAll$.
	We will also incorporate our previous results to simultaneously optimize for shape mesh and volume mesh quality. 
	As before, we show a way to calculate modified pre-shape gradients without need for solving additional (non-)linear systems compared to standard shape gradient calculations.
	
	\subsubsection{Suitable Pre-Shape Spaces for Hold-All Mesh Quality Optimization}
	Using pre-shape techniques to increase volume mesh quality is a different task than regularizing the shape mesh as we have seen in \cref{SubSection_TangentialMeshQOpt}.
	Using a pre-shape space $\operatorname{Emb}(\Manifold, \R^{n+1})$ for this endeavor does not make sense, since this space contains all shapes in $\R^{n+1}$ combined with their parameterization.
	For $\operatorname{Emb}(\Manifold, \R^{n+1})$, parameterizations are described via the $n$-dimensional modeling or parameter manifold $\Manifold$.
	We now need a second, different parameter space modeling the hold-all domain.
	
	For this, let $\HoldAll\subset \R^{n+1}$ be a compact, orientable, path-connected, $n+1$-dimensional $C^\infty$-manifold.
	This hold-all domain will replace the unbounded hold-all domain $\R^{n+1}$ of our previous discussion.
	In most numerical cases, the hold-all domain has a non-empty boundary $\partial\HoldAll$.
	Hence we also permit non-empty $\partial\HoldAll$, but demand $C^\infty$-regularity of the boundary.
	We remind the reader that $\HoldAll$ is a compact manifold with boundary, so we have to distinguish from its interior, which we denote by $\operatorname{int}(\HoldAll)$.
	
	\begin{remark}[H\"older Regularity Setting]
		The setting of $C^\infty$-smoothness is taken because it is necessary to have a meaningful definition of shape space $B_e^n$.
		However, results concerning existence of optimal parameterizations $\ShapeEmbedding$ for \cref{Eq_GeneralPreShapeProblemTangential}, its pre-shape derivative formula \cref{Eq_PreShapeDerSplitting},  regularization strategies featuring the modified gradient system \cref{Eq_GradientSystemTangTang} and its consistency \cref{Thrm_ReguTang} all remain valid for the H\"older-regularity case $C^{k,\alpha}$ for $k\geq0$ and $1 > \alpha > 0$ (cf. \cite[Remark 3]{luft2020pre}).
		However, if $\HoldAll$ has $C^{k,\alpha}$-regularity and non-empty boundary, it is necessary to additionally assume $C^{k+3,\alpha}$-regularity for $\partial\HoldAll$.
		If this is given, all previous and following results remain valid.		
	\end{remark}
	A first choice for a pre-shape space to the hold-all domain would be $\operatorname{Emb}_{\partial\HoldAll}(\HoldAll, \HoldAll)$, which is the set of all $C^\infty$-embeddings leaving $\partial\HoldAll$ pointwise invariant, i.e.
	\begin{equation}\label{Eq_PreShapeSpaceNonTrivBoundary}
	\operatorname{Emb}_{\partial\HoldAll}(\HoldAll, \HoldAll):= \{\ShapeEmbedding\in \operatorname{Emb}(\HoldAll, \HoldAll): \ShapeEmbedding(p) = p \quad \forall p\in \partial\HoldAll \}.
	\end{equation}
	Leaving the outer boundary $\partial\HoldAll$ invariant can be beneficial from practical standpoint, since it might be part of an underlying problem formulation, such a boundary conditions of a PDE.
	
	The pre-shape space \cref{Eq_PreShapeSpaceNonTrivBoundary} in fact is the diffeomorphism group $\operatorname{Diff}_{\partial\HoldAll}(\HoldAll)$ leaving the boundary pointwise invariant.
	This is very important to notice, because it means  $\operatorname{Diff}_{\partial\HoldAll}(\HoldAll)$ is a pre-shape space consisting exactly of one fiber.
	The shape of hold-all domain $\HoldAll$ is fixed, and as a consequence its tangent bundle $T\operatorname{Diff}_{\partial\HoldAll}(\HoldAll)$ consists of vector fields with vanishing normal components on $\partial\HoldAll$ (cf. \cite[Thrm. 3.18]{smolentsev2007diffeomorphism}), i.e.
	\begin{equation}\label{Eq_TangentialSpaceVolDiffeo}
	T_\VolumeEmbedding\operatorname{Diff}_{\partial\HoldAll}(\HoldAll)
	\cong  
	C^\infty_{\partial\HoldAll}(\HoldAll,\R^{n+1}) := \{V\in C^\infty(\HoldAll,\R^{n+1}): \operatorname{Tr}_{\vert\partial\HoldAll}(V)=0 \} \quad \forall \VolumeEmbedding \in \operatorname{Diff}_{\partial\HoldAll}(\HoldAll).
	\end{equation}
	Here, $\operatorname{Tr}_{\vert\partial\HoldAll}(V)$ is the trace of $V$ on $\partial\HoldAll$.
	Even further, the structure theorem for pre-shape derivatives \cite[Thrm. 1]{luft2020pre} tells us that a pre-shape differentiable functional $\TargetPrShp^\HoldAll: \operatorname{Diff}_{\partial\HoldAll}(\HoldAll) \rightarrow \R$ always has trivial shape component $g^\NormalSpaceShape = 0$ of $\PrShpDeriv\TargetPrShp$.
	In particular, ordinary shape calculus is not applicable for functionals of this type.
	
	Our task to design mesh regularization strategies, which are numerically feasible and do not interfere with the original shape optimization problem.
	For a given shape $\ShapeEmbedding(\Manifold) \subset \HoldAll$, the latter is not guaranteed when a pre-shape space $\operatorname{Diff}_{\partial\HoldAll}$ is used to model parameterizations of the hold-all domain.
	A diffeomorphism $\VolumeEmbedding\in \operatorname{Diff}_{\partial\HoldAll}$ could change a given intermediate or optimal shape $\ShapeEmbedding(\Manifold)$, i.e. $\VolumeEmbedding\big(\ShapeEmbedding(\Manifold)\big) \neq \ShapeEmbedding(\Manifold)$ in general.
	Therefore we have to enforce invariance of shapes by further restricting the pre-shape space for $\HoldAll$.
	For this, we use the concept of diffeomorphism groups leaving submanifolds invariant (cf. \cite[Ch. 3.6]{smolentsev2007diffeomorphism}).
	With our previous notation \cref{Eq_PreShapeSpaceNonTrivBoundary}, the final pre-shape  space is $\operatorname{Diff}_{\partial\HoldAll\cup\ShapeEmbedding(\Manifold)}(\HoldAll)$, for a given fixed shape $\ShapeEmbedding(\Manifold)\subset\HoldAll$.
	This space is a subgroup of $\operatorname{Diff}_{\partial\HoldAll}(\HoldAll)$ and $\operatorname{Diff}(\HoldAll)$, with a tangential bundle consisting of vector fields vanishing both on $\partial\HoldAll$ and $\ShapeEmbedding(\Manifold)$ (cf. \cite[Thrm. 3.18]{smolentsev2007diffeomorphism}).
	
	\subsubsection{Volume Parameterization Tracking Problem with Invariant Shapes}
	Now we have a pre-shape space suitable to model reparameterizations of the hold-all domain $\HoldAll$, which can leave a given shape $\ShapeEmbedding(\Manifold)$ invariant.
	Next, we formulate an analogue of parameterization tracking problem \cref{Eq_GeneralPreShapeProblemTangential}, which is tracking for the volume parameterization of the hold-all domain.
	Let us fix a $\ShapeEmbedding \in \operatorname{Emb}(\Manifold, \HoldAll)$ and the corresponding shape $\ShapeEmbedding(\Manifold)\subset\HoldAll$.
	The volume parameterization tracking problem is then
	\begin{equation}\label{Eq_GeneralPreShapeProblemVolume}
		\underset{\VolumeEmbedding \in\operatorname{Diff}_{\partial\HoldAll\cup\ShapeEmbedding(\Manifold)}(\HoldAll)}{\min} \frac{1}{2}\int_{\HoldAll}
		\Big(
		g^\HoldAll \circ\VolumeEmbedding^{-1}(s) \cdot \det D \VolumeEmbedding^{-1}(s)
		-
		f^\HoldAll_{\ShapeEmbedding(\Manifold)}(s) 
		\Big)^2 \; \diff x =: \TargetPrShp^\HoldAll(\VolumeEmbedding).
	\end{equation}
	Notice the similarities and differences to shape parameterization tracking problem \cref{Eq_GeneralPreShapeProblemTangential}.
	Both pre-shape functionals track for a parameterization dictated by target $f$, and both feature a function $g$ describing the initial parameterization.
	The volume tracking functional $\TargetPrShp^\HoldAll$ differs from $\TargetPrShp^\tau$ by featuring a volume integral, instead of a surface one.
	Also, the covariant derivative of the Jacobian determinant in $\TargetPrShp^\HoldAll$ is just the regular Jacobian matrix of $\VolumeEmbedding^{-1}$.
	The two most important differences concern their sets of feasible solutions and targets $f$.
	The target $f^\HoldAll_{\ShapeEmbedding(\Manifold)}$ does not depend on the pre-shape $\VolumeEmbedding$ for the hold-all domain, but instead depends on the shape $\ShapeEmbedding(\Manifold)$ which is left to be invariant.
	Having $f^\HoldAll$ depend on the shape of $\HoldAll$ does not make sense, since $\operatorname{Diff}_{\partial\HoldAll\cup\ShapeEmbedding(\Manifold)}(\HoldAll)$ consists only of one fiber, i.e. the shape of $\HoldAll$ remains invariant.
	Hence there is a dependence of both the feasible set of pre-shapes $\operatorname{Diff}_{\partial\HoldAll\cup\ShapeEmbedding(\Manifold)}(\HoldAll)$ and the target $f^\HoldAll_{\ShapeEmbedding(\Manifold)}$ on the shape $\ShapeEmbedding(\Manifold)$, because we desired $\ShapeEmbedding(\Manifold)$ to stay unaltered.
	For this reason, the volume parameterization tracking problem \cref{Eq_GeneralPreShapeProblemVolume} differs from \cref{Eq_GeneralPreShapeProblemTangential} to such an extent, that \cite[Prop. 2]{luft2020pre} does not cover existence of solutions $\VolumeEmbedding\in\operatorname{Diff}_{\partial\HoldAll\cup\ShapeEmbedding(\Manifold)}(\HoldAll)$ for \cref{Eq_GeneralPreShapeProblemVolume}.
	This makes it necessary to formulate a result guaranteeing existence of solutions for \cref{Eq_GeneralPreShapeProblemVolume} under appropriate conditions.
	\begin{theorem}[Existence of Solutions for the Volume Parameterization Tracking Problem]\label{Thrm_ExistenceVolParamTracking}
	Assume $\HoldAll\subset \R^{n+1}$ to a compact, orientable, path-connected, $n+1$-dimensional $C^\infty$-manifold, possibly with boundary.
	Let $\Manifold$ be an $n$-dimensional, orientable,  path-connected and compact $C^\infty$-submanifold of $\HoldAll$ without boundary.
	Fix a $\ShapeEmbedding\in\operatorname{Emb}(\Manifold, \HoldAll)$ generating a shape $\ShapeEmbedding(\Manifold) \subset \HoldAll$.
	Denote by $\HoldAll_\ShapeEmbedding^{\text{in}}$ and $\HoldAll_\ShapeEmbedding^{\text{out}}$ the disjoint inner and outer part of $\HoldAll$ partitioned by $\ShapeEmbedding(\Manifold)$.

	Let $g^\HoldAll: \HoldAll \rightarrow (0,\infty)$ by a $C^\infty$-function and $f^\HoldAll_{\ShapeEmbedding(\Manifold)}: \HoldAll \rightarrow (0,\infty)$ be $C^\infty$-regular on $\HoldAll\setminus\ShapeEmbedding(\Manifold)$. 
	Further assume the normalization conditions 
	\begin{equation}\label{Assumption_TheoremVolfNormed}
	\int_{\HoldAll_\ShapeEmbedding^{\text{in}}} f_{\ShapeEmbedding(\Manifold)}^\HoldAll(x) \; \diff x 
	=
	\int_{\HoldAll_\ShapeEmbedding^{\text{in}}} g^\HoldAll\; \diff x 
	\quad\text{ and }\quad \int_{\HoldAll_\ShapeEmbedding^{\text{out}}} f_{\ShapeEmbedding(\Manifold)}^\HoldAll(x) \; \diff x 
	=
	\int_{\HoldAll_\ShapeEmbedding^{\text{out}}} g^\HoldAll\; \diff x.
	\end{equation}
	
	Then there exists a global $C^\infty$-solution to problem  \cref{Eq_GeneralPreShapeProblemVolume}, i.e. a diffeomorphism $\tilde{\VolumeEmbedding}\in \operatorname{Diff}(\HoldAll)$ satisfying
	\begin{equation}\label{Eq_EquiCharSolVol}
	(g^\HoldAll\circ\tilde{\VolumeEmbedding}^{-1})\cdot\det D \tilde{\VolumeEmbedding}^{-1} =  f_{\ShapeEmbedding(\Manifold)}^\HoldAll \; \text{ on } \HoldAll\setminus\ShapeEmbedding(\Manifold) \quad \text{ and } \quad
	\tilde{\VolumeEmbedding}(p) = p \quad \forall p\in \partial\HoldAll\cup\ShapeEmbedding(\Manifold).
	\end{equation}
	\end{theorem}
	\begin{proof}
		Let the assumptions on $\HoldAll\subset\R^{n+1}$ and $\Manifold\subset\HoldAll$ be true.
		We fix a $\ShapeEmbedding\in\operatorname{Emb}(\Manifold, \HoldAll)$, and see that $\ShapeEmbedding(\Manifold)\subset\HoldAll$ is an $n$-dimensional, orientable, path-connected and compact $C^\infty$-submanifold of $\HoldAll$ as well.
		With $\ShapeEmbedding(\Manifold)$ being a connected and compact hypersurface of $\HoldAll$, the celebrated Jordan-Brouwer separation theorem (cf. \cite[p. 89]{guillemin2010differential}) guarantees existence of open and disjoint inner $\HoldAll_\ShapeEmbedding^{\text{in}}$ and outer $\HoldAll_\ShapeEmbedding^{\text{out}}$ of $\HoldAll$ separated by $\ShapeEmbedding(\Manifold)$.
		Next, let $g^\HoldAll$ and $f_{\ShapeEmbedding(\Manifold)}^\HoldAll$ be as described, in particular satisfying normalization conditions \cref{Assumption_TheoremVolfNormed}.
		With existence of a separated inner and outer, we can decouple the volume tracking problem \cref{Eq_GeneralPreShapeProblemVolume} into two independent subproblems
		\begin{equation}\label{Eq_VolTrackingSubProblems}
			\underset{\VolumeEmbedding_{\text{in}} \in\operatorname{Diff}_{\partial\HoldAll^{\text{in}}}(\overline{\HoldAll^{\text{in}}})}{\min} \TargetPrShp^{\HoldAll^{\text{in}}}(\VolumeEmbedding_{\text{in}})
			 \quad \text{ and } \quad 
 			\underset{\VolumeEmbedding_{\text{out}} \in\operatorname{Diff}_{\partial\HoldAll^{\text{out}}}(\overline{\HoldAll^{\text{out}}})}{\min} \TargetPrShp^{\HoldAll^{\text{out}}}(\VolumeEmbedding_{\text{out}}).
		\end{equation}
		Both problems do not feature invariant submanifolds in the interior anymore, since $\partial\HoldAll^{\text{in}}_\ShapeEmbedding = \ShapeEmbedding(\Manifold)$ and $\partial\HoldAll^{\text{out}}_\ShapeEmbedding = \ShapeEmbedding(\Manifold)\cup\partial\HoldAll$.
		Thus interior and exterior are both $C^\infty$-manifolds with $C^\infty$-boundaries.
		With this, and two required normalization conditions \cref{Assumption_TheoremVolfNormed}, we are in position to apply existence theorem \cite[Prop. 2]{luft2020pre} for both independent subproblems \cref{Eq_VolTrackingSubProblems}.
		This gives us two $C^\infty$-diffeomorphism $\tilde{\VolumeEmbedding}_{\text{in}}\in\operatorname{Diff}_{\partial\HoldAll^{\text{in}}_\ShapeEmbedding}(\overline{\HoldAll^{\text{in}}_\ShapeEmbedding})$ and $\tilde{\VolumeEmbedding}_{\text{out}}\in\operatorname{Diff}_{\partial\HoldAll^{\text{out}}_\ShapeEmbedding}(\overline{\HoldAll^{\text{out}}_\ShapeEmbedding})$ globally solving \cref{Eq_VolTrackingSubProblems}, which in particular satisfy
		\begin{equation}\label{Temp2}
		(g^\HoldAll\circ\tilde{\VolumeEmbedding}_{\text{in}}^{-1})\cdot\det D \tilde{\VolumeEmbedding}_{\text{in}}^{-1} =  f_{\ShapeEmbedding(\Manifold)}^\HoldAll \; \text{on}\;  \HoldAll^{\text{in}}_\ShapeEmbedding 
		\quad \text{and} \quad (g^\HoldAll\circ\tilde{\VolumeEmbedding}_{\text{out}}^{-1})\cdot\det D \tilde{\VolumeEmbedding}_{\text{out}}^{-1} =  f_{\ShapeEmbedding(\Manifold)}^\HoldAll \; \text{on}\;  \HoldAll^{\text{out}}_\ShapeEmbedding.
		\end{equation}
		We define a global solution candidate for \cref{Eq_GeneralPreShapeProblemVolume} by setting $\tilde{\VolumeEmbedding} := \tilde{\VolumeEmbedding}_{\text{in}} + \tilde{\VolumeEmbedding}_{\text{out}}$.
		It is clear that $\tilde{\VolumeEmbedding}$ is a bijection.
		Also, $\tilde{\VolumeEmbedding}$ is the identity on $\partial\HoldAll\cup\ShapeEmbedding(\Manifold)$, which is the second property of \cref{Eq_EquiCharSolVol}. 
		We know that $\tilde{\VolumeEmbedding}_{\text{in}}$ is $C^\infty$ on $\overline{\HoldAll^{\text{in}}_\ShapeEmbedding}$ and $\tilde{\VolumeEmbedding}_{\text{out}}$ is $C^\infty$ on $\overline{\HoldAll^{\text{out}}_\ShapeEmbedding}$.
		With this, and $\tilde{\VolumeEmbedding}_{\text{in}} = \tilde{\VolumeEmbedding}_{\text{out}} = 0$ on $\ShapeEmbedding(\Manifold)$, we get that $\tilde{\VolumeEmbedding}$ has $C^\infty$ on the entire hold-all domain $\HoldAll$.
		In combination, this means $\tilde{\VolumeEmbedding}\in\operatorname{Diff}_{\partial\HoldAll\cup\ShapeEmbedding(\Manifold)}(\HoldAll)$.
		Using \cref{Temp2}, we also get the first assertion of \cref{Eq_EquiCharSolVol}.
		Finally, we can use \cref{Eq_EquiCharSolVol} to see that $\TargetPrShp^\HoldAll(\tilde{\VolumeEmbedding}) = 0$, since $\ShapeEmbedding(\Manifold)$ is a set of measure zero.
		Due to quadratic nature of \cref{Eq_GeneralPreShapeProblemVolume} we have $\TargetPrShp^\HoldAll \geq 0$, which tells us that $\tilde{\VolumeEmbedding}$ is a global solution to the volume parameterization tracking problem.
		Since we did not use any special property of $\ShapeEmbedding$, the argumentation holds for all $\ShapeEmbedding\in\operatorname{Emb}(\Manifold,\HoldAll)$ and concludes the proof.
	\end{proof}
	
	\begin{remark}[Necessity of two Normalization Conditions for $g^\HoldAll$ and $f_{\ShapeEmbedding(\Manifold)}^\HoldAll$]
	For guaranteed existence of optimal solutions $\VolumeEmbedding\in\operatorname{Diff}_{\partial\HoldAll\cup\Shape}(\HoldAll)$ to volume parameterization tracking \cref{Eq_GeneralPreShapeProblemVolume}, it is necessary to assume to separate normalization conditions \cref{Assumption_TheoremVolfNormed}.
	An invariant shape $\ShapeEmbedding(\Manifold)\subset\HoldAll$ acts as a boundary, partitioning the hold-all domain into inner and outer.
	As we require solutions $\VolumeEmbedding$ to leave $\ShapeEmbedding(\Manifold)$ pointwise invariant, the diffeomorphism $\VolumeEmbedding$ is not allowed to transport volume from outside to inside and vice versa.
	Hence, in general, a single normalization condition on the entire hold-all domain $\HoldAll$ of type \cref{Assumption_fgNormalization} is not sufficient for existence of solutions.
	A direct application of Dacorogna and Moser's theorem \cite{dacorogna1990partial} to \cref{Eq_GeneralPreShapeProblemVolume} yields a $\VolumeEmbedding\in \operatorname{Diff}_{\partial\HoldAll}(\HoldAll)$ possibly transporting volume across $\ShapeEmbedding(\Manifold)$, which we have to restrict if $\ShapeEmbedding(\Manifold)$ is left to be invariant.
	As the total inner and outer volume is changing with varying  $\ShapeEmbedding(\Manifold)$, we have to require normalization condition \cref{Assumption_TheoremVolfNormed} for each $\ShapeEmbedding(\Manifold)$ separately.
	For this reason our target $f^\HoldAll_{\ShapeEmbedding(\Manifold)}$ has to depend on $\ShapeEmbedding(\Manifold)$, even though the shape of $\HoldAll$ stays the same.
	This has crucial implications on the design of targets $f^\HoldAll$ in numerical applications, which we will discuss in \cref{SubSubSection_NumericsGandFTargetConstruction}.
	\end{remark}
	\begin{remark}[Generality of Invariant Shapes]
		In existence result \cref{Thrm_ExistenceVolParamTracking}, we have required the invariant shape to be of form $\ShapeEmbedding(\Manifold)$ for some $\ShapeEmbedding\in\operatorname{Emb}(\Manifold, \HoldAll)$.
		This is solely due to the context of optimization problem \cref{Eq_GeneralShapeProblem} using shape spaces $B_e^n$.
		It is absolutely possible to use any compact and connected hypersurface $\Shape\subset\HoldAll$ as an invariant shape for \cref{Eq_GeneralPreShapeProblemVolume}. 
		Existence of  $\VolumeEmbedding\in\operatorname{Diff}_{\partial\HoldAll\cup\Shape}(\HoldAll)$ globally solving the volume parameterization tracking problem is still guaranteed for general $\Shape$.
	\end{remark}
	As we want to regularize gradient descent methods for shape optimization, we need to explicitly specify the pre-shape derivative to the volume parameterization tracking problem \cref{Eq_GeneralPreShapeProblemVolume}.
	However, it is also of theoretical interest, since it serves as an example of a problem with a pre-shape derivative having vanishing shape component.
	This means that neither the formulation of \cref{Eq_GeneralPreShapeProblemVolume} in the context of classical shape optimization, nor the derivation of a derivative using classical shape calculus are possible. 
	Since the form of $\TargetPrShp^\HoldAll$ is similar to $\TargetPrShp^\tau$, we can mimic the steps from \cite[Thrm. 2]{luft2020pre}, which we do not restate for brevity.
	Then, for given $\VolumeEmbedding\in\operatorname{Diff}_{\partial\HoldAll\cup\ShapeEmbedding(\Manifold)}(\HoldAll)$, the pre-shape derivative of $\Target^\HoldAll$ in decomposed form (cf. \cite[Thrm. 1]{luft2020pre}) is given by
	\begin{equation}\label{Eq_PreShapeDerSplittingVol}
	\PrShpDeriv\TargetPrShp^\HoldAll(\VolumeEmbedding)[V] = \langle g_\VolumeEmbedding^\NormalSpaceShape, V \rangle + \langle g_\VolumeEmbedding^\TangentSpaceShape, V \rangle \quad \forall V\in C_{\partial\Manifold\cup\ShapeEmbedding(\Manifold)}^\infty(\HoldAll, \R^{n+1}),
	\end{equation}
	with normal/ shape component
	\begin{align}\label{Eq_PreShapeDerTrackingNormalVol}
	\begin{split}
	\langle g_\VolumeEmbedding^\NormalSpaceShape, V \rangle 
	= 0
	\end{split}
	\end{align}
	and tangential/ pre-shape component
	\begin{align}\label{Eq_PreShapeDerTrackingTangentialVol}
	\begin{split}
	\langle g_\VolumeEmbedding^\TangentSpaceShape, V \rangle 
	= 
	&-\int_{\HoldAll} \frac{1}{2}\cdot\Big(\big(g^\HoldAll\circ\VolumeEmbedding^{-1}\cdot \det D\VolumeEmbedding^{-1}\big)^2 - {f^\HoldAll_{\ShapeEmbedding(\Manifold)}}^2\Big)
	\cdot\operatorname{div}(V) 
	\\
	&\qquad\qquad\quad+
	\Big(g^\HoldAll\circ\VolumeEmbedding^{-1}\cdot \det D \VolumeEmbedding^{-1} - f^\HoldAll_{\ShapeEmbedding(\Manifold)}\Big)
	\cdot
	\PrShpDeriv_m\big(f^\HoldAll_{\ShapeEmbedding(\Manifold)}\big)[V] \;\diff x.
	\end{split}
	\end{align}
	Here, $n$ is the outer unit normal vector field of the invariant submanifold $\ShapeEmbedding(\Manifold)$.
	As previously mentioned, we signify a duality pairing by $\langle \cdot, \cdot \rangle$.
	It is also important to see that the restriction of $\operatorname{Diff}(\HoldAll)$ to $\operatorname{Diff}_{\partial\Manifold\cup\ShapeEmbedding(\Manifold)}(\HoldAll)$ does not influence the form of pre-shape derivative $\PrShpDeriv\TargetPrShp^\HoldAll$.
	Instead, it influences the space of applicable directions $V$ by restricting $C^\infty(\HoldAll, \R^{n+1})$ to $C^\infty_{\partial\Manifold\cup\ShapeEmbedding(\Manifold)}(\HoldAll, \R^{n+1})$.
	This stems from the relationship of tangential bundles of diffeomorphism groups with invariant submanifolds \cref{Eq_TangentialSpaceVolDiffeo}.
	There is also another subtle difference of $\PrShpDeriv\TargetPrShp^\HoldAll$ and $\PrShpDeriv\TargetPrShp^\tau$.
	Namely, the featured pre-shape material derivative $\PrShpDeriv_m\big(f^\HoldAll_{\ShapeEmbedding(\Manifold)}\big)$ depends on the invariant shape $\ShapeEmbedding(\Manifold)\subset\HoldAll$ instead of the volume pre-shape $\VolumeEmbedding$.
		
	It is straight forward to formulate a pre-shape gradient system for $\PrShpDeriv\TargetPrShp^\HoldAll$ in the style of \cref{SubSection_TangentialMeshQOpt} using Sobolev functions.
	For a symmetric, positive-definite bilinear form $a(.,.)$ it takes the form
	\begin{equation}\label{Eq_GradientSystemVolSolo}
		a(U^{\TargetPrShp^\HoldAll}, V) = 	
		\alpha^\HoldAll\cdot\PrShpDeriv\TargetPrShp^\HoldAll(\VolumeEmbedding)[V] 
		\quad \forall V\in H^1_{\partial\HoldAll\cup\ShapeEmbedding(\Manifold)}(\HoldAll, \R^{n+1}).
	\end{equation}
	Notice that the space of test functions enforces vanishing  $U^{\TargetPrShp^\HoldAll}$ on $\ShapeEmbedding(\Manifold)$.
	With a pre-shape gradient $U^{\TargetPrShp^\HoldAll}$, it is possible to optimize for the volume parameterization tracking problem \cref{Eq_GeneralPreShapeProblemVolume} without altering the shape of $\ShapeEmbedding(\Manifold)$.
	
	\subsubsection{Regularization of Shape Optimization Problems by Volume Parameterization Tracking}\label{SubSubSection_VolReguConsistency}
	At this point, we have provided a suitable space $\operatorname{Diff}_{\partial\HoldAll\cup\ShapeEmbedding(\Manifold)}(\HoldAll)$ for pre-shapes representing the parameterization of hold-all domains $\HoldAll$, which leave a given shape $\ShapeEmbedding(\Manifold)\subset\HoldAll$ invariant.
	Also, we are able to guarantee existence for global minimizers to the volume version of parameterization tracking problem \cref{Eq_GeneralPreShapeProblemVolume} for all shapes $\ShapeEmbedding(\Manifold)$.
	In this subsection we are going to incorporate the volume mesh quality regularization simultaneously with shape optimization and shape mesh quality regularization.
	
	To formulate a regularized version of the original shape optimization problem \cref{Eq_GeneralShapeProblem}, we have to keep the different types of pre-shapes involved in mind.
	These pre-shapes $\ShapeEmbedding\in\operatorname{Emb}(\Manifold, \HoldAll)$ and $\VolumeEmbedding\in \operatorname{Diff}_{\partial\HoldAll\cup\ShapeEmbedding(\Manifold)}(\HoldAll)$ correspond to completely different shapes $\ShapeEmbedding(\Manifold)$ and $\HoldAll$.
	This is also illustrated by looking at the pre-shapes as maps
	\begin{equation}
		\ShapeEmbedding: \Manifold \rightarrow \HoldAll \quad \text{and} \quad \VolumeEmbedding: \HoldAll \rightarrow \HoldAll.
	\end{equation}
	For this reason, we cannot simply proceed by adding $\TargetPrShp^\HoldAll$ in style of a regularizer to increase volume mesh quality.
	From the viewpoint of problem formulation, this signifies a main difference in application of shape mesh quality regularization via $\TargetPrShp^\tau$ and volume mesh quality regularization $\TargetPrShp^\HoldAll$.
	To avoid this issue, we formulate the volume and shape mesh regularized shape optimization problem using a bi-level approach.
	We have already seen in \cref{Remark_BiLevelTangential}, that simultaneous shape parameterization tracking and shape optimization can be put into the bi-level framework.
	In fact, this was seen to be equivalent to the added regularizer approach \cref{Eq_GeneralPreShapeProblemTangential} with regard to gradient systems.
	
	Let us consider weights $\alpha^\HoldAll, \alpha^\tau >0$. Then we formulate the simultaneous volume and shape mesh regularization of shape optimization problem \cref{Eq_GeneralShapeProblem} as
	\begin{align}\label{Eq_GeneralPreShapeProblemVolumeTang}
	\begin{split}
		\underset{\VolumeEmbedding\in \operatorname{Diff}_{\partial\HoldAll\cup\ShapeEmbedding(\Manifold)}(\HoldAll)}{\min}&\;\alpha^\HoldAll\cdot\TargetPrShp^\HoldAll(\VolumeEmbedding) \\
		\text{ s.t. } 		
		\ShapeEmbedding =&\; \underset{\ShapeEmbedding\in \operatorname{Emb}(\Manifold, \HoldAll)}{\operatorname{arg} \min}\;\big(\TargetShp\circ \ProjectionCanonical\big)(\ShapeEmbedding) + \alpha^\tau\cdot\TargetPrShp^\tau(\ShapeEmbedding).
	\end{split}
	\end{align}
	Of course, the bi-level problem \cref{Eq_GeneralPreShapeProblemVolumeTang} can be formulated for $\alpha^\HoldAll = 1$ without loss of generality.
	To stay coherent with \cref{SubSection_TangentialMeshQOpt} regarding pre-shape gradient systems, which feature weighted force terms, we prefer to formulate \cref{Eq_GeneralPreShapeProblemVolumeTang} with a factor $\alpha^\HoldAll>0$.
	Notice, that the regularization of shape optimization problem \cref{Eq_GeneralShapeProblem} needs to use its pre-shape extension \cref{Eq_GeneralShapeProblemPrShpExt}.
	This is necessary in order to rigorously apply the pre-shape regularization strategies.
	In contrast to \cref{Eq_GeneralPreShapeProblemTangential} and \cref{Eq_BilevelTang}, the simultaneous volume and shape mesh quality regularized problem \cref{Eq_GeneralPreShapeProblemVolumeTang} is not minimizing for one pre-shape, but for two different pre-shapes $\ShapeEmbedding\in\operatorname{Emb}(\Manifold, \HoldAll)$ and $\VolumeEmbedding\in \operatorname{Diff}_{\partial\HoldAll\cup\ShapeEmbedding(\Manifold)}(\HoldAll)$.
	The lower level problem solves for a pre-shape corresponding to the actual parameterized shape solving the shape mesh regularized optimization problem \cref{Eq_GeneralPreShapeProblemTangential}.
	On the other hand, the upper level problem looks for a pre-shape corresponding to the parameterization of the hold-all domain $\HoldAll$ with specified volume mesh quality.
	The set of feasible solutions $\operatorname{Diff}_{\partial\HoldAll\cup\ShapeEmbedding(\Manifold)}(\HoldAll)$ to the upper level problem depends on the lower level problem, because the optimal shape to the latter is demanded to stay invariant.
	\begin{remark}[Volume Mesh Quality Regularization]\label{Remark_VolReguNoTangProblem}
		It is of course possible to regularize a shape optimization problem \cref{Eq_GeneralShapeProblem} for volume mesh quality only, neglecting shape mesh parameterization tracking.
		In this scenario, the regularized problem takes the bi-level formulation
		\begin{align}\label{Eq_GeneralPreShapeProblemVolumeNoTang}
		\begin{split}
			\underset{\VolumeEmbedding\in \operatorname{Diff}_{\partial\HoldAll\cup\Shape}(\HoldAll)}{\min}&\;\alpha^\HoldAll\cdot\TargetPrShp^\HoldAll(\VolumeEmbedding) \\
			\text{ s.t. } 		
			\Shape =&\; \underset{\Shape\in B_e^n}{\operatorname{arg} \min}\;\TargetShp(\Shape).
		\end{split}
		\end{align}
		Here, it is not necessary to use the pre-shape expansion \cref{Eq_GeneralShapeProblemPrShpExt} of the original shape optimization problem.
	\end{remark}
	In the remainder of this section we propose a regularized gradient system for simultaneous volume- and shape mesh quality optimization during shape optimization, and prove a corresponding existence and consistency result for the fully regularized problem.
	As done in \cref{SubSection_TangentialMeshQOpt}, we change the space of directions from $C^\infty$ to $H^1$, as it is more suitable for numerical application.
	We remind the reader that the pre-shape derivative $\PrShpDeriv\TargetPrShp^\HoldAll(\VolumeEmbedding)[V]$ is defined only for directions $V\in H^1_{\partial\Manifold\cup\ShapeEmbedding(\Manifold)}(\HoldAll, \R^{n+1})$ which vanish on $\partial\Manifold\cup\ShapeEmbedding(\Manifold)$.
	This is inevitable, since a criterion for successful application of volume mesh regularization for shape optimization routines has to leave optimal or intermediate shapes $\ShapeEmbedding(\Manifold)$ invariant.
	If $\PrShpDeriv\TargetPrShp^\HoldAll(\VolumeEmbedding)[V]$ was used for regularizing the gradient in style of an added source term (cf. \cref{Eq_GradientSystemTangTang}) for general directions $V\in H^1(\HoldAll, \R^{n+1})$, it would most certainly alter shapes and interfere with shape optimization. 
	Also, it is not possible to put Dirichlet conditions on $\ShapeEmbedding(\Manifold)$, or to use a restricted space of test function as in \cref{Eq_GradientSystemVolSolo}.
	Doing so would prohibit shape optimization itself.
	Hence we have to modify $\PrShpDeriv\TargetPrShp^\HoldAll$, such that general directions $V\in H^1(\HoldAll, \R^{n+1})$ are applicable as test functions, while the shape at hand is preserved.
	To resolve this problem, we introduce a projection
	\begin{equation}\label{Eq_ProjectionZeroBoundGeneral}
		\operatorname{Pr}_{H^1_{\partial\Manifold\cup\ShapeEmbedding(\Manifold)}}: H^1 \rightarrow  H^1_{\partial\Manifold\cup\ShapeEmbedding(\Manifold)},
	\end{equation}
	which is demanded to be the identity on $H^1_{\partial\Manifold\cup\ShapeEmbedding(\Manifold)}$.
	We leave the operator projecting a given direction $V\in H^1$ onto $H^1_{\partial\Manifold\cup\ShapeEmbedding(\Manifold)}$ general.
	Suitable options include the projection via solution of a least squares problem
	\begin{equation}
		\operatorname{Pr}_{H^1_{\partial\Manifold\cup\ShapeEmbedding(\Manifold)}}(V) := \underset{W \in H^1_{\partial\Manifold\cup\ShapeEmbedding(\Manifold)}}{\operatorname{arg} \min} \frac{1}{2}\Vert W - V \Vert^2_{H^1}.
	\end{equation}
	In practice, it is feasible to construct a projection \cref{Eq_ProjectionZeroBoundGeneral} by using a finite element representation of $V$ and setting coefficients of basis functions on the discretization of $\ShapeEmbedding(\Manifold)$ to zero.
	With this projection operator, we can extend the volume tracking pre-shape derivative $\PrShpDeriv\TargetPrShp^\HoldAll$ to the space $H^1_{\partial\Manifold\cup\ShapeEmbedding(\Manifold)}$ by using 
	\begin{equation}
		\PrShpDeriv\TargetPrShp^\HoldAll(\VolumeEmbedding)[\operatorname{Pr}_{H^1_{\partial\Manifold\cup\ShapeEmbedding(\Manifold)}}(\cdot)]: H^1(\HoldAll, \R^{n+1}) \rightarrow \R,\; V \rightarrow \PrShpDeriv\TargetPrShp^\HoldAll(\VolumeEmbedding)[\operatorname{Pr}_{H^1_{\partial\Manifold\cup\ShapeEmbedding(\Manifold)}}(V)],
	\end{equation}
	with $\PrShpDeriv\TargetPrShp^\HoldAll$ as in \cref{Eq_PreShapeDerSplittingVol}.
	
	Now we can formulate the fully regularized pre-shape gradient system for simultaneous volume-, shape mesh quality and shape optimization.
	We motivate the combined gradient system with the same formal calculations as in the bi-level formulation for shape quality regularization \cref{Remark_BiLevelTangential}, which are inspired by \cite{savard1994steepest}.
	Given a symmetric, positive-definite bilinear form $a(.,.)$, the gradient system takes the form
	\begin{equation}\label{Eq_GradientSystemTangVol}
		a(U, V) = 	
		\ShpDeriv\TargetShp(\Shape)[V]
		\;+\;
		\alpha^\tau\cdot\langle g_\ShapeEmbedding^\TangentSpaceShape, V \rangle 
		\;+\;
		\alpha^\HoldAll\cdot \PrShpDeriv\TargetPrShp^\HoldAll(\VolumeEmbedding)[\operatorname{Pr}_{H^1_{\partial\Manifold\cup\ShapeEmbedding(\Manifold)}}(V)]
		\quad \forall V\in H^1(\HoldAll, \R^{n+1}),
	\end{equation}
	where $g_\ShapeEmbedding^\TangentSpaceShape$ is the tangential component of shape regularization \cref{Eq_PreShapeDerTrackingTangential} and $\ShpDeriv\TargetShp(\Shape)$ is the shape derivative of the original shape objective with $\Shape=\ProjectionCanonical(\ShapeEmbedding)$.
	Notice that the fully regularized pre-shape gradient system \cref{Eq_GradientSystemTangVol} looks similar to the shape gradient system \cref{Eq_GradientSystemOriginal} of the original problem, differing only by two added force terms on the right hand side.
	These force terms can be thought of as regularizing terms to the original shape gradient.
	In practice, this means simultaneous volume and shape mesh quality improvement for shape optimization amounts to adding two terms on the right hand side of the gradient system.
	Hence they can also be viewed as a (pre-)shape gradient regularization by added force terms.
	
	\begin{theorem}[Volume and Shape Regularized Problems]\label{Thrm_ReguVolTang}
		Let shape optimization problem \cref{Eq_GeneralShapeProblem} be shape differentiable and have a minimizer $\Shape\in B_e^n$.
		For shape and volume parameterization tracking, let the assumptions of both \cref{Thrm_ReguTang} and \cref{Thrm_ExistenceVolParamTracking} be true.
		
		Then there exists a $\ShapeEmbedding\in \ProjectionCanonical(\ShapeEmbedding) = \Shape \subset \operatorname{Emb}(\Manifold, \HoldAll)$ and a  $\VolumeEmbedding\in\operatorname{Diff}_{\partial\HoldAll\cup\ShapeEmbedding(\Manifold)}(\HoldAll)$  
		minimizing the volume and shape regularized bi-level problem \cref{Eq_GeneralPreShapeProblemVolumeTang}.
		
		The fully regularized pre-shape gradient $U$ from system \cref{Eq_GradientSystemTangVol} is consistent with the modified shape regularized gradient $\tilde{U}$ from system \cref{Eq_GradientSystemTangTang} and volume tracking pre-shape gradient $U^{\TargetPrShp^\HoldAll}$ from system \cref{Eq_GradientSystemTangFull}, in the sense that
		\begin{equation}\label{Eq_GradientSystemConsistencyVolTang}
		U = 0 \iff \tilde{U} = 0 \text{ and } U^{\TargetPrShp^\HoldAll} = 0.
		\end{equation}
		In particular, if $U=0$ is satisfied, the necessary first order conditions for volume tracking  \cref{Eq_GeneralPreShapeProblemVolume}, shape tracking  \cref{Eq_GeneralPreShapeProblemTangential} and the original problem \cref{Eq_GeneralShapeProblem} are all satisfied simultaneously.
	\end{theorem}
	\begin{proof}
		For the proof, we need to guarantee existence of solutions to \cref{Eq_GeneralPreShapeProblemVolumeTang}, consistency of gradients \cref{Eq_GradientSystemConsistencyVolTang} and conclude the last assertion concerning necessary first order conditions.
		
		First, let the assumptions of \cref{Thrm_ReguVolTang} be given.
		This includes all assumptions made on $\Manifold$, $\HoldAll$ and functions $g^\Manifold, g^\HoldAll, f_\ShapeEmbedding, f^\HoldAll_{\ShapeEmbedding(\Manifold)}$ summarized in \cref{Thrm_ReguTang} and \cref{Thrm_ExistenceVolParamTracking}.
		Fix a solution $\Shape\in B_e^n$ to the original problem \cref{Eq_GeneralShapeProblem}.
		With the theorem for shape regularized problems \cref{Thrm_ReguTang}, we have guaranteed existence of a solution $\ShapeEmbedding\in\ProjectionCanonical(\ShapeEmbedding)=\Shape\subset\operatorname{Emb}(\Manifold, \HoldAll)$ to the lower level problem of \cref{Eq_GeneralPreShapeProblemVolumeTang}, which coincides with the shape regularized problem \cref{Eq_GeneralPreShapeProblemTangential}.
		Let us select such a solution $\ShapeEmbedding\in\operatorname{Emb}(\Manifold, \HoldAll)$.
		This fixes the set of solution candidates $\operatorname{Diff}_{\partial\HoldAll\cup\ShapeEmbedding(\Manifold)}(\HoldAll)$.
		The existence theorem for volume tracking with invariant shapes \cref{Thrm_ExistenceVolParamTracking} gives existence of a pre-shape $\VolumeEmbedding\in\operatorname{Diff}_{\partial\HoldAll\cup\ShapeEmbedding(\Manifold)}(\HoldAll)$ solving the upper level problem of \cref{Eq_GeneralPreShapeProblemVolumeTang} while leaving $\ShapeEmbedding(\Manifold)$ invariant.
		This proves existence of solutions to the volume and shape regularized bi-level problem \cref{Eq_GeneralPreShapeProblemVolumeTang}.
		
		For consistency of pre-shape gradients \cref{Eq_GradientSystemConsistencyVolTang}, we first prove '$\Rightarrow$' by assuming $U=0$.
		The right hand side of the volume and shape regularized gradient system \cref{Eq_GradientSystemTangVol} consists of three added functionals $\ShpDeriv\TargetShp(\Shape)$, 
		$g_\ShapeEmbedding^\TangentSpaceShape$ and  $\PrShpDeriv\TargetPrShp^\HoldAll(\VolumeEmbedding)[\operatorname{Pr}_{H^1_{\partial\Manifold\cup\ShapeEmbedding(\Manifold)}}(\cdot)]$.
		Due to $U=0$, the full right hand side of \cref{Eq_GradientSystemConsistencyVolTang} vanishes in particular for all directions $V\in H^1_{\partial\Manifold\cup\ShapeEmbedding(\Manifold)}(\HoldAll, \R^{n+1})\subset H^1(\HoldAll, \R^{n+1})$.
		Our functionals $\ShpDeriv\TargetShp(\Shape)$ and 
		$g_\ShapeEmbedding^\TangentSpaceShape$ are only supported on directions $V$ not vanishing on $\ShapeEmbedding(\Manifold)$ due to the structure theorem for pre-shape derivatives \cite[Thrm. 1]{luft2020pre} and their underlying pre-shape space being $\operatorname{Emb}(\Manifold, \HoldAll)$.
		This implies vanishing of $\PrShpDeriv\TargetPrShp^\HoldAll(\VolumeEmbedding)[\operatorname{Pr}_{H^1_{\partial\Manifold\cup\ShapeEmbedding(\Manifold)}}(\cdot)]$ for all $V\in H^1_{\partial\Manifold\cup\ShapeEmbedding(\Manifold)}(\HoldAll, \R^{n+1})$, which in turn gives a vanishing right hand side for \cref{Eq_GradientSystemVolSolo} and $U^{\TargetPrShp^\HoldAll}=0$.
		This also means that individually considered  $\PrShpDeriv\TargetPrShp^\HoldAll(\VolumeEmbedding)[\operatorname{Pr}_{H^1_{\partial\Manifold\cup\ShapeEmbedding(\Manifold)}}(\cdot)]$ vanishes for all directions $V\in H^1(\HoldAll, \R^{n+1})$.
		Thus the remaining part $\ShpDeriv\TargetShp(\Shape)+ \alpha^\tau\cdot \langle g_\ShapeEmbedding^\TangentSpaceShape, \cdot \rangle$ vanishes for all $V\in H^1(\HoldAll, \R^{n+1})$ as well, which immediately gives us $\tilde{U}=0$ by \cref{Eq_GradientSystemTangTang}.
		
		For '$\Leftarrow$', let us assume $\tilde{U}=U^{\TargetPrShp^\HoldAll}=0$.
		Since $\tilde{U}$ vanishes, we see from \cref{Eq_GradientSystemTangTang} that $\ShpDeriv\TargetShp(\Shape)+ \alpha^\tau\cdot \langle g_\ShapeEmbedding^\TangentSpaceShape, \cdot \rangle$ has to vanish too for all $V\in H^1(\HoldAll, \R^{n+1})$.
		And as $\PrShpDeriv\TargetPrShp^\HoldAll$ vanishes for all $V\in H^1_{\partial\Manifold\cup\ShapeEmbedding(\Manifold)}$,  considering the projection operator gives that $\PrShpDeriv\TargetPrShp^\HoldAll(\VolumeEmbedding)[\operatorname{Pr}_{H^1_{\partial\Manifold\cup\ShapeEmbedding(\Manifold)}}(\cdot)]$ vanishes for all $V\in H^1(\HoldAll, \R^{n+1})$.
		Together, this means the complete right hand side of \cref{Eq_GradientSystemTangVol} vanishes, which gives us $U=0$ and establishes consistency \cref{Eq_GradientSystemConsistencyVolTang}.
		
		The last assertion concerning necessary optimality conditions for volume tracking  \cref{Eq_GeneralPreShapeProblemVolume}, shape tracking  \cref{Eq_GeneralPreShapeProblemTangential} and the original problem \cref{Eq_GeneralShapeProblem} are a consequence of consistency \cref{Eq_GradientSystemConsistencyVolTang}.
		If $U=0$, we immediately get $U^{\TargetPrShp^\HoldAll}=0$, which implies the necessary first order condition for volume tracking via \cref{Eq_GradientSystemVolSolo}.
		Consistency \cref{Eq_GradientSystemConsistencyVolTang} and vanishing $U=0$ also give $\tilde{U}=0$.
		The last part of main theorem \cref{Thrm_ReguTang} for shape regularized problems then tells us that necessary first order conditions for shape tracking  \cref{Eq_GeneralPreShapeProblemTangential} and the original problem \cref{Eq_GeneralShapeProblem} are satisfied as well.
	\end{proof}
	
	The theorem for volume and shape regularized problems \cref{Thrm_ReguVolTang} is of great importance, since it guarantees the existence of solutions to the regularized problem \cref{Eq_GeneralPreShapeProblemVolumeTang} for a given shape optimization problem.
	It also tells us, that the shape $\Shape$ solving the original problem \cref{Eq_GeneralShapeProblem} remains unchanged by the volume and shape regularization.
	This is due to two invariance properties. 
	For one, it stems from guaranteed existence of a minimizing pre-shape $\ShapeEmbedding$ in the fiber $\ProjectionCanonical(\ShapeEmbedding)$ corresponding to the optimal shape $\Shape$.
	And secondly, the optimal pre-shape $\VolumeEmbedding$ representing the parameterization of the hold-all domain $\HoldAll$ comes from $\operatorname{Diff}_{\partial\HoldAll\cup\ShapeEmbedding(\Manifold)}(\HoldAll)$, which means it leaves the optimal shape $\ShapeEmbedding(\Manifold)$ pointwise invariant.
	Furthermore \cref{Thrm_ReguVolTang} justifies the use of pre-shape gradient system \cref{Eq_GradientSystemTangVol} modified by force terms for volume and shape regularization.
	This gives a practical and straight forward applicability of volume and shape regularization strategies in shape optimization problems.

	\begin{remark}[Numerical Feasibility]
		Our second criterion for a good regularization strategy also holds.	
		Calculation of a regularized gradient via \cref{Eq_GradientSystemTangVol} is numerically feasible, since it does not require additional solves of (non-)linear systems if compared to the standard shape gradient system \cref{Eq_GradientSystemOriginal}.
		In fact, the volume and shape regularized pre-shape gradient is a combination of three gradients
		\begin{equation}
			U = U^\TargetShp + U^\TangentSpaceShape + U^{\TargetPrShp^\HoldAll},
		\end{equation}
		coming from the original problem \cref{Eq_GradientSystemOriginal}, modified shape tracking \cref{Eq_GradientSystemTangModifiedSolo} and volume tracking \cref{Eq_GradientSystemVolSolo}.
		Instead of solving three systems separately, our approach permits a combined solution of only one linear system with the exact same size of the gradient system \cref{Eq_GradientSystemOriginal} to the original problem.
		Together with invariance of the optimal shape, both criteria for a satisfactory mesh regularization technique are achieved.
	\end{remark}
	\begin{remark}[Applicability of Volume and Shape Regularization for General Shape Optimization Problems]\label{Remark_ValidityVolTangentialGenerality}
		We want to remind the reader, that there is no need for the original shape optimization problem \cref{Eq_GeneralShapeProblem} to have a specific structure.
		It solely needs to be shape differentiable and to have a solution in order to successfully apply volume and shape regularization.
		The exact same assertions made in \cref{Remark_ValidityTangentialGenerality} for shape regularization are also true for simultaneous volume and shape regularization. 
	\end{remark}
	\begin{remark}[Volume Regularization without Shape Regularization]
		As mentioned in \cref{Remark_VolReguNoTangProblem}, it is of course possible to apply volume mesh regularization without shape regularization.
		A result similar to \cref{Thrm_ReguVolTang} can be formulated for the volume regularized problem \cref{Eq_GeneralPreShapeProblemVolumeNoTang} by following analogous arguments.
		In particular, the volume regularized pre-shape gradient system takes the form 
		\begin{equation}
			a(U^{\TargetShp+\TargetPrShp^\HoldAll}, V) = 	
			\ShpDeriv\TargetShp(\Shape)[V]
			\;+\;
			\alpha^\HoldAll\cdot \PrShpDeriv\TargetPrShp^\HoldAll(\VolumeEmbedding)[\operatorname{Pr}_{H^1_{\partial\Manifold\cup\ShapeEmbedding(\Manifold)}}(V)]
			\quad \forall V\in H^1(\HoldAll, \R^{n+1}).
		\end{equation}
		Also, the according consistency of gradients is given by 
		\begin{equation}
			U^{\TargetShp+\TargetPrShp^\HoldAll}=0 \iff U^{\TargetShp}=0 \text{ and } U^{\TargetPrShp^\HoldAll}=0,
		\end{equation}
		for pre-shape gradients $U^{\TargetShp}$ of the original problem \cref{Eq_GradientSystemOriginal} and $U^{\TargetPrShp^\HoldAll}$ volume tracking problem \cref{Eq_GradientSystemVolSolo}.
		The necessary first order conditions for \cref{Eq_GeneralShapeProblem} and \cref{Eq_GeneralPreShapeProblemVolume} if $U^{\TargetShp+\TargetPrShp^\HoldAll}=0$.
	\end{remark}
\section{Implementation of Methods}\label{Section_ImplementationPreShapeRegu}
The theoretical results of shape and volume regularization for shape optimization problems given in \cref{SubSection_TangentialMeshQOpt} and \cref{SubSection_VolumeMeshQOpt} was given in an abstract setting, where the objects involved remained general.
In this section, we give an example which displays the regularization approaches for mesh quality in practice.
The abstract systems and functionals will be stated explicitly, so that the user can apply regularization by referencing the exemplary problem as a guideline.
In \cref{SubSubSection_NumericsModelProblem} we elaborate the process of regularizing a model problem.
We also propose an additional modification for simultaneous shape and volume regularization, which allows for movement of the boundary of the hold-all domain $\partial\HoldAll$ to increase mesh quality.
Thereafter, we present numerical results \cref{SubSubSection_NumericsResults}, comparing several (un-)regularized optimization approaches.
To be more specific, we test two bilinear forms and four regularizations of gradients for a standard gradient descent algorithm with a backtracking line search.
The two bilinear forms are given by the linear elasticity as found in \cite{SchulzSiebenborn} and the $p$-Laplacian inspired by \cite{hinze2021pLap} and studies found in \cite{hinze2021W1Infty}.
Different gradients tested will be the unregularized, the shape regularized, the shape and volume regularized, and the shape and volume regularized one with varying outer boundary.

\subsection{Model Problem and Application of Pre-Shape Mesh Quality Regularization}\label{SubSubSection_NumericsModelProblem}
\subsubsection{Model Problem Formulation and Regularization}
In this section we formulate a model problem to test our pre-shape regularization strategies.
For this, we choose a tracking type shape optimization problem in two dimensions, constrained by a Poisson equation with varying source term.
To highlight the difference of shape and pre-shape calculus techniques, we formulate and test the model problem in two ways.
First, we use the classical shape space framework. 
The second reformulation uses the pre-shape setting, where pre-shape parameterization tracking regularizers can be added. 

To start, we set the model manifold for shapes and the hold-all domain to
\begin{equation}
	\Manifold = S^{0.35}_{(0.5, 0.5)} \text{ and } \HoldAll = [0, 1]\times[0, 2.35].
\end{equation}
The model manifold $\Manifold\subset\HoldAll$ is a sphere with radius $0.35$ centered in $(0.5, 0.5)$, consisting of $63$ surface nodes and edges.
It is embedded in the hold-all domain $\HoldAll$, which is given by a rectangle $[0,1]\times[0, 2.35]$ with non-trivial boundary $\partial\HoldAll$.
The hold-all domain $\HoldAll$ consists of $1402$
nodes and has $741$ volume cells.
They are illustrated in \cref{Fig_NumStartTargetBump3}.
This problem is hard because solution requires a large deformation at a single local region of the initial shape.
Since the mesh is locally refined near the shape, nearby cells are especially prone to degeneration by large deformations.

Notice that the manifold $\Manifold$ acts as an initial shape for the optimization routines to start. 
This approach is always applicable, i.e. the manifold $\Manifold$ for the pre-shape space $\operatorname{Emb}(\Manifold, \HoldAll)$ can always be picked as the initial shape.
With this, the corresponding starting pre-shape is the identity $\operatorname{id}_\Manifold$ of $\Manifold$.

For the shape optimization problem, we employ a piecewise constant source term varying dependent on the shape
\begin{align}\label{Eq_PoissonSourceTerm}
r_{\ShapeEmbedding(\Manifold)}(x) = \begin{cases}
r_1\in \R \; &\text{ for } x\in \overline{\HoldAll^\text{out}}  \\ 
r_2\in \R \; &\text{ for } x\in \HoldAll^\text{in}. 
\end{cases}
\end{align}
A perimeter regularization with parameter $\nu >0$ is added as well.
Combining this, the optimization problem takes the form
\begin{align}\label{Eq_NumericalPoissonTrClassical}
\begin{split}
\underset{\Shape\in B_e^n}{\min}\; \frac{1}{2}\int_{\HoldAll}\vert y &- \bar{y} \vert^2 \; \diff x + \nu \int_{\Shape}1\; \diff s \\
\text{s.t. }-\Delta y &= r_\Shape \quad \text{in } \,\operatorname{int}(\HoldAll) \\
y &= 0 \quad \;\text{ on } \partial\HoldAll.
\end{split}
\end{align}
To calculate the target $\bar{y}\in H^1(\HoldAll)$ of the shape problem, we use the source term \cref{Eq_PoissonSourceTerm} and solve the Poisson problem for the target shape pictured in \cref{Fig_NumStartTargetBump3}.
Problem \cref{Eq_NumericalPoissonTrClassical} is formulated using the classical shape space approach, since the control variable $\Shape$ is stemming from the shape space $B_e^n$, and represents \cref{Eq_GeneralShapeProblem} from the theoretical \cref{Section_TheoryPreShapeRegularization}.

Next, we reformulate \cref{Eq_NumericalPoissonTrClassical} using pre-shapes, while we also add the regularizing term $\TargetPrShp^\tau$ for shape mesh quality with parameter $\alpha^\tau > 0$ 
\begin{align}\label{Eq_Numerics_ProblemTangReg}
\begin{split}
\underset{\ShapeEmbedding\in \operatorname{Emb}(\Manifold, \HoldAll)}{\min}\; \frac{1}{2}\int_{\HoldAll}\vert y &- \bar{y} \vert^2 \; \diff x + \nu \int_{\ShapeEmbedding(\Manifold)}1\; \diff s + \frac{\alpha^\tau}{2}\int_{\ShapeEmbedding(\Manifold)}
\Big(
g^\Manifold \circ\ShapeEmbedding^{-1}(s) \cdot \det D^\tau \ShapeEmbedding^{-1}(s)
-
f_{\ShapeEmbedding}(s) 
\Big)^2 \; \diff s  \\
\text{s.t. }-\Delta y &= r_{\ShapeEmbedding(\Manifold)} \quad \text{in } \,\HoldAll \\
y &= 0 \qquad \;\ \text{ on } \partial\HoldAll.
\end{split}
\end{align}
We remind the reader, that the regularizer can only be added in the pre-shape context, since it is not shape differentiable. 

Technically, the combined volume and shape mesh quality regularized problem is given by formulating a bi-level problem with volume regularizing objective $\TargetPrShp^\HoldAll$ as the upper level problem and lower level problem \cref{Eq_Numerics_ProblemTangReg}, i.e.
\begin{align}\label{Eq_Numerics_ProblemVolTangReg}
\begin{split}
\underset{\VolumeEmbedding\in \operatorname{Diff}_{\partial\HoldAll\cup\ShapeEmbedding(\Manifold)}(\HoldAll)}{\min}\; \frac{\alpha^\HoldAll}{2}&\int_{\HoldAll}
\Big(
g^\HoldAll \circ\VolumeEmbedding^{-1}(x) \cdot \det D \VolumeEmbedding^{-1}(x)
-
f^\HoldAll_{\ShapeEmbedding(\Manifold)}(x) 
\Big)^2 \; \diff x\\
\text{s.t. } \ShapeEmbedding = \underset{\ShapeEmbedding\in \operatorname{Emb}(\Manifold, \HoldAll)}{\operatorname{arg} \min}&\; \frac{1}{2}\int_{\HoldAll}\vert y - \bar{y} \vert^2 \; \diff x + \nu \int_{\ShapeEmbedding(\Manifold)}1\; \diff s \\
&+ \frac{\alpha^\tau}{2}\int_{\ShapeEmbedding(\Manifold)}
\Big(
g^\Manifold \circ\ShapeEmbedding^{-1}(s) \cdot \det D^\tau \ShapeEmbedding^{-1}(s)
-
f_{\ShapeEmbedding}(s) 
\Big)^2 \; \diff s  \\
\text{s.t. }-\Delta y &= r_{\ShapeEmbedding(\Manifold)} \quad \text{in } \,\HoldAll \\
y &= 0 \qquad \;\ \text{ on } \partial\HoldAll.
\end{split}
\end{align}
We remind the reader that, despite its intimidating form, bi-level problem \cref{Eq_Numerics_ProblemVolTangReg} has guaranteed existence of solutions by \cref{Thrm_ReguVolTang}.
The same is true for the shape regularized problem \cref{Eq_Numerics_ProblemTangReg} by \cref{Thrm_ReguTang}.

\subsubsection{Constructing Initial and Target Node Densities $f$ and $g$}\label{SubSubSection_NumericsGandFTargetConstruction}
To explicitly construct the regularizing terms, we need initial node densities $g^\Manifold \in H^1(\Manifold, (0,\infty))$ of $\Manifold$ and $g^\HoldAll \in H^1(\HoldAll, (0,\infty))$ of $\HoldAll$.
Also, we need to specify target node densities $f_\ShapeEmbedding$ and $f^\HoldAll_{\ShapeEmbedding(\Manifold)}$, which describe the cell volume structure of optimal meshes representing $\ShapeEmbedding(\Manifold)$ and $\HoldAll$.

The approach used in this work is to represent the initial point distributions $g^\Manifold$ and $g^\HoldAll$ by using a continuous Galerkin Ansatz with linear elements similar to \cite[Ch. 3]{luft2020pre}. 
Degrees of freedom are situated at the mesh vertices and set to the average of inverses of surrounding cell volumes, i.e.
\begin{equation}\label{Eq_Numerics_gEstimation}
g(p) = \frac{1}{\vert \mathcal{C}\vert}\cdot\sum_{C\in \mathcal{C}}\frac{1}{\operatorname{vol}(C)}.
\end{equation}
In the shape case $g=g^\Manifold$, 
a vertex $p$ is part of the initial discretized shape $\Manifold$ and $\mathcal{C}$ is the set of its neighboring cells $C$ in $\Manifold$.
For $1$-dimensional $\Manifold$ cells $C$ correspond to edges, for $2$-dimensional $\Manifold$ to faces.
In the volume mesh case $g=g^\HoldAll$, $p$ is a vertex of the initial discretized hold-all domain  $\HoldAll$ and $\mathcal{C}$ is the set of its neighboring volume cells $C$ in $\HoldAll$.

Next, we specify a way to construct target parameterizations $f_\ShapeEmbedding$ and $f^\HoldAll_{\ShapeEmbedding(\Manifold)}$, together with their pre-shape material derivatives.
We define a target for shape parameterization tracking $f_\ShapeEmbedding$ by a global target field $\RieszEnergyExtForce: \HoldAll \rightarrow (0,\infty)$.
In order to satisfy normalization condition \cref{Assumption_fgNormalization}, which is necessary for existence of solutions and stable algorithms, a normalization is included.
This gives 
\begin{equation}\label{Eq_Numerics_FTargetShape}
f_{\ShapeEmbedding} 
=
\frac{\int_{\Manifold}g^{\Manifold}\;\diff s}{\int_{\ShapeEmbedding(\Manifold)}\RieszEnergyExtForce_{\vert\ShapeEmbedding(\Manifold)}\;\diff s}\cdot \RieszEnergyExtForce_{\vert\ShapeEmbedding(\Manifold)}.
\end{equation}
With this construction, the targeted parameterization of $\ShapeEmbedding(\Manifold)$ depends on its location and shape in $\HoldAll$, as $\RieszEnergyExtForce: \HoldAll \rightarrow (0,\infty)$ is allowed to vary on the whole domain.

The according material derivative is derived in \cite{luft2020pre} and has closed form 
\begin{equation}\label{Eq_Numerics_MaterialDerivFShape}
	\PrShpDeriv_m(f_\ShapeEmbedding)[V] = 
	-\frac{\int_{\Manifold}g^{\Manifold}\;\diff s}{\big(\int_{\ShapeEmbedding(\Manifold)}\RieszEnergyExtForce\;\diff s\big)^2}
	\cdot\RieszEnergyExtForce\cdot
	\int_{\ShapeEmbedding(\Manifold)} \frac{\partial \RieszEnergyExtForce}{\partial n}\cdot\langle V, n \rangle \;\diff s
	+
	\frac{\int_{\Manifold}g^{\Manifold}\;\diff s}{\int_{\ShapeEmbedding(\Manifold)}\RieszEnergyExtForce\;\diff s} 
	\cdot
	\nabla \RieszEnergyExtForce^T V.
\end{equation}
Notice that \cref{Eq_Numerics_MaterialDerivFShape} includes both normal and tangential components.
However, only its tangential component is needed if regularized gradient systems \cref{Eq_GradientSystemTangTang} and \cref{Eq_GradientSystemTangVol} are applied.
We will write down explicit right hand sides to gradient systems for our exemplary problem in \cref{SubSubSection_NumericsGradientSystems}.

For a volume target $f^\HoldAll_{\ShapeEmbedding(\Manifold)}$, we have to satisfy the different normalization condition \cref{Assumption_TheoremVolfNormed} to guarantee existence of solutions.
We propose to use a field $q^\HoldAll: \HoldAll \rightarrow (0,\infty)$ defined on the hold-all domain.
Then, an according target can be defined as
\begin{equation}\label{Eq_Numerics_FTargetVol}
	f^\HoldAll_{\ShapeEmbedding(\Manifold)} = 
	\begin{cases}
		\frac{\int_{\HoldAll^{\text{in}}_\ShapeEmbedding}g^{\HoldAll}\;\diff x}{\int_{\HoldAll^{\text{in}}_\ShapeEmbedding}\RieszEnergyExtForce^\HoldAll\;\diff x}\cdot \RieszEnergyExtForce^\HoldAll 
		\quad \;\,\text{ for } x \in \HoldAll^{\text{in}}_\ShapeEmbedding\\
		\frac{\int_{\HoldAll^{\text{out}}_\ShapeEmbedding}g^{\HoldAll}\;\diff x}{\int_{\HoldAll^{\text{out}}_\ShapeEmbedding}\RieszEnergyExtForce^\HoldAll\;\diff x}\cdot \RieszEnergyExtForce^\HoldAll 
		\quad \text{ for } x \in \overline{\HoldAll^{\text{out}}_\ShapeEmbedding}.
	\end{cases}
\end{equation}
This is different to the construction of targets  $f_\ShapeEmbedding$ for embedded shapes, since the function $f^\HoldAll_{\ShapeEmbedding(\Manifold)}$ does only change in order to guarantee normalization condition \cref{Assumption_TheoremVolfNormed}.
It cannot vary due to the change of shape of $\HoldAll$, which remains fixed.
This stays in contrast to the situation for $\ShapeEmbedding(\Manifold)\subset\HoldAll$, which can change its position in $\HoldAll$.
Also notice that $f^\HoldAll_{\ShapeEmbedding(\Manifold)}$ as defined in \cref{Eq_Numerics_FTargetVol} can be non-continuous on the shape $\ShapeEmbedding(\Manifold)$.
However, in existence and consistency results \cref{Thrm_ExistenceVolParamTracking} and \cref{Thrm_ReguVolTang} we have not demanded continuity or smoothness of $f^\HoldAll_{\ShapeEmbedding(\Manifold)}$ on the entire domain $\HoldAll$.
Smoothness of $f^\HoldAll_{\ShapeEmbedding(\Manifold)}$ is only demanded for the inner $\HoldAll^{\text{in}}_\ShapeEmbedding$ and outside $\HoldAll^{\text{out}}_\ShapeEmbedding$ partitioned by $\ShapeEmbedding(\Manifold)$.

Now we derive the pre-shape material derivative $\PrShpDeriv_m(f^\HoldAll_{\ShapeEmbedding(\Manifold)})[V]$ for directions $H^1_{\partial\HoldAll}$. 
These directions are not forced to vanish on the shape $\ShapeEmbedding(\Manifold)$, which is needed to assemble combined gradients systems with $V$ acting as test functions.
This poses a difficulty in its derivation, the partitioning depends on the pre-shape $\ShapeEmbedding\in\operatorname{Emb}(\Manifold,\HoldAll)$, but not on $\VolumeEmbedding\in\operatorname{Diff}_{\partial\HoldAll\cup\ShapeEmbedding(\Manifold)}(\HoldAll)$.
Let us fix a $\ShapeEmbedding\in\operatorname{Emb}(\Manifold,\HoldAll)$ and compute on the outer domain $\HoldAll^\text{out}_\ShapeEmbedding$ with pre-shape calculus rules from \cite{luft2020pre}.
In this computation we write $\HoldAll^\text{out}$ instead of   $\HoldAll^\text{out}_\ShapeEmbedding$ for readability.
\begin{align}
\begin{split}
\PrShpDeriv_m(f^\HoldAll_{\ShapeEmbedding(\Manifold)})[V]_{\vert \HoldAll^\text{out}} 
=&\;
\PrShpDeriv_m\Bigg(\frac{\int_{\HoldAll^{\text{out}}}g^{\HoldAll}\;\diff x}{\int_{\HoldAll^{\text{out}}}\RieszEnergyExtForce^\HoldAll\;\diff x}\cdot \RieszEnergyExtForce^\HoldAll \Bigg)[V] \\
%&= -\frac{1}{\big(\int_{\ShapeEmbedding(\Manifold)}\RieszEnergyExtForce\;\diff s\big)^2}
%\cdot
%\PrShpDeriv_m\Big(\int_{\ShapeEmbedding(\Manifold)}\RieszEnergyExtForce\;\diff s\Big)[V]
%\cdot
%\RieszEnergyExtForce \\
%&\qquad + \frac{1}{\int_{\ShapeEmbedding(\Manifold)}\RieszEnergyExtForce\;\diff s} 
%\cdot \PrShpDeriv_m(\RieszEnergyExtForce)[V]  \\
= &\;\frac{\RieszEnergyExtForce^\HoldAll}{\int_{\HoldAll^\text{out}}\RieszEnergyExtForce^\HoldAll\;\diff x}
\cdot
\int_{\HoldAll^\text{out}}\big(\PrShpDeriv(g^\HoldAll)[V] + \nabla (g^\HoldAll)^T V
+ \operatorname{div}(V)\cdot g^\HoldAll\big)\;\diff x  \\
& -\frac{\int_{\HoldAll^\text{out}}g^{\HoldAll}\;\diff x}{\big(\int_{\HoldAll^\text{out}}\RieszEnergyExtForce^\HoldAll\;\diff x\big)^2}
\cdot\RieszEnergyExtForce^\HoldAll\cdot
\int_{\HoldAll^\text{out}}\big(\PrShpDeriv(\RieszEnergyExtForce^\HoldAll)[V] + \nabla (\RieszEnergyExtForce^\HoldAll)^T V
+ \operatorname{div}(V)\cdot \RieszEnergyExtForce^\HoldAll\big)\;\diff x  \\
&+ 
\frac{\int_{\HoldAll^\text{out}}g^{\HoldAll}\;\diff x}{\int_{\HoldAll^\text{out}}\RieszEnergyExtForce^\HoldAll\;\diff x} 
\cdot 
\Big(\PrShpDeriv(\RieszEnergyExtForce^\HoldAll)[V] + \nabla (\RieszEnergyExtForce^\HoldAll)^T V\Big)	\\
= &\;\frac{\RieszEnergyExtForce^\HoldAll}{\int_{\HoldAll^\text{out}}\RieszEnergyExtForce^\HoldAll\;\diff x}
\cdot
\int_{\HoldAll^\text{out}}\operatorname{div}\big(g^\HoldAll \cdot V\big)\;\diff x  \\
& -\frac{\int_{\HoldAll^\text{out}}g^{\HoldAll}\;\diff x}{\big(\int_{\HoldAll^\text{out}}\RieszEnergyExtForce^\HoldAll\;\diff x\big)^2}
\cdot\RieszEnergyExtForce^\HoldAll\cdot
\int_{\HoldAll^\text{out}}\operatorname{div}\big(\RieszEnergyExtForce^\HoldAll\cdot V\big)\;\diff x  \\
&+ 
\frac{\int_{\HoldAll^\text{out}}g^{\HoldAll}\;\diff x}{\int_{\HoldAll^\text{out}}\RieszEnergyExtForce^\HoldAll\;\diff x} 
\cdot 
\nabla (\RieszEnergyExtForce^\HoldAll)^T V \\
=&\; \frac{\RieszEnergyExtForce^\HoldAll}{\int_{\HoldAll^\text{out}}\RieszEnergyExtForce^\HoldAll\;\diff x}
\cdot \int_{\partial\HoldAll\cup\ShapeEmbedding(\Manifold)} \Big( g^\HoldAll - \frac{\int_{\HoldAll^{\text{out}}}g^{\HoldAll}\;\diff x}{\int_{\HoldAll^{\text{out}}}\RieszEnergyExtForce^\HoldAll\;\diff x}\cdot \RieszEnergyExtForce^\HoldAll  \Big) \cdot \big\langle V, n_{\HoldAll^\text{out}}\big\rangle \;\diff s \\
&+ 
\frac{\int_{\HoldAll^\text{out}}g^{\HoldAll}\;\diff x}{\int_{\HoldAll^\text{out}}\RieszEnergyExtForce^\HoldAll\;\diff x} 
\cdot 
\nabla (\RieszEnergyExtForce^\HoldAll)^T V \\
=&\; -\frac{\RieszEnergyExtForce^\HoldAll}{\int_{\HoldAll^\text{out}}\RieszEnergyExtForce^\HoldAll\;\diff x}
\cdot \int_{\ShapeEmbedding(\Manifold)} \Big( g^\HoldAll - \frac{\int_{\HoldAll^{\text{out}}}g^{\HoldAll}\;\diff x}{\int_{\HoldAll^{\text{out}}}\RieszEnergyExtForce^\HoldAll\;\diff x}\cdot \RieszEnergyExtForce^\HoldAll  \Big) \cdot \big\langle V, n_{\ShapeEmbedding(\Manifold)}\big\rangle \;\diff s \\
&+ 
\frac{\int_{\HoldAll^\text{out}}g^{\HoldAll}\;\diff x}{\int_{\HoldAll^\text{out}}\RieszEnergyExtForce^\HoldAll\;\diff x} 
\cdot 
\nabla (\RieszEnergyExtForce^\HoldAll)^T V.
\end{split}
\end{align}
Here, $n_{\HoldAll^{\text{out}}}$ is the outer unit normal vector field on $\partial\HoldAll^\text{out} = \partial\HoldAll\cup\ShapeEmbedding(\Manifold)$, and $n_{\ShapeEmbedding(\Manifold)}$ is the outer unit normal vector field on $\ShapeEmbedding(\Manifold)$.
In particular, we used that $g^\HoldAll$ and $\RieszEnergyExtForce^\HoldAll$ do neither depend on $\VolumeEmbedding$ nor on $\ShapeEmbedding$, which lets their pre-shape derivatives vanish.
Also, we have applied Gauss' theorem and used  $V_{\partial\HoldAll}=0$.
Notice the change of sign for the first summand of the last equality, due to $n_{\HoldAll^{\text{out}}} = -n_{\ShapeEmbedding(\Manifold)}$ on $\ShapeEmbedding(\Manifold)$.
Analogous computation on the interior $\HoldAll^{\text{in}}$ with boundary $\partial\HoldAll^{\text{in}} = \ShapeEmbedding(\Manifold)$ give us the pre-shape material derivative 
\begin{equation}\label{Eq_Numerics_MaterialDerivVolTarg}
	\PrShpDeriv_m(f^\HoldAll_{\ShapeEmbedding(\Manifold)})[V] = \begin{cases}
		\frac{\RieszEnergyExtForce^\HoldAll}{\int_{\HoldAll^\text{in}}\RieszEnergyExtForce^\HoldAll\;\diff x}
		\cdot \int_{\ShapeEmbedding(\Manifold)} \Big( g^\HoldAll - \frac{\int_{\HoldAll^{\text{in}}}g^{\HoldAll}\;\diff x}{\int_{\HoldAll^{\text{in}}}\RieszEnergyExtForce^\HoldAll\;\diff x}\cdot \RieszEnergyExtForce^\HoldAll  \Big) \cdot \big\langle V, n_{\ShapeEmbedding(\Manifold)}\big\rangle \;\diff s \\
		\qquad\qquad\qquad\qquad\qquad\qquad+\;
		\frac{\int_{\HoldAll^\text{in}}g^{\HoldAll}\;\diff x}{\int_{\HoldAll^\text{in}}\RieszEnergyExtForce^\HoldAll\;\diff x} 
		\cdot 
		\nabla (\RieszEnergyExtForce^\HoldAll)^T V
		\qquad \;\;\,\text{ for } x \in \HoldAll^{\text{in}}\\
		-\frac{\RieszEnergyExtForce^\HoldAll}{\int_{\HoldAll^\text{out}}\RieszEnergyExtForce^\HoldAll\;\diff x}
		\cdot \int_{\ShapeEmbedding(\Manifold)} \Big( g^\HoldAll - \frac{\int_{\HoldAll^{\text{out}}}g^{\HoldAll}\;\diff x}{\int_{\HoldAll^{\text{out}}}\RieszEnergyExtForce^\HoldAll\;\diff x}\cdot \RieszEnergyExtForce^\HoldAll  \Big) \cdot \big\langle V, n_{\ShapeEmbedding(\Manifold)}\big\rangle \;\diff s \\
		\qquad\qquad\qquad\qquad\qquad\qquad+\; 
		\frac{\int_{\HoldAll^\text{out}}g^{\HoldAll}\;\diff x}{\int_{\HoldAll^\text{out}}\RieszEnergyExtForce^\HoldAll\;\diff x} 
		\cdot 
		\nabla (\RieszEnergyExtForce^\HoldAll)^T V
		\qquad\;\, \text{ for } x \in \overline{\HoldAll^{\text{out}}}.
	\end{cases}
\end{equation}

This pre-shape material derivative is interesting from a theoretical perspective, since it is an example of a derivative depending on the shape of a submanifold $\ShapeEmbedding(\Manifold)\subset\HoldAll$, where the actual pre-shape at hand $\HoldAll$ is of different dimension.
Also, we see that the sign of boundary integral on $\ShapeEmbedding(\Manifold)$ depends on whether the inside or outside of $\HoldAll$ is regarded.
This nicely reflects that changing $\ShapeEmbedding(\Manifold)$ adds volume on one side and takes it away from the other.
We remind the reader that normal directions $n_{\ShapeEmbedding(\Manifold)}$ are not normal directions corresponding to the shape of $\HoldAll$.
They rather lie in the interior of $\HoldAll$, and hence are part of the fiber or tangential component of $\operatorname{Diff}(\HoldAll)=\operatorname{Emb}(\HoldAll, \HoldAll)$.

\subsubsection{Pre-Shape Gradient Systems}\label{SubSubSection_NumericsGradientSystems}
To compute pre-shape gradients $U$ we need suitable bilinear forms $a(.,.)$.
The systems for our gradients are always of form
\begin{align}\label{Eq_Numerics_GradientSystemAbstract}
\begin{split}
	a(U,V) &= \operatorname{RHS}(\ShapeEmbedding, \VolumeEmbedding)[V] \quad \forall H^1_{\partial\HoldAll}(\HoldAll, \R^{n+1}) \\
	U &= \operatorname{BC} \qquad\quad \text {on } \partial\HoldAll
\end{split}
\end{align}
In our numerical implementations, we test two bilinear forms and four different right hand sides.
We abbreviate the right hand sides by $\operatorname{RHS}(\ShapeEmbedding, \VolumeEmbedding)[V]$  depending on pre-shapes $\ShapeEmbedding\in\operatorname{Emb}(\Manifold, \HoldAll), \VolumeEmbedding\in\operatorname{Diff}_{\partial\HoldAll\cup\ShapeEmbedding(\Manifold)}(\HoldAll)$ and test functions $V$, and boundary conditions by $\operatorname{BC}$.
First, we consider the weak formulation of the linear elasticity equation with zero first Lam\'e parameter as found in \cite{SchulzSiebenborn}
\begin{align}\label{Eq_MetricLinElas}
\begin{split}
\int_{\HoldAll}\mu\cdot\epsilon(U):\epsilon(V)\;\diff x &= \operatorname{RHS}(\ShapeEmbedding, \VolumeEmbedding)[V] \qquad \forall V\in H^1_0(\HoldAll, \mathbb{R}^{n+1})\\
\epsilon(U) &= \frac{1}{2}(\nabla U^T + \nabla U)\\
\epsilon(V) &= \frac{1}{2}(\nabla V^T + \nabla V)\\
U &= 0 \qquad \text{ on } \partial\HoldAll.
\end{split}	
\end{align}
As the second bilinear form, we consider the weak formulation of the vector valued $p$-Laplacian equation.
Since systems stemming from the $p$-Laplacian have the issue to be indefinite, we employ a standard regularization by adding a parameter $\varepsilon>0$.
To make a comparison with the linear elasticity \cref{Eq_MetricLinElas} viable, we use a local weighting $\mu: \HoldAll \rightarrow (0,\infty)$ in the bilinear form, which then is 
\begin{align}\label{Eq_MetricPLaplacian}
	\int_{\HoldAll}\mu\cdot\Big(\varepsilon^2 + \nabla U : \nabla U \Big)^{\frac{p}{2}-1}\cdot \nabla U:\nabla V\;\diff x &= \operatorname{RHS}(\ShapeEmbedding, \VolumeEmbedding)[V] \qquad \forall V\in H^1_0(\HoldAll, \mathbb{R}^{n+1}).
\end{align}
We chose the local weighting $\mu$ as the solution of Poisson problem
\begin{equation}\label{Eq_Numerics_LamePoisson}
\begin{split}
-\Delta \mu &= 0 \qquad \;\;\; \text{in } \HoldAll \\
\mu &= \mu_{\text{max}} \quad\, \text{on } \ShapeEmbedding(\Manifold) \\
\mu &= \mu_{\text{min}} \quad\; \text{on } \partial\HoldAll
\end{split}
\end{equation}
for $\mu_{\text{max}}, \mu_{\text{min}} > 0$.
In the context of linear elasticity \cref{Eq_MetricLinElas} it can be interpreted as the so-called second Lam\'e parameter.

\begin{remark}[Sufficiency of Linear Elements for Pre-Shape Regularization]
	In order to apply pre-shape regularization approaches presented in this work, it is completely sufficient to use continuous linear elements to represent involved functions.
	As we can see in the pre-shape derivative formulas  \cref{Eq_PreShapeDerTrackingTangential} and  \cref{Eq_PreShapeDerTrackingTangentialVol}, the highest order of featured derivatives is one.
	This is important for application in practice, since existing shape gradient systems do not require higher order elements for volume and shape mesh quality tracking.
	In particular, all following systems are built by using continuous first order elements in FEniCS.
\end{remark}

Next, we need the shape derivative of the PDE constrained tracking type shape optimization objective $\TargetShp$.
It can be derived by a Lagrangian approach using standard shape or pre-shape calculus rules, giving
\begin{align}
\begin{split}
\ShpDeriv\TargetShp(\ProjectionCanonical(\ShapeEmbedding))[V] = &
\int_{\HoldAll}-(y-\bar{y})\nabla \bar{y}^TV 
- \nabla y^T(\nabla V^T + \nabla V)\nabla p
+ \operatorname{div}(V)\Big(\frac{1}{2}(y-\bar{y})^2 + \nabla y^T \nabla p - r_{\ShapeEmbedding(\Manifold)} p\Big)\;\diff x.
\end{split}
\end{align}
Here, $p$ is the adjoint solving the adjoint system
\begin{align}\label{Eq_Numerics_AdjointPoisson}
\begin{split}
-\Delta p &= -(y-\bar{y}) \quad \text{in } \,\HoldAll \\
p &= 0 \qquad\qquad \text{ on } \partial\HoldAll.
\end{split}
\end{align}
It is straight forward to derive the shape derivative of the perimeter regularization $\TargetShp^{\text{Perim}}$, which takes the form
\begin{align}
\ShpDeriv\TargetShp^{\text{Perim}}(\ProjectionCanonical(\ShapeEmbedding))[V] = \int_{\ShapeEmbedding(\Manifold)}\operatorname{div}_\Gamma(V)\; \diff s,
\end{align}
where $\operatorname{div}_\Gamma(V)$ is the tangential divergence of $V$ on $\ShapeEmbedding(\Manifold)$.

In the following we give four right hand sides representing various (un-)regularized approaches to calculate pre-shape gradients.
They correspond to the unregularized shape gradient, the shape parameterization tracking regularized pre-shape gradient, the volume and parameterization tracking regularized pre-shape gradient, and the volume and parameterization tracking regularized pre-shape gradient with free tangential outer boundary.

For the \emph{unregularized shape gradient}, the right hand side of the gradient system \cref{Eq_Numerics_GradientSystemAbstract} takes the standard form
\begin{align}\label{Eq_Numerics_RHS_Unregularized}
\begin{split}
\operatorname{RHS}(\ShapeEmbedding, \VolumeEmbedding)[V]
=\;&
 \ShpDeriv\TargetShp(\ProjectionCanonical(\ShapeEmbedding))[V] + \nu\cdot\ShpDeriv\TargetShp^{\text{Perim}}(\ProjectionCanonical(\ShapeEmbedding))[V] \\
=\;& 
\int_{\HoldAll}-(y-\bar{y})\nabla \bar{y}^TV 
- \nabla y^T(\nabla V^T + \nabla V)\nabla p
+ \operatorname{div}(V)\Big(\frac{1}{2}(y-\bar{y})^2 + \nabla y^T \nabla p - r_{\ShapeEmbedding(\Manifold)} p\Big)\;\diff x \\
&+
\nu\cdot \int_{\ShapeEmbedding(\Manifold)}\operatorname{div}_\Gamma(V)\; \diff s
\qquad \forall V\in H^1_{\partial\HoldAll}(\HoldAll, \mathbb{R}^{n+1}).
\end{split}	
\end{align}
In this case, the respective boundary condition for the gradient system is simply a Dirichlet zero condition $\operatorname{BC} = 0$.

Next, we give the right hand side for the \emph{shape parameterization regularized pre-shape gradient}.
For shape parameterization tracking, we employ a target $f_\ShapeEmbedding$ given by a globally defined function $\RieszEnergyExtForce: \HoldAll \rightarrow (0,\infty)$ (cf. \cref{Eq_Numerics_FTargetShape}), which in combination yields
\begin{align}\label{Eq_Numerics_RHS_ShapeRegu}
\begin{split}
\operatorname{RHS}(\ShapeEmbedding, \VolumeEmbedding)[V]
=\;&
\ShpDeriv\TargetShp(\ProjectionCanonical(\ShapeEmbedding))[V] + \nu\cdot\ShpDeriv\TargetShp^{\text{Perim}}(\ProjectionCanonical(\ShapeEmbedding))[V] + \alpha^\tau\cdot\langle g_\ShapeEmbedding^\TangentSpaceShape, V\rangle \\
=\;& 
\int_{\HoldAll}-(y-\bar{y})\nabla \bar{y}^TV 
- \nabla y^T(\nabla V^T + \nabla V)\nabla p
+ \operatorname{div}(V)\Big(\frac{1}{2}(y-\bar{y})^2 + \nabla y^T \nabla p - r_{\ShapeEmbedding(\Manifold)} p\Big)\;\diff x \\
&+
\nu\cdot \int_{\ShapeEmbedding(\Manifold)}\operatorname{div}_\Gamma(V)\; \diff s \\
&-
\int_{\ShapeEmbedding(\Manifold)} \frac{1}{2}\cdot\Bigg(\big(g^\Manifold\circ\ShapeEmbedding^{-1}\cdot \det D^\tau \ShapeEmbedding^{-1}\big)^2 - \Big(\frac{\int_{\Manifold}g^{\Manifold}\;\diff s}{\int_{\ShapeEmbedding(\Manifold)}\RieszEnergyExtForce\;\diff s}\cdot \RieszEnergyExtForce\Big)^2\Bigg)
\cdot\operatorname{div}_\Shape (V - \langle V, n \rangle \cdot n) 
\\
&\qquad\qquad\quad+
\Big(g^\Manifold\circ\ShapeEmbedding^{-1}\cdot \det D^\tau \ShapeEmbedding^{-1} - \frac{\int_{\Manifold}g^{\Manifold}\;\diff s}{\int_{\ShapeEmbedding(\Manifold)}\RieszEnergyExtForce\;\diff s}\cdot \RieszEnergyExtForce\Big)
\cdot
\frac{\int_{\Manifold}g^{\Manifold}\;\diff s}{\int_{\ShapeEmbedding(\Manifold)}\RieszEnergyExtForce\;\diff s}\cdot\nabla_\Shape \RieszEnergyExtForce^T V \;\diff s
\qquad \forall V\in H^1_{\partial\HoldAll}(\HoldAll, \mathbb{R}^{n+1}).
\end{split}	
\end{align}
The boundary condition is Dirichlet zero $\operatorname{BC} = 0$.
In order to assemble the shape regularization $\langle g_\ShapeEmbedding^\TangentSpaceShape, V\rangle$, it is necessary to compute the tangential Jacobian $\operatorname{det}D^\tau \ShapeEmbedding^{-1}$.
In applications, this means storing the vertex coordinates becomes of the initial shape is necessary.
Then $\ShapeEmbedding^{-1}$ can be calculated simply as the difference of current shape node coordinates to the initial ones.
Hence there is no need to invert matrices to calculate $D^\tau\ShapeEmbedding^{-1}$.
We give a strong reminder that $D^\tau$ is the covariant derivative, and must not be confused with the tangential derivative (cf. \cite[Ex. 2]{luft2020pre}).
In the case of an $n$-dimensional manifold $\ShapeEmbedding(\Manifold$, the covariant derivative is a $n\times n$-matrix, whereas the tangential derivative is $(n+1)\times (n+1)$.
The use of covariant derivatives requires to calculate local orthonormal frames, which can be done by standard Gram-Schmidt algorithms.
Knowing this, the computation of Jacobian determinants is inexpensive, since matrices from applications are of size smaller $3\times 3$.

Building on \cref{Eq_Numerics_RHS_ShapeRegu}, we can construct the right hand side for the \emph{shape and volume parameterization regularized pre-shape gradient}.
For this, we use a volume tracking target $f^\HoldAll_{\ShapeEmbedding(\Manifold)}$ defined by a field $\RieszEnergyExtForce^\HoldAll: \HoldAll \rightarrow (0,\infty)$ (\cref{Eq_Numerics_FTargetVol}). 
This finally gives
\begin{align}\label{Eq_Numerics_RHS_VolShapeRegu}
\hspace{-3.7cm}
\begin{split}
&\qquad\qquad\qquad\operatorname{RHS}(\ShapeEmbedding, \VolumeEmbedding)[V]
=\;
\ShpDeriv\TargetShp(\ProjectionCanonical(\ShapeEmbedding))[V] + \nu\cdot\ShpDeriv\TargetShp^{\text{Perim}}(\ProjectionCanonical(\ShapeEmbedding))[V] + \alpha^\tau\cdot\langle g_\ShapeEmbedding^\TangentSpaceShape, V\rangle 
+
\alpha^\HoldAll \PrShpDeriv\TargetPrShp^\HoldAll(\VolumeEmbedding)[V]\\
=\;& 
\int_{\HoldAll}-(y-\bar{y})\nabla \bar{y}^TV 
- \nabla y^T(\nabla V^T + \nabla V)\nabla p
+ \operatorname{div}(V)\Big(\frac{1}{2}(y-\bar{y})^2 + \nabla y^T \nabla p - r_{\ShapeEmbedding(\Manifold)} p\Big)\;\diff x \\
&+
\nu\cdot \int_{\ShapeEmbedding(\Manifold)}\operatorname{div}_\Gamma(V)\; \diff s \\
&-
\int_{\ShapeEmbedding(\Manifold)} \frac{1}{2}\cdot\Bigg(\big(g^\Manifold\circ\ShapeEmbedding^{-1}\cdot \det D^\tau \ShapeEmbedding^{-1}\big)^2 - \Big(\frac{\int_{\Manifold}g^{\Manifold}\;\diff s}{\int_{\ShapeEmbedding(\Manifold)}\RieszEnergyExtForce\;\diff s}\cdot \RieszEnergyExtForce\Big)^2\Bigg)
\cdot\operatorname{div}_\Shape (V - \langle V, n \rangle \cdot n) 
\\
&\qquad\qquad\quad+
\Big(g^\Manifold\circ\ShapeEmbedding^{-1}\cdot \det D^\tau \ShapeEmbedding^{-1} - \frac{\int_{\Manifold}g^{\Manifold}\;\diff s}{\int_{\ShapeEmbedding(\Manifold)}\RieszEnergyExtForce\;\diff s}\cdot \RieszEnergyExtForce\Big)
\cdot
\frac{\int_{\Manifold}g^{\Manifold}\;\diff s}{\int_{\ShapeEmbedding(\Manifold)}\RieszEnergyExtForce\;\diff s}\cdot\nabla_\Shape \RieszEnergyExtForce^T V \;\diff s \\
&-
\int_{\HoldAll^\text{out}} \frac{1}{2}\cdot\Bigg(\big(g^\HoldAll\circ\VolumeEmbedding^{-1}\cdot \det D\VolumeEmbedding^{-1}\big)^2 - \Big(\frac{\int_{\HoldAll^\text{out}}g^\HoldAll\circ\VolumeEmbedding^{-1}\cdot \det D \VolumeEmbedding^{-1}\;\diff x}{\int_{\HoldAll^\text{out}}\RieszEnergyExtForce^\HoldAll\;\diff x} 
\cdot 
\RieszEnergyExtForce^\HoldAll\Big)^2\Bigg)
\cdot\operatorname{div}(\operatorname{Pr}_{H^1_{\partial\Manifold\cup\ShapeEmbedding(\Manifold)}}(V)) 
\\
&\qquad\qquad\quad+
\Bigg(g^\HoldAll\circ\VolumeEmbedding^{-1}\cdot \det D \VolumeEmbedding^{-1} - \frac{\int_{\HoldAll^\text{out}}g^\HoldAll\circ\VolumeEmbedding^{-1}\cdot \det D \VolumeEmbedding^{-1}\;\diff x}{\int_{\HoldAll^\text{out}}\RieszEnergyExtForce^\HoldAll\;\diff x} 
\cdot 
\RieszEnergyExtForce^\HoldAll\Bigg)
\cdot
\frac{\int_{\HoldAll^\text{out}}g^\HoldAll\circ\VolumeEmbedding^{-1}\cdot \det D \VolumeEmbedding^{-1}\;\diff x}{\int_{\HoldAll^\text{out}}\RieszEnergyExtForce^\HoldAll\;\diff x} 
\cdot 
\nabla (\RieszEnergyExtForce^\HoldAll)^T \operatorname{Pr}_{H^1_{\partial\Manifold\cup\ShapeEmbedding(\Manifold)}}(V) \;\diff x \\
&-
\int_{\HoldAll^\text{in}} \frac{1}{2}\cdot\Bigg(\big(g^\HoldAll\circ\VolumeEmbedding^{-1}\cdot \det D\VolumeEmbedding^{-1}\big)^2 - \Big(\frac{\int_{\HoldAll^\text{in}}g^\HoldAll\circ\VolumeEmbedding^{-1}\cdot \det D \VolumeEmbedding^{-1}\;\diff x}{\int_{\HoldAll^\text{in}}\RieszEnergyExtForce^\HoldAll\;\diff x} 
\cdot 
\RieszEnergyExtForce^\HoldAll\Big)^2\Bigg)
\cdot\operatorname{div}(\operatorname{Pr}_{H^1_{\partial\Manifold\cup\ShapeEmbedding(\Manifold)}}(V)) 
\\
&\qquad\qquad\quad+
\Bigg(g^\HoldAll\circ\VolumeEmbedding^{-1}\cdot \det D \VolumeEmbedding^{-1} - \frac{\int_{\HoldAll^\text{in}}g^\HoldAll\circ\VolumeEmbedding^{-1}\cdot \det D \VolumeEmbedding^{-1}\;\diff x}{\int_{\HoldAll^\text{in}}\RieszEnergyExtForce^\HoldAll\;\diff x} 
\cdot 
\RieszEnergyExtForce^\HoldAll\Bigg)
\cdot
\frac{\int_{\HoldAll^\text{in}}g^\HoldAll\circ\VolumeEmbedding^{-1}\cdot \det D \VolumeEmbedding^{-1}\;\diff x}{\int_{\HoldAll^\text{in}}\RieszEnergyExtForce^\HoldAll\;\diff x} 
\cdot 
\nabla (\RieszEnergyExtForce^\HoldAll)^T \operatorname{Pr}_{H^1_{\partial\Manifold\cup\ShapeEmbedding(\Manifold)}}(V) \;\diff x \\
&\hspace{15cm}\qquad \forall V\in H^1_{\partial\HoldAll}(\HoldAll, \mathbb{R}^{n+1}).
\end{split}	
\end{align}
The last two integrals correspond to the regularizer for volume parameterization tracking.
As in previous cases, the corresponding Dirichlet condition is given by $\operatorname{BC} = 0$.
All previous remarks on assembling the right and side are still valid.
Additionally, it is necessary to store coordinates of the entire initial hold-all domain.
With these, the volume pre-shape $\VolumeEmbedding^{-1}$ can be calculated as the difference of initial to current coordinates the volume mesh.
For volume regularization, calculation of Jacobian determinants $\operatorname{det}D\VolumeEmbedding^{-1}$ does not require local orthonormal frames via Gram-Schmidt algorithms, as no covariant derivatives are used.
It is very important to use a correct normalization for $\RieszEnergyExtForce^\HoldAll$ to ensure existence of solutions.
This is necessary, since in practical applications an optimization step leads to change of the underlying shape, and thus inner and outer components of $\HoldAll$.
Hence it is not enough to simply estimate $g^\HoldAll$ once in the beginning.
Either, $g^\HoldAll$ needs to be estimated by \cref{Eq_Numerics_gEstimation} in every iteration in which the shape of $\ShapeEmbedding(\Manifold)$ changes.
Or $g^\HoldAll$ is replaced by $g^\HoldAll\circ\VolumeEmbedding^{-1}\cdot \det D \VolumeEmbedding^{-1}$, which is motivated by the transformation rule.
We have decided for the latter, which can be seen in the last two terms of \cref{Eq_Numerics_RHS_VolShapeRegu}.
This also needs to be taken into account when calculating $\TargetPrShp^\HoldAll$, e.g. for line search.
As explained in \cref{SubSubSection_VolReguConsistency}, it is necessary to use $\operatorname{Pr}_{H^1_{\partial\Manifold\cup\ShapeEmbedding(\Manifold)}}(V)$ as directions for the volume regularization, if shapes are enforced to stay invariant.
The projection can be realized by setting the degrees of freedom of the finite element representation of $V$ to zero on the shape $\ShapeEmbedding(\Manifold)$.
This leads to vanishing of the first term of $\PrShpDeriv(f_{\ShapeEmbedding(\Manifold)}^\HoldAll)[V]$ (cf. \cref{Eq_Numerics_MaterialDerivVolTarg}), which does not occur in \cref{Eq_Numerics_RHS_VolShapeRegu}.

Lastly, the right hand side for a \emph{volume and parameterization tracking regularized pre-shape gradient with free tangential outer boundary} is given by \cref{Eq_Numerics_RHS_VolShapeRegu} as well.
However, instead of employing Dirichlet zero boundary conditions, we permit the boundary $\partial\HoldAll$ to move tangentially.
For this, we set 
\begin{equation}\label{Eq_Numerics_BCFreeTangOuter}
	\operatorname{BC}_\VolumeEmbedding = \alpha^{\partial\HoldAll} \cdot U_{L^2} \quad \text{ on } \partial\HoldAll,
\end{equation}
for a scaling factor $\alpha^{\partial\HoldAll}>0$.
Here, $U_{L^2}$ is the $L^2$-representation of tangential components of $\PrShpDeriv\TargetPrShp^\HoldAll(\VolumeEmbedding)$, i.e
\begin{equation}
	\int_{\partial\HoldAll}\langle U_{L^2},V\rangle\;\diff s = \PrShpDeriv\TargetPrShp^\HoldAll(\VolumeEmbedding)[V - \langle V, n_{\partial\HoldAll} \rangle \cdot V] \qquad \forall V\in L^2(\partial\HoldAll, \R^{n+1}).
\end{equation}
Notice that in practice, this does not require solution of a PDE on $\partial\HoldAll$, since the tangential values of $\PrShpDeriv\TargetPrShp^\HoldAll(\VolumeEmbedding)$ can be extracted directly from its finite element representation.
We remind the reader that this is more a heuristic approach, which we will refine in further works.

\subsection{Numerical Results and Comparison of Algorithms}\label{SubSubSection_NumericsResults}
In this subsection we explore computational results of employing unregularized and various pre-shape regularized gradient descents for shape optimization problem \cref{Eq_NumericalPoissonTrClassical}.
We propose an algorithm \ref{Algo_MainAlgoPreShape}, which is a modified gradient descent method with a backtracking line search featuring regularized gradients.
We present 7 implementations of pre-shape gradient descent methods.
The first 4 feature the linear elasticity metric \cref{Eq_MetricLinElas} with unregularized, shape regularized, volume and shape regularized, and volume and shape regularized free tangential outer boundary right hand sides.
The other 3 feature the regularized $p$-Laplacian metric \cref{Eq_MetricPLaplacian} with unregularized, shape regularized, volume and shape regularized right hand sides.
For the $p$-Laplacian metric we dismiss the free tangential outer boundary regularization, since solving it requires a modified Newton's method and slightly complicates our approach.
Both the linear elasticity metric \cref{Eq_MetricLinElas} and the regularized $p$-Laplacian metric \cref{Eq_MetricPLaplacian} involve a local weighting function $\mu$ stemming from \cref{Eq_Numerics_LamePoisson} inspired by \cite{SchulzSiebenborn}.
The two approaches for these metrics without any type of pre-shape regularization are denoted as their 'Vanilla' versions.
For implementations we use the open-source finite-element software FEniCS (cf. \cite{LoggMardalEtAl2012a, AlnaesBlechta2015a}).
Construction of meshes is done via the free meshing software Gmsh (cf. \cite{geuzaine2007gmsh}).
We perform our calculations using a single Intel(R) Core(TM) i5-3210M CPU @ 2.50GHz featuring 6 GB of RAM.

Algorithm \ref{Algo_MainAlgoPreShape} is essentially a steepest descent with a backtracking line search.
The regularization procedures for shape and volume mesh quality take place by modifying the right hand sides as described in \cref{SubSubSection_NumericsGradientSystems}.
However we want to pinpoint some important differences of algorithm \ref{Algo_MainAlgoPreShape} compared to a standard gradient descent for shape optimization.
First, notice that the initial mesh coordinates are stored in order to calculate $\ShapeEmbedding_k^{-1}$ and $\VolumeEmbedding_k^{-1}$.
This corresponds to setting initial pre-shapes $\VolumeEmbedding_0 = \operatorname{id}_{\HoldAll_0}$ and $\ShapeEmbedding_0 = \operatorname{id}_{\Manifold}$.
Since the current mesh coordinates are necessarily stored in a standard gradient descent, $\ShapeEmbedding_k^{-1}$ and $\VolumeEmbedding_k^{-1}$ are calculated as mesh coordinate differences.
Calculating these inverse embeddings amounts to a matrix difference operation, and therefore is of negligible computational burden.
Estimating initial vertex distributions $g^\Manifold$ and $g^\HoldAll$ needs to be done only once at the beginning of our routine.
Hence it does not contribute to computational cost in a significant way.
If shape regularization is partaking in the gradient system, it is necessary to compute and store local tangential orthonormal frames of the initial shape $\TangentVector_0$.
Together with calculation of local tangential orthonormal frames $\TangentVector_k$ for the current shape $\ShapeEmbedding_k(\Manifold)$, these are used to assemble the covariant Jacobian determinant for the regularized right-hand side of the gradient systems.
Since this need to be done for each new iterate $\ShapeEmbedding_k$, it indeed increases computational cost.
If required, this can be mitigated by parallel computing, since tangential orthonormal bases can be calculated simultaneously for all points $p\in \ShapeEmbedding_k(\Manifold)$.
Another difference to standard steepest descent methods concerns the condition of convergence in line $6$ of algorithm \ref{Algo_MainAlgoPreShape}.
It features two conditions, namely sufficient decrease in either the absolute or relative norm of the pre-shape gradient, and sufficient decrease of relative values for the original shape objetive $\TargetShp$.
We use this approach, since several objective functionals participate simultaneously in formation of pre-shape gradients $U_k$.
If shape or volume regularization take place, they influence the size of gradients depending on the mesh quality.
In order to compare different (un-)regularized gradient systems, we use this criterion to guarantee the same decrease of the original problem's objective for all strategies.
For the same reason, the line search checks for a sufficient decrease of the combined objective functionals matching the gradient regularizations.
In some sense, this is a weighted descent for multi criterial optimization, where the objectives are $\TargetShp$ and regularizations $\TargetPrShp^\tau$ and $\TargetPrShp^\HoldAll$.

Furthermore, we mention the difference of our two tested metrics $a(.,.)$ acting as left hand sides.
The linear elasticity metric \cref{Eq_MetricLinElas} leads to a linear system, which is solvable by use of standard techniques such the CG-method.
It is reported in \cite{hinze2021pLap}, that the $p$-Laplacian metric has particular advantages in resolution of sharp edges or kinks of optimal shapes.
Illustration of this is not the goal of this paper.
However, the $p$-Laplacian system \cref{Eq_MetricPLaplacian} is increasingly non-linear for larger $p\geq2$.
This significantly increases computational cost and burden of implementation, since Newton's method requires multiple linear system solves.
Also, systems are possibly indefinite if regularization parameter $\varepsilon>0$ is too small.
If chosen too large, we pay for positive definiteness by overregularizing the gradient systems.
In order to achieve convergence of Newton's method for the $p$-Laplacian, we use gradients from the previous shape optimization step as an initial guess.

\begin{remark}[Integrating Shape and Volume Regularization in Existing Solvers]
	Implementing shape and volume regularization with the pre-shape approach does not require a big overhead, if an existing solver for the shape optimization problem of concern is available.
	It solely requires accessibility of gradient systems \cref{Eq_Numerics_GradientSystemAbstract} and mesh morphing to update meshes and shapes.
	With this, adding regularization terms in style of \cref{Eq_Numerics_RHS_ShapeRegu} or \cref{Eq_Numerics_RHS_VolShapeRegu} to existing right hand sides is all that needs to be done.
	This does not affect the user's choice of preferred metrics $a(.,.)$ to represent gradients.
	We highlight this by implementing and comparing our regularizations for the linear elasticity and the non-linear $p$-Laplacian metrics. 
	From this perspective, algorithm \ref{Algo_MainAlgoPreShape} is only an in-depth explanation how right-hand side modifications of gradient systems are assembled.
\end{remark}

\DontPrintSemicolon
\LinesNumbered
\SetAlCapSkip{10pt}
\begin{algorithm2e}\caption{Simultaneous Shape and Volume Regularized Shape Optimization	\label{Algo_MainAlgoPreShape}}
	Set starting domain $\HoldAll_0$ and shape $\ShapeEmbedding_0(\Manifold)=\ProjectionCanonical(\ShapeEmbedding_0)$ and save according vertex coordinates for future computations\;
	Choose pre-shape regularizations by setting $\alpha^\tau, \alpha^{\HoldAll}\geq0$ \;
	Set shape and volume targets $\RieszEnergyExtForce, \RieszEnergyExtForce^\HoldAll: \HoldAll \rightarrow (0,\infty)$\;
	Estimate initial infinitesimal point distributions $g^\Manifold$ for $\ShapeEmbedding_0(\Manifold)$ and  $g^\HoldAll$ for $\HoldAll_0$ according to \cref{Eq_Numerics_gEstimation}\;
	Calculate local orthonormal tangential bases $\TangentVector_0(p)$ for each vertex of $p\in\ShapeEmbedding_0(\Manifold)$ using  Gram-Schmidt orthonormalization, and save them for future iterations\;
	\textbf{While} $\Big(\Vert U_k \Vert > \varepsilon_{\text{abs}}$ \text{ and } $\frac{\Vert U_k \Vert}{\Vert U_0 \Vert} > \varepsilon_{\text{rel}}\Big)$ \text{ or } $\frac{\TargetShp(\ProjectionCanonical(\ShapeEmbedding_k))}{\TargetShp(\ProjectionCanonical(\ShapeEmbedding_0))} > \varepsilon_{\text{rel}}^\TargetShp$ \textbf{do:} \;
	\Indp
	Assemble right-hand-side of pre-shape gradient system \cref{Eq_Numerics_GradientSystemAbstract}: \;
	\Indp
	solve for state solution $y_k$ via \cref{Eq_NumericalPoissonTrClassical}\;
	solve for adjoint solution $p_k$ via \cref{Eq_Numerics_AdjointPoisson} \;
	Calculate local orthonormal tangential bases $\TangentVector^{\ShapeEmbedding_k}$ for each vertex of $\ShapeEmbedding_k(\Manifold)$ with same orientation as $\TangentVector_0$ using  Gram-Schmidt orthonormalization \;
	\textbf{if $\;\;\;\alpha^\tau=0, \alpha^\HoldAll=0$:} 
	Assemble $\operatorname{RHS}(\ShapeEmbedding_k, \VolumeEmbedding_k)$ according to \cref{Eq_Numerics_RHS_Unregularized} \;
	\textbf{elif $\alpha^\tau\neq0, \alpha^\HoldAll=0$:} 
	Assemble $\operatorname{RHS}(\ShapeEmbedding_k, \VolumeEmbedding_k)$ according to \cref{Eq_Numerics_RHS_ShapeRegu} \;
	\textbf{elif $\alpha^\tau\neq0, \alpha^\HoldAll\neq0$:}  	Assemble $\operatorname{RHS}(\ShapeEmbedding_k, \VolumeEmbedding_k)$ according to \cref{Eq_Numerics_RHS_VolShapeRegu} \;
	\Indm
	Solve for pre-shape gradient $U_k$:\;
	\Indp
	Calculate local weighting parameters $\mu$ by solving  \cref{Eq_Numerics_LamePoisson} \;
	\textbf{if } linear elasticity: \;
	\Indp 
	Assemble left-hand-side $a(.,.)$ by \cref{Eq_MetricLinElas} and solve by preconditioned CG-method\;
	\Indm
	\textbf{elif } $p$-Laplacian: \;
	\Indp 
	Use preconditioned Newton's method to solve \cref{Eq_Numerics_GradientSystemAbstract} with left-hand-side $a(.,.)$ by \cref{Eq_MetricPLaplacian}\;	
	\Indm
	\Indm
	Perform a linesearch to get a sufficient descent direction $\tilde{U}_k$:\;
	\Indp
		$\tilde{U}_k \gets \frac{1}{\Vert U_k \Vert}\cdot U_k$ \\
		\textbf{while }$\TargetShp(\ProjectionCanonical(\ShapeEmbedding_k + \tilde{U}_k\circ\ShapeEmbedding_k)) + \alpha^\tau\cdot\TargetPrShp^\tau(\ShapeEmbedding_k + \tilde{U}_k\circ\ShapeEmbedding_k) + \alpha^\HoldAll\cdot\TargetPrShp^\HoldAll(\VolumeEmbedding_k + \tilde{U}_k\circ\VolumeEmbedding_k)$ \\
		$\qquad\qquad\geq  
			\TargetShp(\ProjectionCanonical(\ShapeEmbedding_k)) + \alpha^\tau\cdot\TargetPrShp^\tau(\ShapeEmbedding_k) + \alpha^\tau\cdot\TargetPrShp^\HoldAll(\VolumeEmbedding_k)$ \textbf{do:}\;
	\Indp
			$\tilde{U}_k \gets 0.5\cdot\tilde{U}_k$\;
	\Indm		
	\Indm
	Perform updates: \;
	\Indp
	$\ShapeEmbedding_{k+1} \gets \ShapeEmbedding_k + \tilde{U}_k\circ\ShapeEmbedding_k$ \;
	$\VolumeEmbedding_{k+1} \gets \VolumeEmbedding_k + \tilde{U}_k\circ\VolumeEmbedding_k$ \;
	\Indm
	\Indm
\end{algorithm2e} 

For a meaningful comparison of the $7$ mentioned approaches, we use the same parameters for the problem throughout.
Parameters for source term $r_{\ShapeEmbedding(\Manifold)}$ of the PDE constraint in \cref{Eq_Numerics_AdjointPoisson} are chosen as $r_1 = -1000$ and $r_2 = 1000$.
The scaling factor for perimeter regularization is $\nu = 0.00001$.
Parameters for calculating local weightings $\mu$ via \cref{Eq_Numerics_LamePoisson} are $\mu_{\text{max}} = 1$ and $\mu_\text{min} = 0.05$ for all approaches.
The stopping criteria for all routines tested remain the same.
Specifically, the tolerance for relative decrease of gradient norms is $\varepsilon_{\text{rel}} = 0.001$, absolute decrease of gradient norms $\varepsilon_{\text{abs}} = 0.00001$ and relative main objective decrease $\varepsilon^\TargetShp_{\text{rel}} = 0.0005$.
If shape regularization is employed, it is weighted with $\alpha^\tau = 1000$ and uses a constant target $\RieszEnergyExtForce \equiv 1$.
This targets a uniform distribution of surface cell volume of shapes.
For volume regularization, the weighting is $\alpha^\HoldAll = 100$ with a constant $\RieszEnergyExtForce^\HoldAll \equiv 1$ targeting uniform volume cells of the hold-all domain.
If we permit free tangential movement of the outer boundary \cref{Eq_Numerics_BCFreeTangOuter}, we choose a weighting parameter $\alpha^{\partial\HoldAll} = 250$.
In case of the $p$-Laplacian metric, we chose a parameter $p=6$.
Its regularization parameter is chosen as $\varepsilon=8$ for the unregularized, and shape and volume regularized routines.
If shape without volume regularization takes place, we had to increase regularization to $\varepsilon=9.5$.
This was necessary, since at some point lower values for $\varepsilon$ resulted in indefinite systems during descent with $p$-Laplacian gradients. 

We compare relative values of $\TargetShp$, $\TargetPrShp^\tau$ and $\TargetPrShp^\HoldAll$, which are illustrated in \cref{Fig_RPlots}.
Here, $\TargetPrShp^\tau$ is interpretable as the deviation of the shape mesh from a surface mesh with equidistant edges.
Similarly, $\TargetPrShp^\HoldAll$ can be understood as the deviation of the volume mesh from a volume mesh with uniform cell volumes.
Since a change of mesh coordinates leads to different qualities of solutions to the PDE constraint of \cref{Eq_NumericalPoissonTrClassical}, the regularizations do affect the original objective $\TargetShp$.
Hence, we also measure distance of shapes $\ShapeEmbedding_k(\Manifold)$ and the target shape by a Hadamard like distance function
\begin{equation}\label{Eq_Numerics_HadamardDistance}
	\operatorname{dist}(\ShapeEmbedding_k(\Manifold), N) := \int_{\ShapeEmbedding_k(\Manifold)} \underset{p\in N}{\operatorname{max}}\;\Vert s - p\Vert \; \diff s.
\end{equation}
This gives us a geometric value for convergence of our algorithms, complementing the value of objective functionals for our results.

\begin{figure}[h]
		\centering
	\begin{minipage}{.5\linewidth}
			\includegraphics[width=0.9\linewidth, height=5cm]{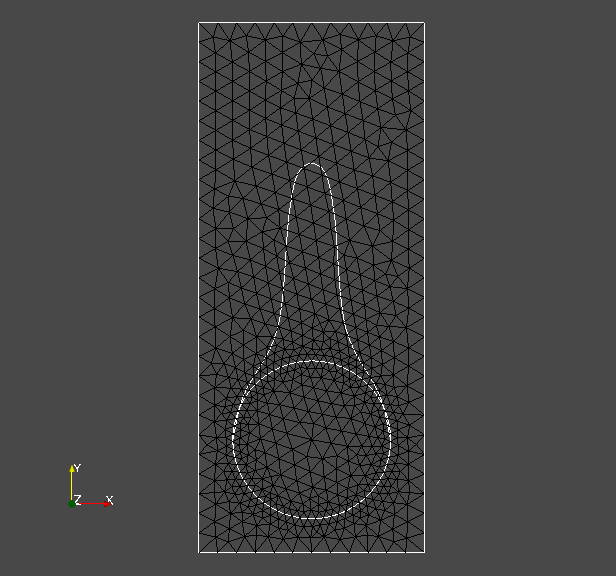}
	\end{minipage}
	\caption{\label{Fig_NumStartTargetBump3}
	Initial mesh $\HoldAll = [0, 1]\times[0, 2.35]$ with embedded initial shape $\Manifold = S^{0.35}_{(0.5, 0.5)}$. The bottle like target shape is included as well.}	
\end{figure}

\begin{table}
\hspace{-2.5cm}
\begin{tabular}{c|ccccccc}
	\toprule
	& LE Vanilla & LE Tang & LE VolTang & LE VolTang Free & $p$-L Vanilla & $p$-L Tang & $p$-L VolTang \\ \midrule
	total time & 49.0s & 345.4s & 199.9s & 320.7s & 135.4s & 316.5s & 325.8s \\[.3cm]
	avg. time step & 1.2s & 2.5s & 2.6s & 2.8s & 2.4s & 2.0s & 4.6s   \\[.3cm]
	number steps & 41 & 137 & 77 & 114 & 55 & 155 & 70   \\[.3cm]
	\bottomrule
\end{tabular}
	\caption{Total times, averaged times per step and number of steps for all 7 methods (for uniform stopping criteria see in the text above).}
\label{Table_NumericsTimes}
\end{table}

From \cref{Fig_RPlots} $(a)$ and $(b)$, we see that all $7$ methods are converging.
They all minimize the original shape objective $\TargetShp$, and the geometric mesh distance to the target shape (cf. \cref{Fig_NumStartTargetBump3}).
Seeing that mesh distance is minimized for all methods confirms that the optimal shape of original problem \cref{Eq_NumericalPoissonTrClassical} is left invariant by our pre-shape regularizations.
Also, convergence to the optimal shape is not affected by the choice of regularization and metric $a(.,.)$.
In \cref{Fig_RPlots} $(a)$ we also see that, given a fixed metric $a(.,.)$, the values of original target $\TargetShp$ for regularized routines vary only slightly from the unregularized one.
This means that intermediate shapes $\ShapeEmbedding_k(\Manifold)$ are left nearly invariant by all regularization approaches as well.
We witness that the $p$-Laplacian metric \cref{Eq_MetricPLaplacian} gives a pre-shape gradient with slightly slower convergence if compared to the linear elasticity \cref{Eq_MetricLinElas} for unregularized and all regularized variants. 
However, one should keep in mind that the shapes considered here have rather smooth boundary.
Notice that all gradients are normed in the line search of algorithm \ref{Algo_MainAlgoPreShape}, which permits this comparison.

In \cref{Table_NumericsTimes} we present the times for all optimization runs.
We see that the fastest method in both time and step count featured the unregularized approach with the linear elasticity metric.
Regularized approaches all need more steps for convergence, since the convergence condition features sufficient minimization of the gradient norms.
As shape and volume tracking objectives $\TargetPrShp^\tau$ and $\TargetPrShp^\HoldAll$ participate in this condition, the optimization routine continues to optimize for mesh quality, despite a sufficient reduction of the original target $\TargetShp$.
This can be verified in \cref{Fig_RPlots}.
Notice that additional volume regularization did not considerably increase average computational time per step for the linear elasticity approach.
The times for approaches featuring shape regularization can be improved by computing tangential orthonormal bases in parallel.
We relied on a rather inefficient but convenient calculation of these solving several projection problems using FEniCS.

From \cref{Table_NumericsTimes}, we see that the unregularized $p$-Laplacian approach for $p=6$ needs more steps to convergence compared to a linear elasticity gradient.
Average time per step is higher too, since Newton's method needs to be applied to solve the non-linear gradient system.
This approach needs careful selection of regularization parameter $\varepsilon>0$ for \cref{Eq_MetricPLaplacian}, since the mesh quality degrades quickly for our problem.
This makes calculation of gradients by Newton's method difficult, since conditioning of systems and indefiniteness at some point of the shape optimization routine are an issue.
If the shape regularized $p$-Laplacian gradient is compared to the unregularized $p$-Laplacian gradient, the computational times were slightly faster on average.
We amount this to faster convergence of the Newton method, since we needed to employ a higher regularization parameter $\varepsilon$ for the shape regularized routine.
This also explains longer computational times for the volume regularized $p$-Laplacian, since the same regularization parameter $\varepsilon$ as in the unregularized approach is permissible, but more Newton iterations are necessary.
Since the shape regularization takes place simultaneously with volume regularization, a lower permissible regularization $\varepsilon$ indeed shows that volume regularization improves condition of linear systems.

\begin{figure}[h]
	\begin{tabular}{cr}
		\begin{subfigure}{0.5\textwidth}
			\includegraphics[width=0.9\linewidth, height=5cm]{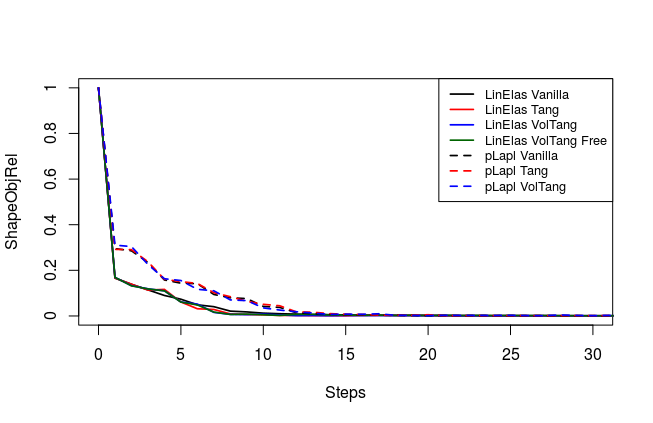}
			\subcaption{relative shape objective value $\frac{\TargetShp(\ProjectionCanonical(\ShapeEmbedding_i))}{\TargetShp(\ProjectionCanonical(\ShapeEmbedding_0))}$}
		\end{subfigure} 
		
		&\begin{subfigure}{0.5\textwidth}
			\includegraphics[width=0.9\linewidth, height=5cm]{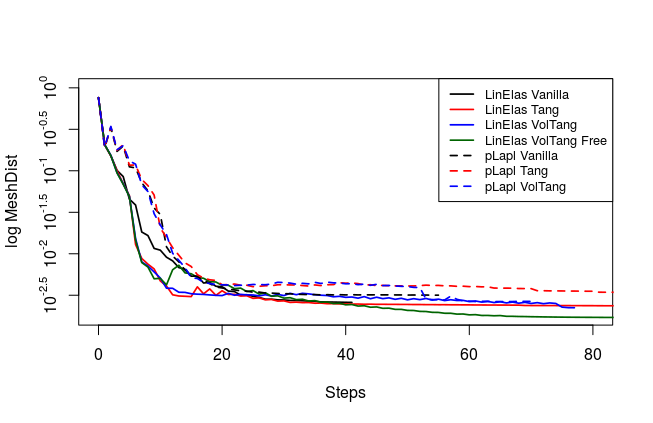}
			\subcaption{$\operatorname{log}$-mesh distance of $\ShapeEmbedding_i(\Manifold)$ and target shape}
		\end{subfigure} \\
		
		\begin{subfigure}{0.5\textwidth}
			\includegraphics[width=0.9\linewidth, height=5cm]{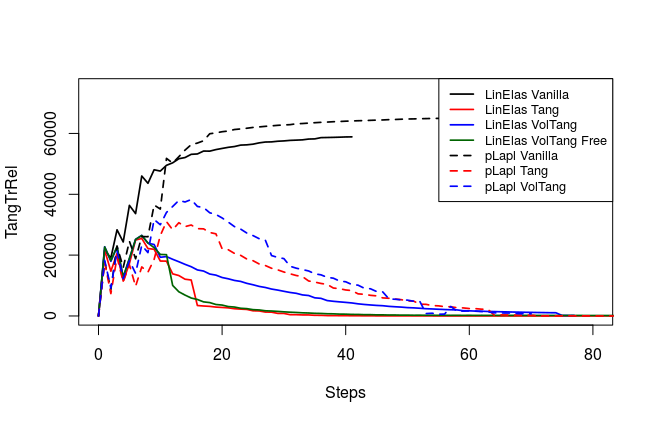}
			\subcaption{relative shape tracking value $\frac{\TargetPrShp^\tau(\ShapeEmbedding_i)}{\TargetPrShp^\tau(\ShapeEmbedding_0)}$}
		\end{subfigure}
		
		&\begin{subfigure}{0.5\textwidth}
			\includegraphics[width=0.9\linewidth, height=5cm]{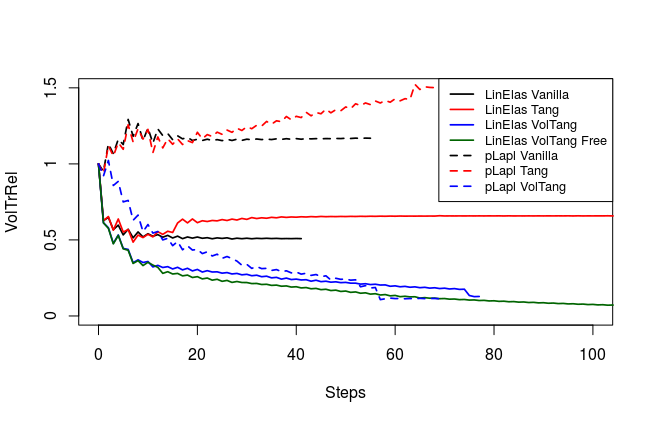}
			\subcaption{relative volume tracking value $\frac{\TargetPrShp^\HoldAll(\VolumeEmbedding_i)}{\TargetPrShp^\HoldAll(\VolumeEmbedding_0)}$}
		\end{subfigure}
	\end{tabular}
	\caption{\label{Fig_RPlots} Relative values for three objective functionals $\Target$, $\TargetPrShp^\tau$ and $\TargetPrShp^\HoldAll$ and logged mesh distance to target shape for gradient descents using 7 different (un-)regularized pre-shape gradients and metrics.}	
\end{figure}

To analyze quality of the shape mesh for all routines, we provide the relative value of the shape parameterization tracking target $\TargetPrShp^\tau$ in \cref{Table_NumericsTimes} $(c)$, as well as $g^\Manifold\circ\ShapeEmbedding^{-1}_k\circ\operatorname{det}D^\tau\ShapeEmbedding_k^{-1}$ for final shapes in \cref{Fig_gshape_LinElas} and \cref{Fig_gshape_pLapl}.
The relative value of the shape parameterization tracking target $\TargetPrShp^\tau$ in \cref{Table_NumericsTimes} $(c)$ measures deviation of the current shape mesh $\ShapeEmbedding_k(\Manifold)$ from a uniform surface mesh.
This means larger values indicate more non-uniformity of shape meshes.
The colors in \cref{Fig_gshape_LinElas} and \cref{Fig_gshape_pLapl} highlight variation of node densities on the shape meshes, where a uniform color indicates approximately equidistant surface nodes.
As the starting mesh seen in \cref{Fig_NumStartTargetBump3} is constructed via Gmsh, it features an approximately uniform surface mesh.
However, both for unregularized linear elasticity and $p$-Laplacian approaches, we see in \cref{Fig_RPlots} $(c)$ that $\TargetPrShp^\tau$ increases during optimization.
This means surface mesh quality deteriorates if no regularization takes place.
For final shapes, this is visualized in \cref{Fig_gshape_LinElas} $(a)$ and \cref{Fig_gshape_pLapl} $(a)$.
There we clearly see an expansion of cell volumes for the targeted bump at the top.
All other routines involve a shape quality regularization by $\TargetPrShp^\tau$.
In \cref{Table_NumericsTimes} $(c)$ it is visible that also for these routines, the deviation $\TargetPrShp^\tau$ from uniform surface meshes increases initially.
Once surface mesh quality becomes sufficiently bad, the shape parameterization takes effect and corrects quality until approximate uniformity is achieved.
We can clearly see this by convergence of $\TargetPrShp^\tau$ for all shape regularized methods in \cref{Table_NumericsTimes} $(c)$.
Also, we see an approximately uniform color of $g^\Manifold\circ\ShapeEmbedding^{-1}_k\circ\operatorname{det}D^\tau\ShapeEmbedding_k^{-1}$ for final shapes in \cref{Fig_gshape_LinElas} and \cref{Fig_gshape_pLapl}, which indicates a nearly equidistant surface mesh.
As a caveat, we see in \cref{Fig_gvol_LinElas} $(b)$ and \cref{Fig_gvol_pLapl} $(c)$ that shape without volume regularization decreases quality of the surrounding volume mesh.
This happens, since surface vertices are transported from areas with low volume at the bottom to areas with high volume at the top.
In case no volume regularization takes place, node coordinates from the hold-all domain are not corrected for this change.
Nevertheless, if a remeshing strategy is employed for shape optimization including shape regularization, the improved surface mesh quality leads to a superior remeshed domain. 
Such routines are an interesting subject for further works.

\begin{figure}[h]
	\begin{tabular}{cr}
		\begin{subfigure}{0.5\textwidth}
			\includegraphics[width=0.9\linewidth, height=5cm]{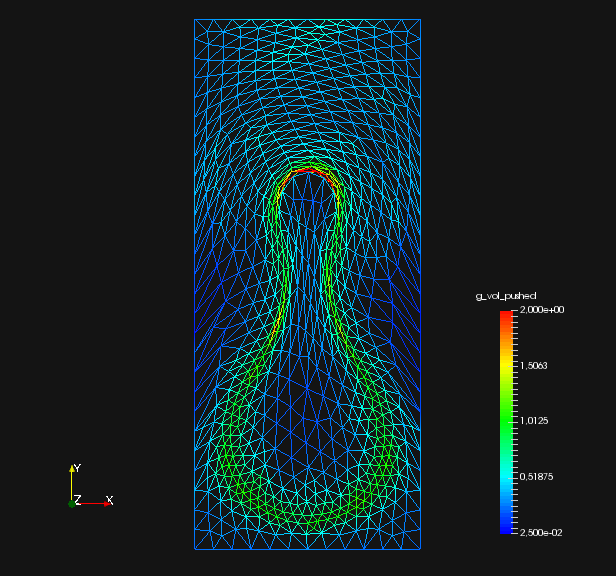}
			\subcaption{Linear elasticity without regularization}
		\end{subfigure} 
		
		&\begin{subfigure}{0.5\textwidth}
			\includegraphics[width=0.9\linewidth, height=5cm]{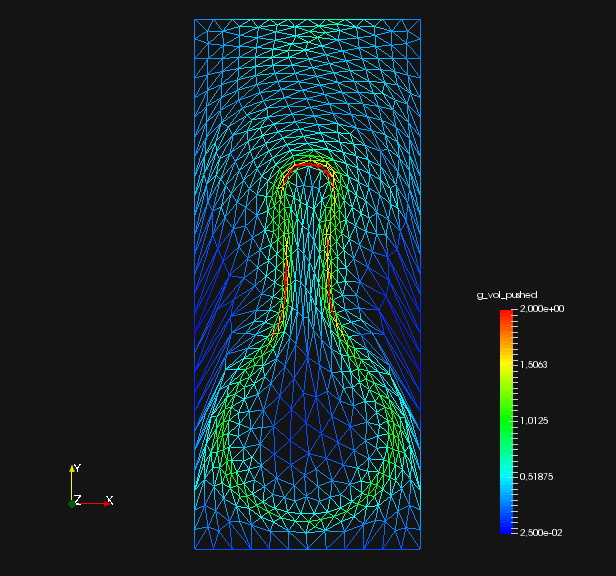}
			\subcaption{Linear elasticity with tangential parameterization tracking}
		\end{subfigure} \\
		
		\begin{subfigure}{0.5\textwidth}
			\includegraphics[width=0.9\linewidth, height=5cm]{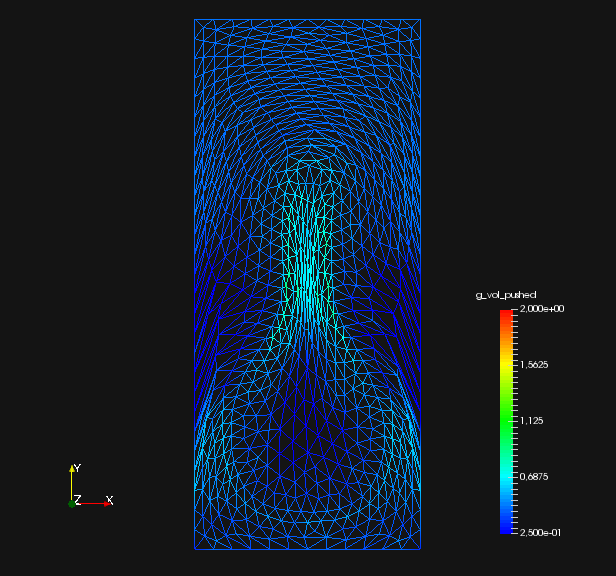}
			\subcaption{Linear elasticity with tangential and volume parameterization tracking}
		\end{subfigure}
		
		&\begin{subfigure}{0.5\textwidth}
			\includegraphics[width=0.9\linewidth, height=5cm]{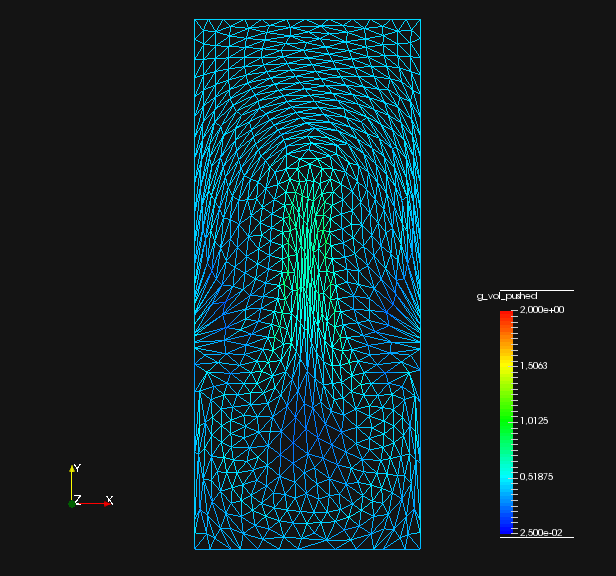}
			\subcaption{Linear elasticity with tangential and free outer boundary volume parameterization tracking}
		\end{subfigure}
	\end{tabular}
	\caption{\label{Fig_gvol_LinElas}Meshes of final steps of respective algorithms. Color depicts the value of $g^{\HoldAll}\circ \ShapeEmbedding^{-1}\cdot \det D\ShapeEmbedding^{-1}$, which is interpretable as the density of allocated volume mesh vertices. A more constant value corresponds to better volume mesh quality.}	
\end{figure}

\begin{figure}[h]
	\begin{tabular}{cr}
		\begin{subfigure}{0.5\textwidth}
			\includegraphics[width=0.9\linewidth, height=5cm]{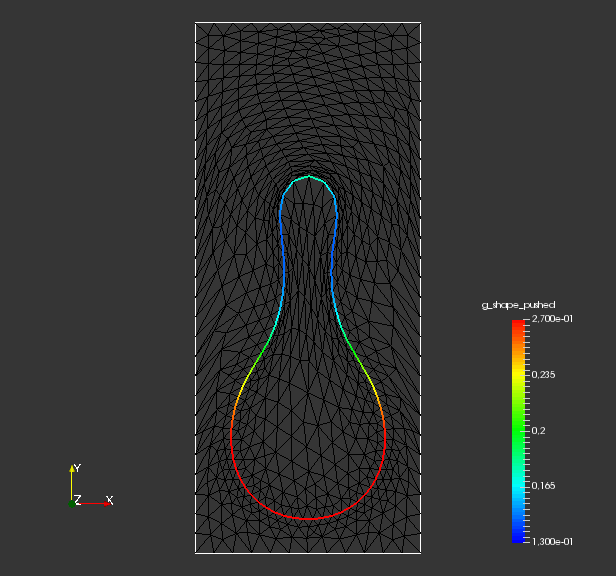}
			\subcaption{Linear elasticity without regularization}
		\end{subfigure} 
		
		&\begin{subfigure}{0.5\textwidth}
			\includegraphics[width=0.9\linewidth, height=5cm]{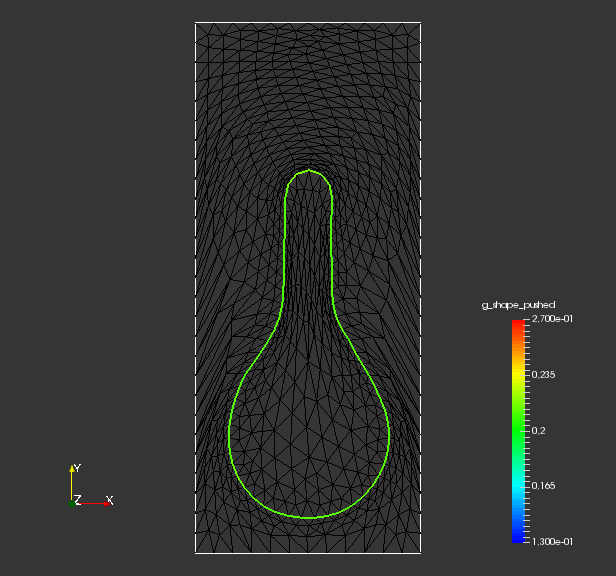}
			\subcaption{Linear elasticity with tangential parameterization tracking}
		\end{subfigure} \\
		
		\begin{subfigure}{0.5\textwidth}
			\includegraphics[width=0.9\linewidth, height=5cm]{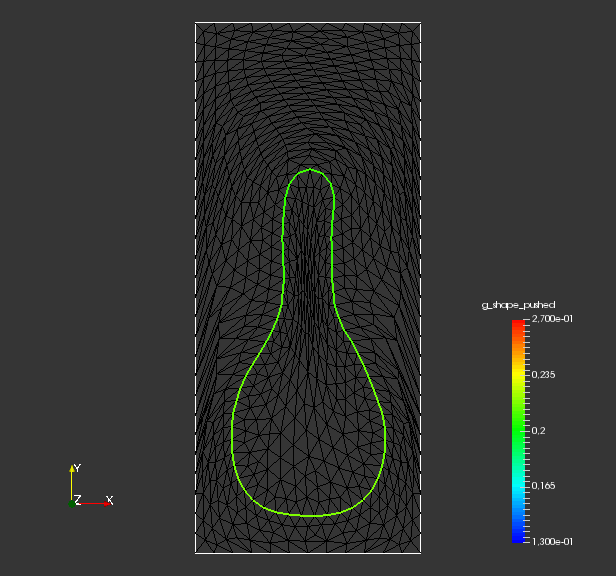}
			\subcaption{Linear elasticity with tangential and volume parameterization tracking}
		\end{subfigure}
		
		&\begin{subfigure}{0.5\textwidth}
			\includegraphics[width=0.9\linewidth, height=5cm]{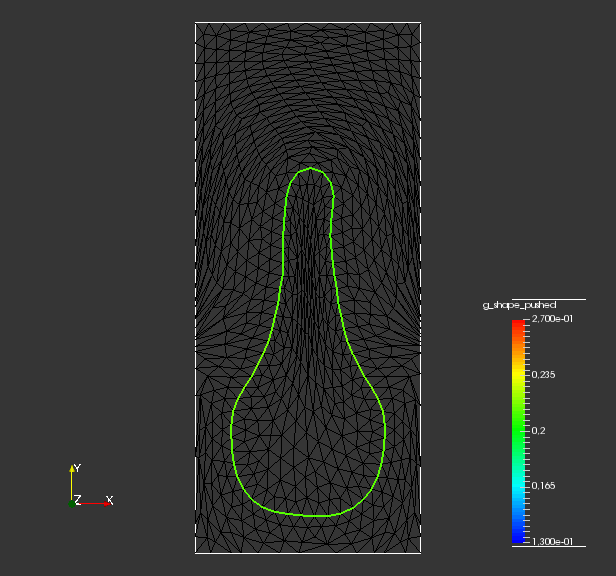}
			\subcaption{Linear elasticity with tangential and free outer boundary volume parameterization tracking}
		\end{subfigure}
	\end{tabular}
	\caption{\label{Fig_gshape_LinElas}Shape meshes of final steps of respective algorithms. Color depicts the value of $g^{\Manifold}\circ \ShapeEmbedding^{-1}\cdot \det D^ \tau\ShapeEmbedding^{-1}$, which is interpretable as the density of allocated shape mesh vertices. A more constant value corresponds to better shape mesh quality.}	
\end{figure}

In \cref{Fig_RPlots} $(d)$ relative values of the volume parameterization tracking functional $\TargetPrShp^\HoldAll$ are depicted for each routine and step.
We interpret these values as a measure for non-uniformity of the volume mesh $\HoldAll$.
The local density of volume vertices $g^\HoldAll\circ\VolumeEmbedding^{-1}\circ\operatorname{det}D\VolumeEmbedding^{-1}$ are visualizing this, and are depicted \cref{Fig_gvol_LinElas} and \cref{Fig_gvol_pLapl}.
For pictures zoomed at the upper tip of final shapes, we refer the reader to \cref{Fig_gvol_LinElas_zoomed} and \cref{Fig_gvol_pLapl_zoomed}.
From \cref{Fig_RPlots} $(d)$ we see that, both for linear elasticity and $p$-Laplacian metrics, non-volume regularized approaches have significantly higher value of $\TargetPrShp^\HoldAll$.
Values even increase for the $p$-Laplacian metric, while there is a slight decrease for the linear elasticity.
Notice that the initial mesh is locally refined near the shape $\Manifold$, which naturally increases the initial value of $\TargetPrShp^\HoldAll$ for a uniform target.
As already discussed, we see in \cref{Fig_RPlots} $(d)$ that shape regularized approaches reduce quality of the volume mesh even further compared to unregularized approaches.
The decrease of mesh quality especially visible in zoomed pictures \cref{Fig_gvol_LinElas_zoomed} and \cref{Fig_gvol_pLapl_zoomed} $(a)$ and $(b)$.
We see that for these approaches, volume cells near the shape are compressed to such an extent that their volumes nearly vanish.
Also, the cell volume distribution for unregularized and shape regularized approaches varies dramatically, which can be seen in \cref{Fig_gvol_LinElas} and \cref{Fig_gvol_pLapl} $(a)$ and $(b)$.
If volume regularization $\TargetPrShp^\HoldAll$ is applied, we see in \cref{Fig_RPlots} $(d)$ that convergence for $\TargetPrShp^\HoldAll$ takes place independent of the metric $a(.,.)$ being used.
This is apparent when looking at the volume mesh quality in \cref{Fig_gvol_LinElas} $(c)$ and $(d)$ and \cref{Fig_gvol_pLapl} $(c)$.
Further notice that in \cref{Fig_gvol_LinElas_zoomed} $(c)$ and $(d)$ and \cref{Fig_gvol_pLapl_zoomed} $(c)$ severe compression of cells neighboring the top of final shapes is avoided.
Volume cells inside the neck of final shapes are still more or less compressed for all approaches.
The interior cell volume cannot be transported through the shape, since it is forced to stay invariant.
Since the mesh topology is not changed during the optimization routine, there is also limited possibility to redistribute the cell volumes inside the shape. 
This situation could be remedied by cell fusion, edge swapping or remeshing strategies, which is beyond the scope of this article.
Finally, we want to highlight the difference of volume regularizations with and without free tangential outer boundary $\partial\HoldAll$.
If \cref{Fig_gvol_LinElas} $(c)$ and $(d)$ are compared, we see that the nodes on the outer boundary $\partial\HoldAll$ changed position for routine $(d)$.
Indeed, the cell volume distribution is more uniform for free outer boundary routine $(d)$, which is visualized by less variation of color.
This leads to even further increase of volume mesh quality, which can be pinpointed in \cref{Fig_RPlots} $(d)$.

\begin{figure}[h]
	\begin{tabular}{cr}
		\begin{subfigure}{0.5\textwidth}
			\includegraphics[width=0.9\linewidth, height=5cm]{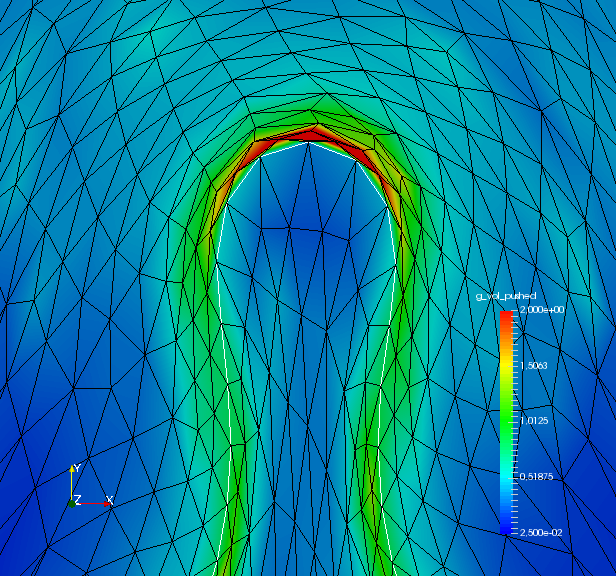}
			\subcaption{Linear elasticity without regularization}
		\end{subfigure} 
		
		&\begin{subfigure}{0.5\textwidth}
			\includegraphics[width=0.9\linewidth, height=5cm]{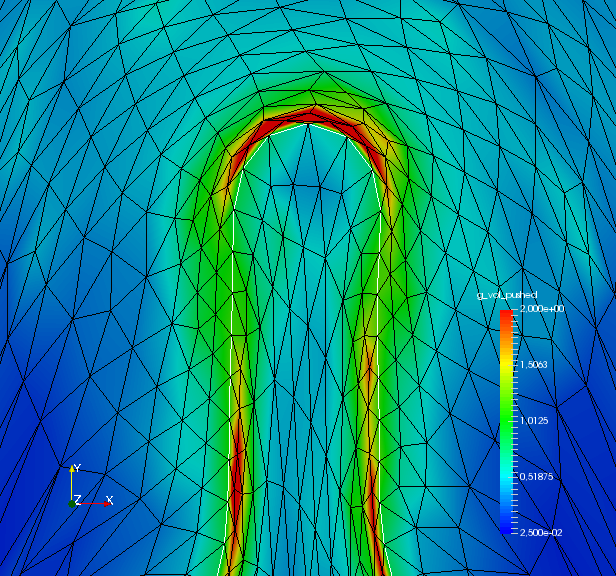}
			\subcaption{Linear elasticity with tangential parameterization tracking}
		\end{subfigure} \\
		
		\begin{subfigure}{0.5\textwidth}
			\includegraphics[width=0.9\linewidth, height=5cm]{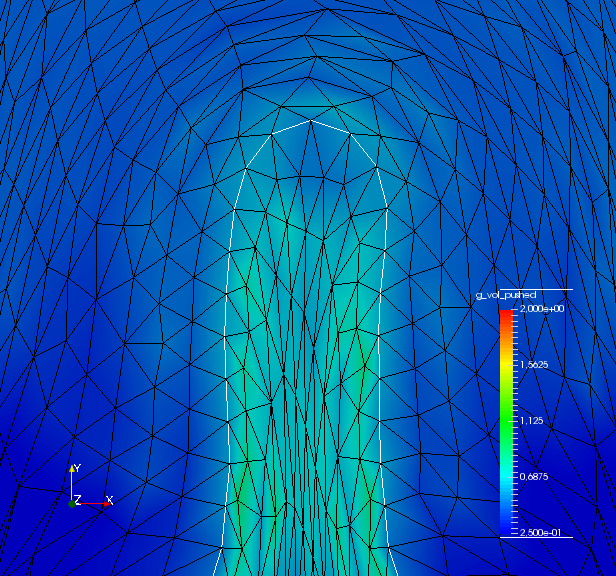}
			\subcaption{Linear elasticity with tangential and volume parameterization tracking}
		\end{subfigure}
		
		&\begin{subfigure}{0.5\textwidth}
			\includegraphics[width=0.9\linewidth, height=5cm]{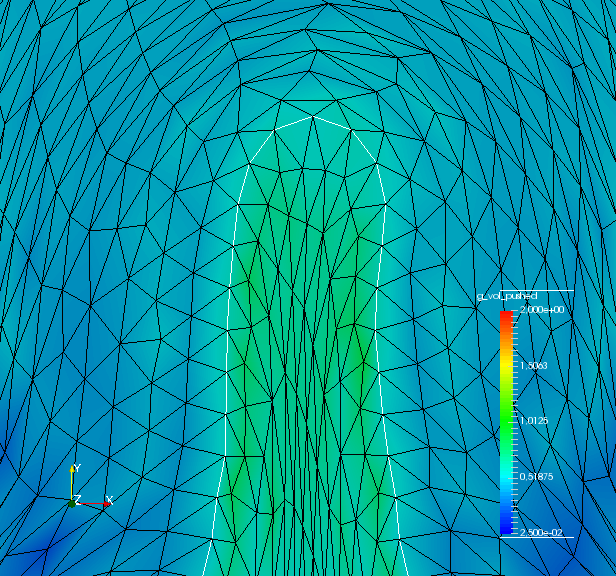}
			\subcaption{Linear elasticity with tangential and free outer boundary volume parameterization tracking}
		\end{subfigure}
	\end{tabular}
	\caption{\label{Fig_gvol_LinElas_zoomed}Meshes of final steps of respective algorithms. Color depicts the value of $g^{\HoldAll}\circ \ShapeEmbedding^{-1}\cdot \det D\ShapeEmbedding^{-1}$, which is interpretable as the density of allocated volume mesh vertices. A more constant value corresponds to better volume mesh quality.}	
\end{figure}

\begin{figure}[h]
	\begin{minipage}{0.5\linewidth}
		\centering
		\begin{subfigure}{1.\textwidth}
			\includegraphics[width=0.9\linewidth, height=5cm]{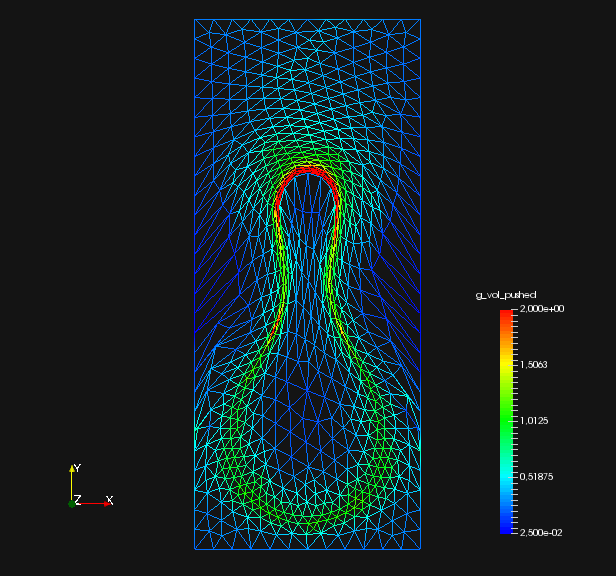}
			\subcaption{$p$-Laplacian without regularization}
		\end{subfigure} 
	\end{minipage}	
	\begin{minipage}{0.5\linewidth}
		\centering
		\begin{subfigure}{1.\textwidth}
			\includegraphics[width=0.9\linewidth, height=5cm]{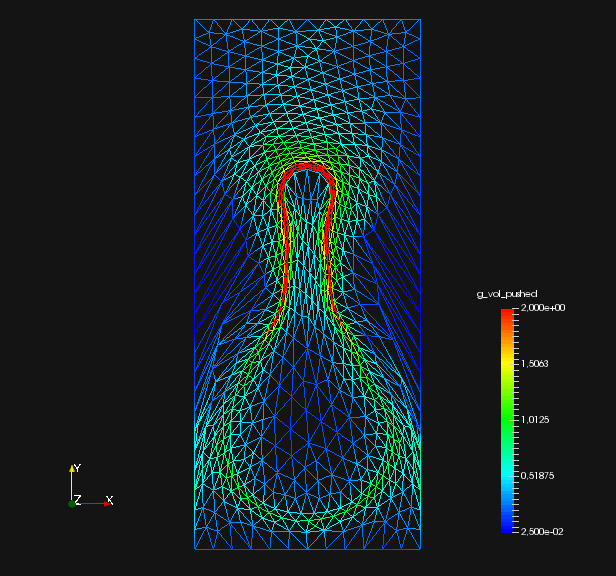}
			\subcaption{$p$-Laplacian with tangential parameterization tracking}
		\end{subfigure} 
	\end{minipage}\\[1ex]
	\begin{minipage}{\linewidth}
		\centering
		\begin{subfigure}{0.5\textwidth}
			\includegraphics[width=0.9\linewidth, height=5cm]{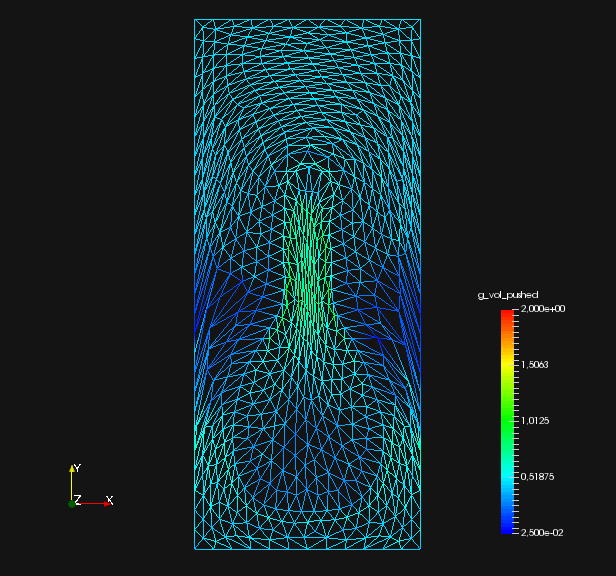}
			\subcaption{$p$-Laplacian with tangential and volume parameterization tracking}
		\end{subfigure}
	\end{minipage}
	\caption{\label{Fig_gvol_pLapl}Meshes of final steps of respective algorithms. Color depicts the value of $g^{\HoldAll}\circ \ShapeEmbedding^{-1}\cdot \det D\ShapeEmbedding^{-1}$, which is interpretable as the density of allocated volume mesh vertices. A more constant value corresponds to better volume mesh quality.}	
\end{figure}

\begin{figure}[h]
	\begin{minipage}{0.5\linewidth}
		\centering
		\begin{subfigure}{1.\textwidth}
			\includegraphics[width=0.9\linewidth, height=5cm]{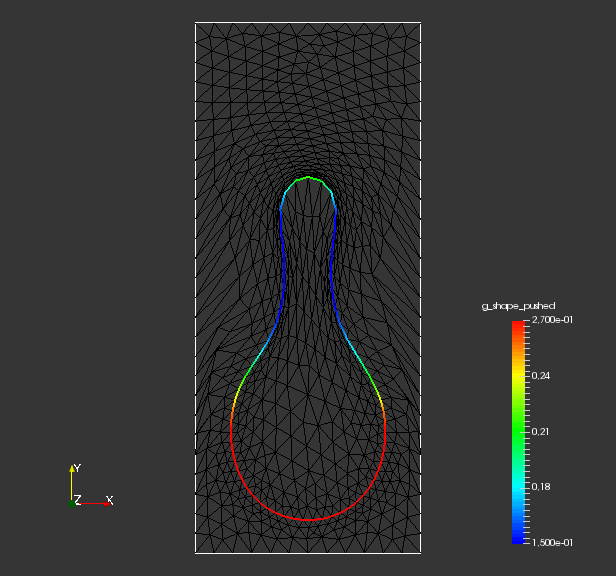}
			\subcaption{$p$-Laplacian without regularization}
		\end{subfigure} 
	\end{minipage}	
	\begin{minipage}{0.5\linewidth}
		\centering
		\begin{subfigure}{1.\textwidth}
			\includegraphics[width=0.9\linewidth, height=5cm]{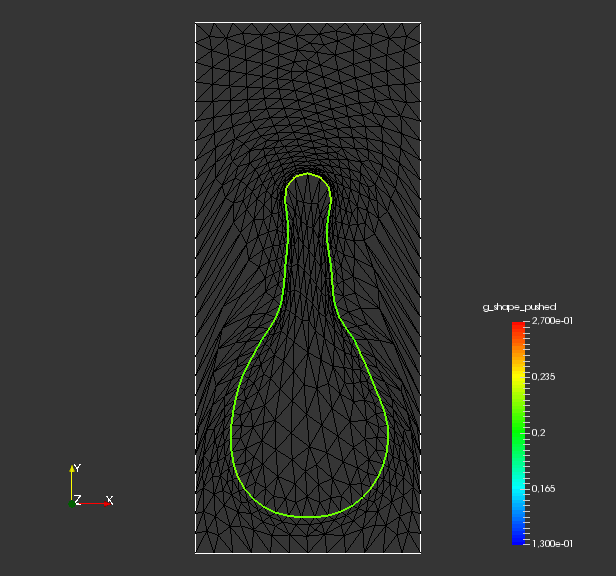}
			\subcaption{$p$-Laplacian with tangential parameterization tracking}
		\end{subfigure} 
	\end{minipage}\\[1ex]
	\begin{minipage}{\linewidth}
		\centering
		\begin{subfigure}{0.5\textwidth}
			\includegraphics[width=0.9\linewidth, height=5cm]{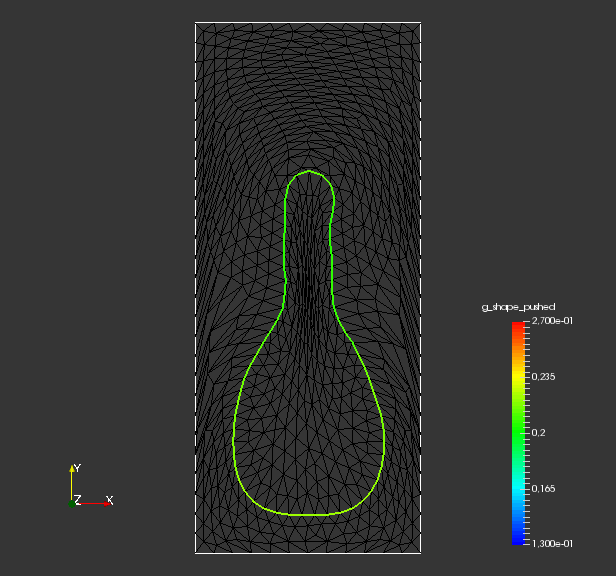}
			\subcaption{$p$-Laplacian with tangential and volume parameterization tracking}
		\end{subfigure}
	\end{minipage}
	\caption{\label{Fig_gshape_pLapl}Shape meshes of final steps of respective algorithms. Color depicts the value of $g^{\Manifold}\circ \ShapeEmbedding^{-1}\cdot \det D^ \tau\ShapeEmbedding^{-1}$, which is interpretable as the density of allocated shape mesh vertices. A more constant value corresponds to better shape mesh quality.}
\end{figure}

\begin{figure}[h]
	\begin{minipage}{0.5\linewidth}
		\centering
		\begin{subfigure}{1.\textwidth}
			\includegraphics[width=0.9\linewidth, height=5cm]{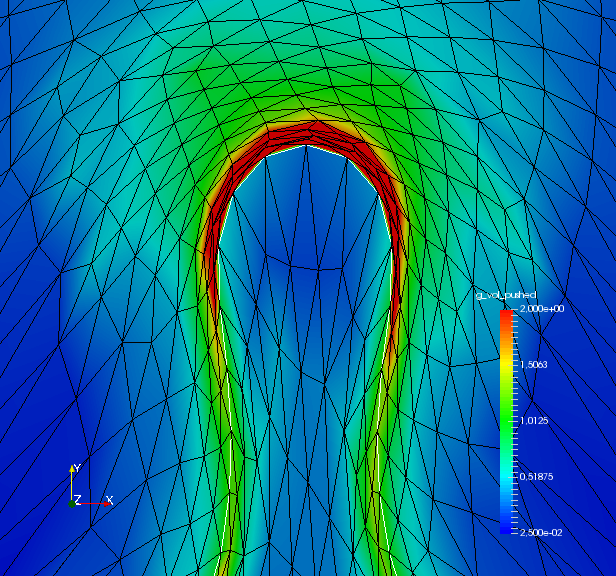}
			\subcaption{$p$-Laplacian without regularization}
		\end{subfigure} 
	\end{minipage}	
	\begin{minipage}{0.5\linewidth}
		\centering
		\begin{subfigure}{1.\textwidth}
			\includegraphics[width=0.9\linewidth, height=5cm]{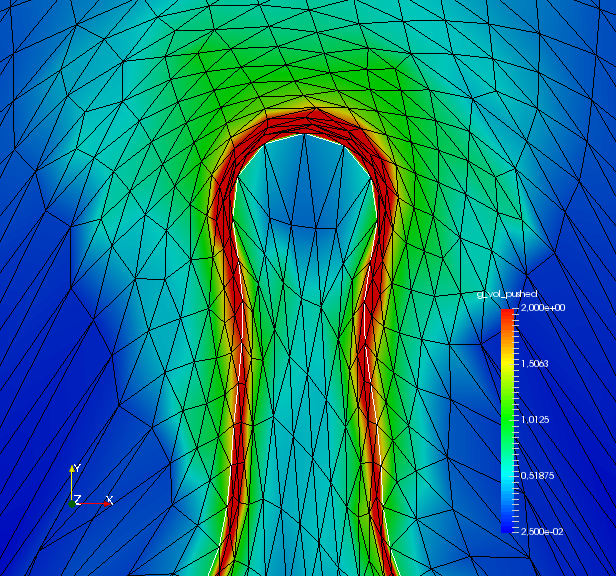}
			\subcaption{$p$-Laplacian with tangential parameterization tracking}
		\end{subfigure} 
	\end{minipage}\\[1ex]
	\begin{minipage}{\linewidth}
		\centering
		\begin{subfigure}{0.5\textwidth}
			\includegraphics[width=0.9\linewidth, height=5cm]{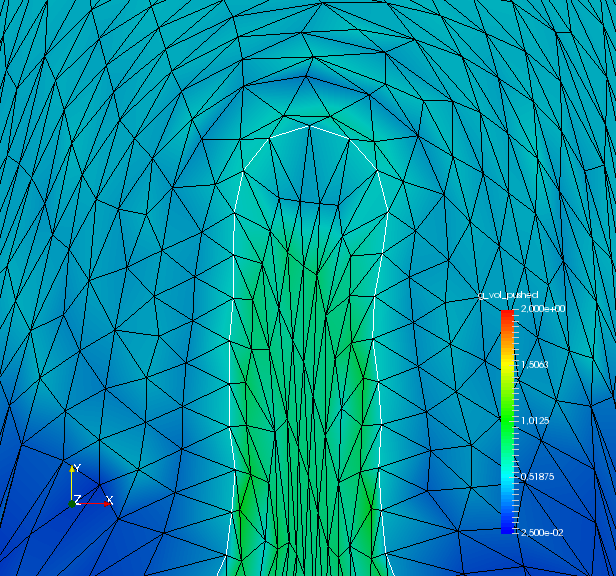}
			\subcaption{$p$-Laplacian with tangential and volume parameterization tracking}
		\end{subfigure}
	\end{minipage}
	\caption{\label{Fig_gvol_pLapl_zoomed}Meshes of final steps of respective algorithms. Color depicts the value of $g^{\HoldAll}\circ \ShapeEmbedding^{-1}\cdot \det D\ShapeEmbedding^{-1}$, which is interpretable as the density of allocated volume mesh vertices. A more constant value corresponds to better volume mesh quality.}	
\end{figure}

\section{Conclusion and Outlook}
In this work, we have provided several approaches to regularize general shape optimization problems to increase shape and volume mesh quality using pre-shape calculus.
Existence of regularized solutions and consistency of modified pre-shape gradient systems is guaranteed by several results for simultaneous shape and volume tracking.
With the presented gradient system modifications, our goal of leaving optimal shapes to the original problem invariant was achieved.
The computational burden is limited, since no additional solution of linear systems for regularized pre-shape gradients is necessary.
We also successfully implemented and compared our pre-shape gradient regularization approaches for linear elasticity and non-linear $p$-Laplacian metrics.

There are several possibilities to further develop pre-shape regularization approaches.
For one, non-constant targets $f$ can be designed to adapt mesh refinement non-uniformly.
In particular, mesh quality targets increasing solution quality of PDE constraints can be envisioned.
Also, we did not touch the topic of pre-shape Hessian, which could be of use to further increase effectiveness of regularization approaches.
Furthermore, a combination with discrete techniques, such as remeshing and edge swapping are possible as well.

\section*{Acknowledgements}
The authors would like to thank Michael Hinze (Koblenz University) and Martin Siebenborn (Universit\"at Hamburg) for a helpful and interesting discussion on the $p$-Laplacian metric.
This work has been supported by the BMBF (Bundesministerium f\"{u}r Bildung und Forschung) within the collaborative project GIVEN (FKZ: 05M18UTA).
Further, the authors acknowledge the support of the DFG research training group 2126 on algorithmic optimization.

\bibliographystyle{plain}
\bibliography{citations_ArXiv.bib} 

\end{document}